\documentclass[reqno,11pt]{amsart}
\usepackage{amsmath,enumerate,amsfonts,amssymb,amsthm,mathrsfs,amsfonts}
\usepackage{graphicx}
\usepackage{bm}
\allowdisplaybreaks[4]
\usepackage{times}
\usepackage{float}
\usepackage[numbers, sort&compress]{natbib}
\usepackage{hyperref}
\usepackage{framed}
\usepackage[normalem]{ulem}
\usepackage[margin=1.25in]{geometry}
\hypersetup{
	colorlinks,
	citecolor=blue,
	filecolor=blue,
	linkcolor=blue,
	urlcolor=black
}
\numberwithin{equation}{section}
\newtheorem{thm}{Theorem}[section]

\newtheorem{lem}{Lemma}[section]
\newtheorem{prop}{Proposition}[section]

\newtheorem{defn}{Definition}[section]

\newtheorem{cor}{Corollary}[section]
\newtheorem{rem}{Remark}[section]

\newcommand{\ud}{\mathrm{d}}
\newcommand{\sB}{\mathscr{B}}
\newcommand{\sT}{\mathscr{T}}

\newcommand{\lam}{\lambda}
\numberwithin{equation}{section}
\newcommand{\N}{\mathbb{N}}
\newcommand{\R}{\mathbb{R}}

\newcommand{\Rn}{\mathbb{R}^n}
\newcommand{\B}{\mathbb{B}}
\newcommand{\cB}{\mathcal{B}}
\newcommand{\cG}{\mathcal{G}}

\newcommand{\cQ}{\mathcal{Q}}
\newcommand{\cI}{\mathcal{I}}
\renewcommand{\S}{\mathbb{S}}
\newcommand{\Sn}{\mathbb{S}^n}
\newcommand{\ga}{\gamma}
\newcommand{\Ga}{\Gamma}
\newcommand{\kap}{\kappa}
\newcommand{\var}{\varepsilon}
\renewcommand{\d}{\mathrm{d}}
\newcommand{\Si}{\Sigma}

\newcommand{\al}{\alpha}
\newcommand{\pa}{\partial}

\newcommand{\be}{\begin{equation}}      \newcommand{\ee}{\end{equation}}
\title[Classification of conformal metrics with constant $Q$ and $T$ curvatures]{Conformal Metrics on the unit Ball with Constant $Q$-Curvature, Constant $T$-Curvature, and Minimal Boundary}
\author{Liming Sun}
\address{State Key Laboratory of Mathematical Sciences, Academy of Mathematics and Systems Science, Chinese Academy of Sciences, Beijing 100190, China}
\email{lmsun@amss.ac.cn}

\author{Heming Wang}
\address{State Key Laboratory of Mathematical Sciences, Academy of Mathematics and Systems Science, Chinese Academy of Sciences, Beijing 100190, China}
\email{hmwang@amss.ac.cn}

\author{Shihong Zhang}
\address{Yau Mathematical Sciences Center, Tsinghua University, Beijing 100084, China}
\email{shihong-zhang@tsinghua.edu.cn}
\date{\today}

\begin{document}

	\begin{abstract}
 We completely classify conformal metrics on the unit ball $(\mathbb{B}^{n+1},|\d x|^2)$, $n\geq4$, with positive constant $Q$-curvature, positive constant $T$-curvature, and minimal boundary. After normalizing the $Q$-curvature, there is a unique conformal metric for each $T$-curvature value in $[0,+\infty)$, up to conformal diffeomorphism. For positive $T$-curvature, these metrics are not Einstein and yield a new family of bubble profiles, distinct from the Aubin--Talenti bubble family except when $T=0$. This new phenomenon has no analogue in either the second-order boundary Yamabe problem or the constant $Q$-curvature problem on closed manifolds. To our knowledge, this is the first classification result for a fourth-order boundary value problem with nonlinear terms both in the interior and on the boundary.

		Using  the conformal M\"obius map from the unit ball  to the
upper half-space, the above geometric problem  admits an equivalent conformally invariant formulation on the upper half-space. More precisely, for $n\geq 4$ and $c_3\in\R_+$, we consider the fourth-order boundary value problem
		\be\tag{$*$}\label{1}
		\left\{\begin{aligned}
			&\Delta^2u=\frac{(n-3)(n-1)(n+1)(n+3)}{16}u^{\frac{n+5}{n-3}} &&\text{ in }\,\R_+^{n+1},\\
			&\sB^3_1(u)=0 &&\text{ on }\,\pa\R^{n+1}_+,\\
			&\sB^3_3(u)=c_3u^{\frac{n+3}{n-3}} &&\text{ on }\,\pa\R^{n+1}_+,
		\end{aligned}\right.
		\ee
		under the finite-volume conditions
	\[
		\int_{\R^{n+1}_+}u^{\frac{2(n+1)}{n-3}}(x,t)\,\d x\,\d t<+\infty
		\quad\text{ and }\quad
		\al^{\frac{2n}{n-3}}:=\frac{1}{|\Sn|}\int_{\Rn}u^{\frac{2n}{n-3}}(x,0)\,\d x<+\infty.
		\]
		We completely classify the positive solutions of \eqref{1}: if $\al\geq1$, there is no positive solution; if $\al\in(0,1)$, then there exists a unique family of positive solutions, depending continuously on $c_3$, which are not Aubin--Talenti bubbles but converge to such bubbles as $c_3\to0^+$.

	\end{abstract}

	\maketitle
	\tableofcontents

	\section{Introduction}
    \subsection{Background and motivation}
	The classical Yamabe problem  asks whether every conformal class  on a closed Riemannian manifold contains a metric of constant scalar curvature; see \cite{LP1987} for a survey. Escobar \cite{E1990,E1996} initiated the corresponding boundary Yamabe problem
on compact manifolds with boundary, seeking conformal metrics with constant
scalar curvature in the interior and constant mean curvature on the boundary.
	One of the  most important model cases is the Euclidean ball, which is conformally
	equivalent to the upper half-space. This naturally leads to Liouville-type classification theorems for conformally invariant equations on the upper half-space.  We first recall the second-order Liouville theory that motivates the present work.



	For $n\geq 2$, 
	write $X=(x,t)\in\R^{n+1}_+=\{(x,t):x=(x_1,\ldots,x_n)\in \Rn, t>0\}$. In the scalar-flat case, prescribing constant boundary mean curvature leads,
after normalization, to the boundary value problem
		\be\label{LZ2}\tag{I}
		\left\{\begin{aligned}
			&\Delta u=0&&\text{ in }\,  \R_{+}^{n+1},\\
			&\pa _tu(x,0)=cu^{\frac{n+1}{n-1}}&&\text{ on }\,   \pa \R^{n+1}_{+}.
		\end{aligned}\right.
		\ee
        Here $c\in \R$ corresponds to the prescribed boundary mean curvature, up to a dimensional constant.
		Li and Zhu \cite{LZ1995}  classified all  nonnegative solutions $u\in C^2(\R^{n+1}_+)\cap C^1(\overline{\R^{n+1}_+})$ of equation \eqref{LZ2}:  if  $c \geq 0$, then $u=a t+b$ with $a, b \geq 0$, $a=c b^{(n+1) /(n-1)}$; if $c<0$, then either $u \equiv 0$ or $u$ takes the form  \[
		u(x,t) =\Big(\frac{2\var}{(t-t_0)^2+|x-x_0|^2}\Big)^{\frac{n-1}{2}}
		\] for some $\varepsilon>0$, $x_0 \in \mathbb{R}^{n}$ and $t_0=2(n-1)^{-1} \varepsilon c$.
		
		The case of positive constant scalar curvature leads, after normalization,
to the critical boundary value problem
		\be\label{LZ1}\tag{II}
		\left\{\begin{aligned}
			&-\Delta u=\frac{(n-1)(n+1)}{4}u^{\frac{n+3}{n-1}}&&\text{ in }\,  \R_{+}^{n+1},\\ &\pa _tu(x,0)=cu^{\frac{n+1}{n-1}}&&\text{ on }\,  \pa \R^{n+1}_{+}.
		\end{aligned}\right.
		\ee
		Escobar \cite{E1990}  classified  all positive solutions $u\in C^2(\R^{n+1}_+)\cap C^1(\overline{\R^{n+1}_+})$ of  equation \eqref{LZ1} satisfying the decay condition $u(X) = O(|X|^{1 - n})$ as $|X|\to \infty$,
		and showed that they must take the form
		\be\label{LZ-bubble}
		u(x,t) =\Big(\frac{2\var}{\var^2+|(x,t)-(x_0,t_0)|^2}\Big)^{\frac{n-1}{2}}
		\ee
		for some $\var>0$, $x_0\in \Rn$, and $t_0=2(n-1)^{-1}\var c$.  Later, Li and Zhu \cite{LZ1995} removed the decay assumption and
proved that every nonnegative solution
		$u\in C^2(\R^{n+1}_+)\cap C^1(\overline{\R^{n+1}_+})$ of \eqref{LZ1} is either identically zero or of the form
\eqref{LZ-bubble}.  See also Chipot-Shafrir-Fila \cite{CSF1996} and  Chipot-Chlebik-Shafrir-Fila \cite{CCFS1998}.

The classification results above describe the rigidity phenomenon for
solutions of the second-order boundary Yamabe problem. The corresponding existence theory
for conformal metrics with prescribed boundary mean curvature has also been
extensively studied; see, for instance,
\cite{LMR2022,R2024,CV2019,LRR2025,ACA2008,DMA2004,XZ2016} and the references
therein.

Motivated by this second-order theory, we investigate its fourth-order
analogue. More precisely, we study Liouville-type classification and rigidity
problems for conformal metrics with constant $Q$-curvature in the interior
and constant $T$-curvature on the boundary.
	To formulate the corresponding
fourth-order boundary equations, we first recall the conformally covariant
boundary operators associated with the Paneitz operator.

	Let $(X^{n+1},g)$ be a smooth Riemannian manifold of dimension $n+1$ with
	boundary $M^n=\partial X^{n+1}$.  In the critical dimensional case $n=3$, Chang and Qing \cite{CQ1997-1} introduced a third-order conformally covariant boundary operator and its associated $T_3$-curvature. For $n\ge4$, the Paneitz operator is accompanied by a full family of conformally covariant boundary operators, developed in the works of Branson-Gover \cite{BG2001}, Grant \cite{G2003}, Stafford \cite{S2006}, Juhl \cite{J2009}, and Gover-Peterson \cite{GP2021}.  We denote these operators by $\sB_i^3$, $i=0,1,2,3$, whose explicit formulas will be recalled in Section \ref{sec:2.1}.  For $i=1,2,3$, the associated boundary $T$-curvatures are denoted by $(\sT_i^3)_g$, see \eqref{Tcurvatures}. 
   In the Euclidean upper half-space
model
		$ (\R^{n+1}_+, \pa  \R^{n+1}_+, \d t^2 + \d x^2)$,  the corresponding boundary
operators reduce to $\sB^3_0(u)=u$,
		\[
		\sB^3_1(u)=-\pa _tu,\quad 
		\sB^3_2(u)=\pa _t^2u-\Delta_{\Rn}u,\quad 
		\sB^3_3(u)=\pa _t\Delta u+2\Delta_{\Rn}\pa _tu
		\]
for all $u\in C^{\infty}(\overline{\mathbb{R}^{n+1}_+})$, and all three background $T$-curvatures vanish.

		Before turning to our nonlinear interior problem, we  recall several
higher-order Liouville-type classification results for $Q$-flat equations
with nonlinear boundary conditions, which may be viewed as fourth-order
analogues of the scalar-flat problem \eqref{LZ2}.

		For problems with a single nonlinear boundary term, the first named author and Xiong \cite{SX2016} investigated polyharmonic equations equipped with highest-order nonlinear conformal boundary data. We recall only the biharmonic case.  Let $n\geq4$, and consider nonnegative solutions
$u\in C^4(\overline{\R^{n+1}_+})$ of
	        \be\label{SX}
		\left\{\begin{aligned}
			&\Delta^2 u=0&&\text{ in }\, \R_{+}^{n+1},\\
			&\pa_t u(x,0)=0&&\text{ on }\,\pa\R^{n+1}_{+},\\
			&\pa_t\Delta u(x,0)=u^{p_3^{*}}&&\text{ on }\,\pa\R^{n+1}_{+},
		\end{aligned}\right.
		\ee
	where $p_3^*=\frac{n+3}{n-3}$.	Assuming additionally that $u(X)=o(|X|^3)$ as $|X|\to+\infty$, they proved that every solution has the form
		\[
			u(x,t)=c(n)\int_{\Rn}\frac{t^3}{(|x-y|^2+t^2)^{\frac{n+3}{2}}}\Big(\frac{\lam}{\lam^2+|y-x_0|^2}\Big)^{\frac{n-3}{2}}\,\d y+a t^2,
		\]
		where $x_0\in\Rn$, $\lam\in\R_+$, $c(n)$ is a normalized constant, and $a\geq0$. Geometrically, after transferring the problem from
the upper half-space to the unit ball by conformal equivalence \eqref{conformalmapS}, this gives
a classification of conformal metrics with vanishing $Q$-curvature,
vanishing $\sT_1^3$-curvature, and constant $\sT_3^3$-curvature, allowing
a single boundary singularity. Recently,  Ndiaye and  the first named author  \cite{NS2024} proved   a classification result
for the biharmonic equation in $\R^4_+$
 with vanishing  $T_3$-curvature  boundary condition.  The existence problem for prescribing
$T$-curvature on four-dimensional manifolds with boundary has also been
widely studied; see, for instance,
\cite{N2009,H2025,NS2024,CQ1997-1,CQ1997-2,HNSW2026} and the references
therein.

	For equations with multiple nonlinear boundary terms, Chen and the last named author
\cite{CZ202406} studied, among other cases,  biharmonic equations with two
nonlinear conformally covariant boundary conditions in dimension $n\geq4$:
		\be\label{PDE:Chen-Zhang}
		\left\{\begin{aligned}
			&\Delta^2 u=0  &&\text{ in }\, \R_+^{n+1},\\
			&\sB^3_i(u)=c_i u^{p^{*}_i}&&\text{ on }\,  \pa \R^{n+1}_{+},\\
			& \sB_j^3(u)=c_ju^{p^{*}_j} &&\text{ on }\,  \pa \R^{n+1}_{+},
		\end{aligned}\right.
		\ee
	 where $c_i,c_j>0$, $i,j\in\{1,2,3\}$ with $i\neq j$, and $p_k^* = \frac{n+2k-3}{n-3}$, $k \in \{1,2,3\}$. Geometrically, \eqref{PDE:Chen-Zhang} corresponds to the $Q$-flat problem
with constant boundary $T$-curvatures.  They derived
explicit biharmonic Poisson kernel representations and established
Liouville-type classification theorems for nonnegative solutions of \eqref{PDE:Chen-Zhang}  under suitable decay assumptions on the boundary trace $u(x,0)$.

	From now on, we assume that $n\geq4$.	The results reviewed above concern the $Q$-flat counterpart of the
scalar-flat problem \eqref{LZ2}. This naturally leads to the higher-order
analogue of the positive scalar-curvature problem \eqref{LZ1}: can one
classify conformal metrics with positive constant $Q$-curvature in the
interior and constant $T$-curvatures on the boundary? Motivated by this
question and by the $Q$-flat classification problem
\eqref{PDE:Chen-Zhang}, we consider
		\be\label{maineq2}
		\left\{\begin{aligned}
			&\Delta^2 u=\kappa u^{p^{*}}  &&\text{ in }\,  \R_+^{n+1},\\
			&\sB^3_i(u)=c_i u^{p^{*}_i} &&\text{ on }\,  \pa \R^{n+1}_{+},\\
			&\sB_3^3(u)=c_3u^{p^{*}_3} &&\text{ on }\,  \pa \R^{n+1}_{+},
		\end{aligned}\right.
		\ee
		where $i\in\{1,2\}$, $p^*=\frac{n+5}{n-3}$, $p_k^*=\frac{n+2k-3}{n-3}$ with $k=1,2,3$, and
\[
\kappa=\frac{n-3}{2}Q_{\S^{n+1}}=\frac{(n-3)(n-1)(n+1)(n+3)}{16}.
\]
Geometrically, if $ g_u=u^{\frac{4}{n-3}}(\d t^2+\d x^2)$, then \eqref{maineq2} prescribes the constant interior curvature $ Q_{g_u}=Q_{\S^{n+1}}$, together with constant boundary $T$-curvatures \[ (\sT_i^3)_{g_u}=\frac{2}{n-3}c_i, \qquad (\sT_3^3)_{g_u}=\frac{2}{n-3}c_3. \]
Compared with \eqref{PDE:Chen-Zhang}, problem \eqref{maineq2} contains a
nonlinear interior term. Moreover, we impose the finite-volume condition
\eqref{Finitevolume} instead of a prescribed pointwise decay rate at infinity.
We also note that the pair $(\sB_1^3,\sB_2^3)$ fails the complementing
condition; see Lemma \ref{lem:complementing-condition}.

		Before stating our results, we highlight three differences between \eqref{maineq2} and the second-order model \eqref{LZ1}. First, the fourth-order existence theory is less developed because local maximum principles are generally unavailable. The bounded-domain barrier argument used by Li and Zhu \cite{LZ1995} therefore has no direct analogue here; for the Paneitz boundary operators, only certain global maximum principles are known in the standard model \cite[Proposition~4.1]{CZ202406}. Second, the existence of positive solutions depends essentially on the boundary-volume parameter $\al$, as Theorem \ref{thm:Main} will show. Third, representation formulas for polyharmonic interior equations, such as \eqref{SX} and \eqref{PDE:Chen-Zhang}, may contain polynomial terms, whereas the nonlinear interior equation in \eqref{maineq2} provides no analogous reason to expect such terms. These distinctions motivate the classification and existence results stated below.

    \subsection{Main results}
Our first result indicates that under appropriate sign conditions on $T$-curvatures,  the classification of nonnegative finite-volume
solutions of \eqref{maineq2} boils down to a radial fourth-order boundary value problem
on the unit ball. 

		\begin{thm}\label{thm:main0}
			Let $n\ge4$, $i\in\{1,2\}$, $c_i\le0$, and $c_3\ge0$.	Assume that  $u\in C^4(\overline{\R^{n+1}_+})$ is a nonnegative solution of  \eqref{maineq2} satisfying the finite-volume conditions
			\be\label{Finitevolume}\tag{F}
			\int_{\R^{n+1}_{+}}u^{\frac{2(n+1)}{n-3}}(x,t)\, \d x\, \d t <+\infty
			\quad \text{ and } \quad
			\al^{\frac{2n}{n-3}}:=\frac{1}{|\Sn|}\int_{\Rn }u^{\frac{2n}{n-3}}(x,0)\, \d x<+\infty.
			\ee

			Then either $u\equiv 0$, or $u>0$ in $\overline{\R^{n+1}_+}$ and there exist $\varepsilon>0$ and $x_0\in\Rn$ such that the
			boundary trace of $u$ is given by a multiple of an Aubin--Talenti bubble
\[
u(x,0)=\al \Big(\frac{2\var}{\var^2+|x-x_0|^2}\Big)^{\frac{n-3}{2}}, \] where $\al $ is the boundary-volume parameter defined in \eqref{Finitevolume}. Furthermore, \[ u(x,t)=\Big(\frac{2\var}{(\var+t)^2+|x-x_0|^2}\Big)^{\frac{n-3}{2}}U\circ \mathcal S\Big(\frac{x-x_0}{\var},\frac{t}{\var}\Big),
\]
	where $\mathcal S$ is the conformal M\"obius transformation defined in
\eqref{conformalmapS}, and $U$ is
			a positive radial solution on $\mathbb B^{n+1}$ of
			\be\label{BVP_U-1}
			\left\{\begin{aligned}
				&\Delta^2 U=\kappa U^{p^{*}}&&\text{ in }\,  \B^{n+1},\\
				&U=\al ,~\sB^3_i(U)=c_i\al ^{p_i^*}&&\text{ on }\,\Sn ,\\
				&\sB^3_3(U)=c_3\al ^{p_3^*}&&\text{ on }\, \Sn .
			\end{aligned}\right.
			\ee
		\end{thm}

		\begin{rem}
		The sign assumptions on the curvature constants are motivated by the sign
structure of the biharmonic Poisson kernels associated with the boundary
operators $\sB_k^3$. More precisely, let $P_k^3$ denote the Poisson kernel
corresponding to the $\sB_k^3$-boundary datum. Then Lemma
\ref{lem:Poissonformula} gives
\[
P^3_1<0\quad\text{ in }\,\R^{n+1}_+,\quad P^3_2<0\quad\text{ in }\,\R^{n+1}_+,\quad \text{ and }\quad   P^3_3>0 \quad\text{ in }\,\R^{n+1}_+.
\]
This sign structure also explains the lack of a direct comparison principle:
even for smooth functions with sufficient decay at infinity,
\[
\Delta^2 u\geq 0 \quad \text{ in }\,\R_+^{n+1},\quad \sB^3_i(u),\, \sB_3^3(u)\geq 0 ~ \text{ on }\, \pa \R^{n+1}_{+} \]does not in general imply $u\ge0$. Accordingly, the assumptions \(c_i\leq0\), \(i\in\{1,2\}\), and
\(c_3\geq0\) in \eqref{maineq2} are compatible with the positivity argument
based on the Green and Poisson representations; see Lemma \ref{lem:Start}.
\end{rem}

Theorem \ref{thm:main0} reduces the classification problem for \eqref{maineq2} to the study of the radial boundary value problem \eqref{BVP_U-1} on the unit ball. The next natural  question is therefore whether such positive radial solutions actually exist. In general, the uniqueness and explicit description of solutions to \eqref{BVP_U-1} appear to be difficult, especially when \(c_1\ne0\).  
		We nevertheless obtain the following partial existence result for small boundary volume.
		\begin{thm}\label{thm:Existence}
			The existence of positive radial solutions to \eqref{BVP_U-1} holds in the following cases.
			\begin{itemize}
				\item[(1)]For $i=1$ and $c_1\leq 0$. If \be
				\label{alphaAssumption}
				0<\al<
				\Big(
				1+\frac{4c_1^2}{(n-3)^2}
				\Big)^{-\frac{n-3}{4}}, 
				\ee
				then there exist a positive smooth radial solution
				$U_{\al }\in C^\infty(\overline{\B^{n+1}})$ and a constant \(c_3\in\R \) such that $U_{\al }$ solves \eqref{BVP_U-1}.
				Moreover, \(c_3\) satisfies
				\be\label{c3bdd}
				\frac{(n^2-1)(n-3)}{4} \al^{-\frac6{n-3}}	L(\al,c_1)	\leq c_3
				\leq 
				\frac{(n^2-1)(n-3)}{4} \al^{-\frac6{n-3}},
				\ee 
				where
				\be\label{L}
				L(\al,c_1) :=
				1-
				\al^{\frac8{n-3}}
				\Big(
				1-\frac{4c_1^2\al^{\frac4{n-3}}}{(n-3)^2}
				\Big)^{-2}
				\Big[
				1-
				\frac{c_1\al^{\frac2{n-3}}}{n-3}
				\Big(
				3-\frac{4c_1^2\al^{\frac4{n-3}}}{(n-3)^2}
				\Big)
				\Big].
				\ee
				\item[(2)]For $i=2$ and $c_2\leq 0$.  There exist \(\al _0>0\) and a
				family of positive smooth radial solutions
				\(\{U_\al \}_{0<\al <\al _0}\) of \eqref{BVP_U-1}. The corresponding constants \(c_3=c_3(\al )\) satisfy
				\[
                c_3(\al )=\frac{(n^2-1)(n-3)}{4}\al ^{1-p^*_3}+O(\al ^{p^*-p^*_3})\quad \text{ as }\, \al \to 0^{+}.
                \] 
                \end{itemize}	
		\end{thm}

       To compare the solutions in Theorem \ref{thm:Existence} with the standard
conformal model, and in particular to clarify the threshold in
\eqref{alphaAssumption}, we introduce the radial Aubin--Talenti family
centered at the origin:
\be\label{bubble}
U_{0,\lam}(r)
=
\Big(
\frac{2\lam}{\lam^2+r^2}
\Big)^{\frac{n-3}{2}},
\quad \lam>0.
\ee
The choice $\lam=1$ gives the standard bubble
\be\label{standardbubble}
U_*(r)
=
\Big(
\frac{2}{1+r^2}
\Big)^{\frac{n-3}{2}}.
\ee
The conformal metric associated with $U_*$ has constant
$Q$-curvature $Q_{\S^{n+1}}$, vanishing $\sT_3^3$-curvature, and
minimal boundary.

		\begin{rem}
		In both cases of Theorem \ref{thm:Existence}, one has \(c_3(\al )>0\) for all sufficiently small $\al>0$. The threshold
\eqref{alphaAssumption} is also sharp with respect to the
Aubin--Talenti family. Indeed, suppose that there exists a bubble
$U_{0,\lam}$ such that  $U_{0,\lam}(1)=\al $  and $\sB^3_1U_{0,\lam}=c_1U^{p^*_1}_{0,\lam}$ on $\Sn$. Then necessarily
			\[
			\al =
			\Big(
			1+\frac{4c_1^2}{(n-3)^2}
			\Big)^{-\frac{n-3}{4}}.
			\]
			Thus \eqref{alphaAssumption} requires the prescribed boundary value
$\al$ to lie strictly below the boundary height of the
Aubin--Talenti bubble with the same mean-curvature parameter $c_1$.
		\end{rem}

		The preceding result gives only a partial existence theory for general
boundary parameters. In the minimal-boundary case \(c_1=0\), however, the
radial problem \eqref{BVP_U-1} enjoys a stronger uniqueness property, which
allows us to obtain a complete classification of finite-volume solutions of
\eqref{maineq2}.

	To state the result, we introduce the \(Q\)-flat comparison profile
associated with the constant boundary value \(1\) and the homogeneous
\(\sB_1^3\)-condition:
		\be\label{Hr}
		H(r)=1+\frac{(n-3)}{4}(1-r^2),\quad 0\leq r\leq 1.
		\ee

We can now state our main classification theorem.
		\begin{thm}\label{thm:Main}
			Let $n \geq 4$, $c_1=0$,  and $c_3 \in \R_{+}$.  Assume that $u\in C^4(\overline{\R^{n+1}_+})$ is a nonnegative  solution of \eqref{maineq2} satisfying the finite-volume conditions  \eqref{Finitevolume}. If $u\not\equiv 0$, then $u>0$ in $\overline{\R^{n+1}_+}$, and the following
classification holds.
			\begin{itemize}
				\item If $\al\geq 1$, then no positive solution exists.

				\item If $\al\in (0,1)$, then there exists a unique continuous strictly decreasing map
				\[\al :(0,+\infty)\to (0,1), \quad c_3\mapsto \al (c_3),\]
				such that
\[
\lim_{c_3\to+\infty}\al (c_3)=0 \quad\text{ and }\quad \lim_{c_3\to 0^{+}}\al (c_3)=1.
\]
				Then there exist $\var>0$ and $x_0\in\Rn$ such that the boundary trace takes the form
\[
u(x,0)=\al \Big(\frac{2\var}{\var^2+|x-x_0|^2}\Big)^{\frac{n-3}{2}}.
\]
				More importantly, the corresponding solution is expressed as
\[
u(x,t)=\Big(\frac{2\var}{(\var+t)^2+|x-x_0|^2}\Big)^{\frac{n-3}{2}} U_{\al (c_3)}\circ \mathcal S\Big(\frac{x-x_0}{\var},\frac{t}{\var}\Big),
\]
					where $\mathcal S$ is the conformal M\"obius transformation defined in \eqref{conformalmapS} and  $U_{\al }=U_{\al (c_3)}$ is the \textbf{unique} positive radial solution of
				\be\label{Intro:ODE}
				\left\{\begin{aligned}
					&\Delta^2 U_{\al }=\kap U_{\al }^{p^{*}}&&\text{ in }\,  \B^{n+1},\\
					& U_{\al }=\al ,~ \sB^3_1(U_{\al })=0&&\text{ on }\,\Sn,\\
					&\sB^3_3(U_{\al })=c_3U^{p_3^*}_{\al }&&\text{ on }\, \Sn .
				\end{aligned}\right.
				\ee
				Furthermore, $ U_{\al }$ is increasing with respect to $\al $ and satisfies the estimates
				\be\label{Intro-U-estimate}
				\al  H\leq U_{\al }\leq \al  U_*,\quad  \forall\, \al \in (0,1); \qquad U_\al =\al  H+O(\al ^{p^*})\quad\text{ as }\, \al \to 0^+;
				\ee
				where $ U_{*}$ and $H$ are given by \eqref{standardbubble} and \eqref{Hr}, respectively.
			\end{itemize}
		\end{thm}

		\begin{rem}\label{rem:1}
			(1)	Theorem \ref{thm:Main} shows that the parameters $\al$ and $c_3$ are in
one-to-one correspondence. In particular,
			$\lim_{c_3\to 0^{+}}U_{\al (c_3)}=U_*$, 
			so the upper bound in \eqref{Intro-U-estimate} becomes sharp in the limit
$\al\to1^-$. (2) At the limiting value $c_3=0$, the solution admits an even reflection
			across the boundary and therefore extends to a positive solution
on the whole space $\R^{n+1}$. Such entire solutions are classified by
Xu and Wei \cite{WX1999} and are necessarily standard bubbles.
		\end{rem}
        Theorem \ref{thm:Main} also reveals a qualitative difference from the
second-order problem \eqref{LZ1}.
\begin{rem}
 For the normalized second-order problem \eqref{LZ1}, the classification of
Li and Zhu \cite{LZ1995} implies that the corresponding boundary-volume
parameter is fixed, namely, \(\alpha=1\). In contrast, for the fourth-order
problem, \(\alpha\) is not fixed by the normalization and provides an
additional parameter for the family of solutions; equivalently, it is in
one-to-one correspondence with \(c_3\). This additional degree of freedom is
related to the fact that the radial fourth-order equation requires more
boundary data than its second-order counterpart; see Remark \ref{rem:Key}.
For \(c_3>0\), or equivalently \(\alpha\in(0,1)\), the corresponding
solutions of \eqref{Intro:ODE} are not Aubin--Talenti bubbles.
\end{rem}

Transferring Theorem \ref{thm:Main} to the unit ball via the conformal
M\"obius transformation
\(\mathcal S:\R^{n+1}_+\to\B^{n+1}\) defined in
\eqref{conformalmapS}, we obtain the following geometric classification.
\begin{cor}\label{cor:ball-main}
			Let \(n\geq4\), and let \(g\) be a smooth metric on
\(\overline{\mathbb B^{n+1}}\) conformal to the Euclidean metric. Suppose that $Q_g = Q_{\mathbb S^{n + 1}}$, $(\sT^3_3)_g = \frac{2}{n - 3}c_3> 0$ and the mean curvature $H_g = 0$. Then there exists a conformal diffeomorphism
\(\Phi:\mathbb B^{n+1}\to\mathbb B^{n+1}\) such that $g = \Phi^*(U_\alpha^{\frac{4}{n - 3}} g_{\mathbb B^{n+1}})$, where $U_\alpha$ is the unique radial solution of \eqref{Intro:ODE}. 
		\end{cor}

\subsection{Geometric remarks and organization}
Under the conformal equivalence \eqref{conformalmapS} between the unit ball
and the upper half-space, any smooth metric on $\overline{\mathbb B^{n+1}}$ conformal
to the Euclidean metric gives rise to a conformal metric satisfying the
finite-volume conditions \eqref{Finitevolume}. At the limiting value
$\alpha=1$, equivalently $c_3=0$, $U_1$ is the Aubin--Talenti bubble and $(\mathbb B^{n + 1}, U_1^{\frac{4}{n - 3}} g_{\mathbb B^{n+1}})$ is isometric to the standard hemisphere. For $\alpha\in (0,1)$, the boundary of $(\mathbb B^{n+1}, U_\alpha^{\frac{4}{n-3}} g_{\mathbb B^{n+1}})$ is totally geodesic.  However, since $c_3>0$, the metric does not admit a
smooth even reflection across the boundary.


	We next compare our results with known rigidity results for the
second-order boundary Yamabe problem and the constant $Q$-curvature
problem on closed manifolds. The classical Obata theorem \cite{O1971} asserts that an Einstein
		metric is rigid under conformal deformations to constant scalar curvature
		metrics. The corresponding boundary rigidity theorem was proved by Escobar \cite{E1990}. There are fourth-order
		analogues of this rigidity for the $Q$-curvature equation. V\'{e}tois \cite{V2024} proved an Obata-type theorem: if $(M^n,g)$ is a closed Einstein manifold with positive scalar curvature and is not conformally diffeomorphic to the round sphere, then every conformal metric with constant $Q$-curvature is a constant multiple of $g$.  
		Case \cite[Theorem 1.1]{C2024} subsequently simplified and extended this argument, showing that, for $a$ in an explicit interval containing $0$, a closed conformally Einstein manifold with nonnegative scalar curvature and constant $Q_g+a\sigma_2$ must be Einstein,   where $\sigma_2$ is the $\sigma_2$-curvature.

Motivated by these closed-manifold rigidity results, it is natural to ask
whether an analogous principle holds for fourth-order conformally
invariant boundary problems.  Starting from the
		Einstein model, one might expect that a conformal metric with constant interior $Q$-curvature together
		with constant boundary $T$-curvature should  be Einstein.  Corollary \ref{cor:ball-main} shows that this expectation is
		false. 
Although the restriction of any such conformal metric on the boundary must be Einstein, the metric is not Einstein in the interior. Indeed, the corresponding interior extension  is governed by the
		fourth-order radial problem \eqref{BVP_U-1}, for which at least two boundary
		data are needed to determine a solution; see the discussion in Section
		\ref{sec:4}. 

		To further illustrate this new family of non-Einstein metrics, we conclude
this subsection with numerical computations of \(U_\alpha\) and an extrinsic
geometric interpretation of the corresponding conformal metrics. More
precisely, we realize $(\mathbb B^{n+1}, U_\alpha ^{4/(n - 3)}g_{\mathbb B^{n +1}})$ as a rotationally symmetric hypersurface in $\mathbb R^{n + 2}$.

We first present numerical computations of the solutions \(U_\alpha\) of
\eqref{Intro:ODE} in the case \(n=4\); see Figure
\ref{fig:geometry-u-and-f}(a). At \(\alpha=1\),
$U_1(r)=(\frac{2}{1+r^2})^{\frac{n-3}{2}}$.
As \(\alpha\to0^+\), it is natural to consider the normalized solutions
\(\alpha^{-1}U_\alpha\). By \eqref{Intro-U-estimate},
\(\alpha^{-1}U_\alpha\) converges to \(H\) defined in \eqref{Hr}; see
Figure \ref{fig:geometry-u-and-f}(b).


We also  provide an extrinsic geometric interpretation of the conformal
metric $(\mathbb B^{n+1}, U_\alpha ^{4/(n - 3)}g_{\mathbb B^{n +1}})$ by realizing it isometrically as a rotationally symmetric hypersurface in $\mathbb R^{n + 2}$. 
	To this end, consider the embedding
\[
F_{\alpha}(r,\omega)=\big(\rho_{\alpha}(r)\omega,h_{\alpha}(r)\big)\in  \R^{n+2}, \quad \omega\in \S^n,
\]
where \(h_\alpha\) denotes the height function. A direct computation gives
\begin{align*}
F_{\alpha}^{*}g_{\R^{n+2}}
&=
|\ud h_{\alpha}(r)|^2+|\ud(\rho_{\alpha}(r)\omega)|^2
=
[(\rho'_{\alpha}(r))^2+(h'_{\alpha}(r))^2]\, \ud r^2
+\rho_{\alpha}^2(r)\,\ud\omega^2 .
\end{align*}
Therefore, in order to have $F_{\alpha}^{*}g_{\R^{n+2}} = U_{\alpha}^{\frac{4}{n-3}}(r)
(\ud r^2+r^2\ud\omega^2)$, we have
\[
\rho_{\alpha}(r)=rU_{\alpha}^{\frac{2}{n-3}}(r),\quad h_{\alpha}(r) = \int_r^1 \sqrt{ U_{\alpha}^{\frac{4}{n-3}}(s) - \big(\rho'_{\alpha}(s)\big)^2 }\,\ud s .
\]
At \(\alpha=1\), the image \(F_1(\mathbb B^{n+1})\) is the standard
hemisphere. As \(\alpha\to0^+\), the hypersurfaces
\(F_\alpha(\mathbb B^{n+1})\) shrink to a point. After the natural
rescaling $\alpha^{-\frac{2}{n - 3}}F_\alpha (\mathbb B^{n + 1})$, they converge to the rotationally symmetric hypersurface determined by
the limiting profile \(H\) defined in \eqref{Hr}; see Figure
\ref{fig:geometry-u-and-f}(c)--(d).
\begin{figure}[ht]
\centering
\begin{minipage}[t]{0.35\textwidth}
\centering
\includegraphics[width=\linewidth]{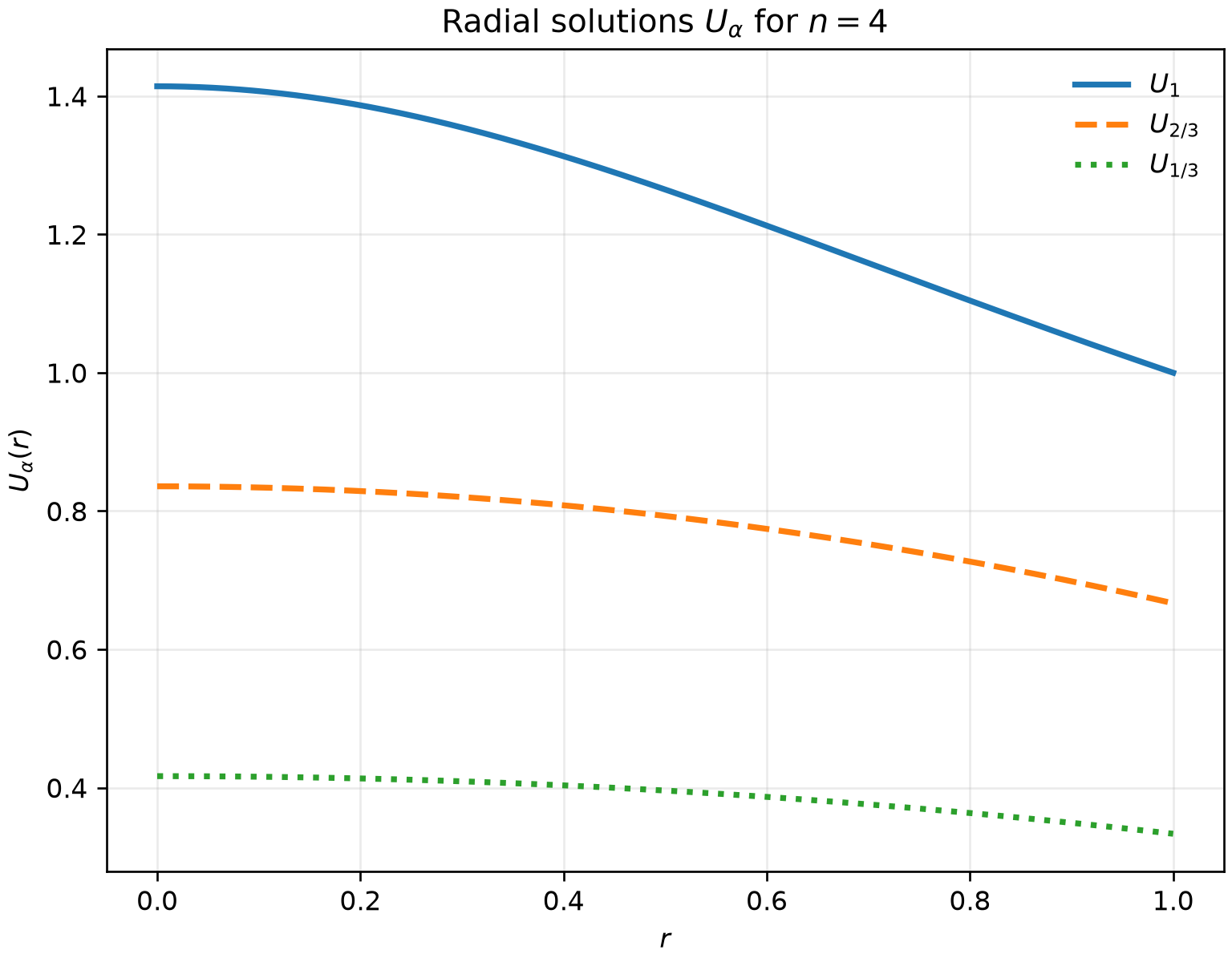}
{\small (a) Numerical radial solutions $U_\alpha$.}
\end{minipage}\hspace{0.05\textwidth}
\begin{minipage}[t]{0.35\textwidth}
\centering
\includegraphics[width=\linewidth]{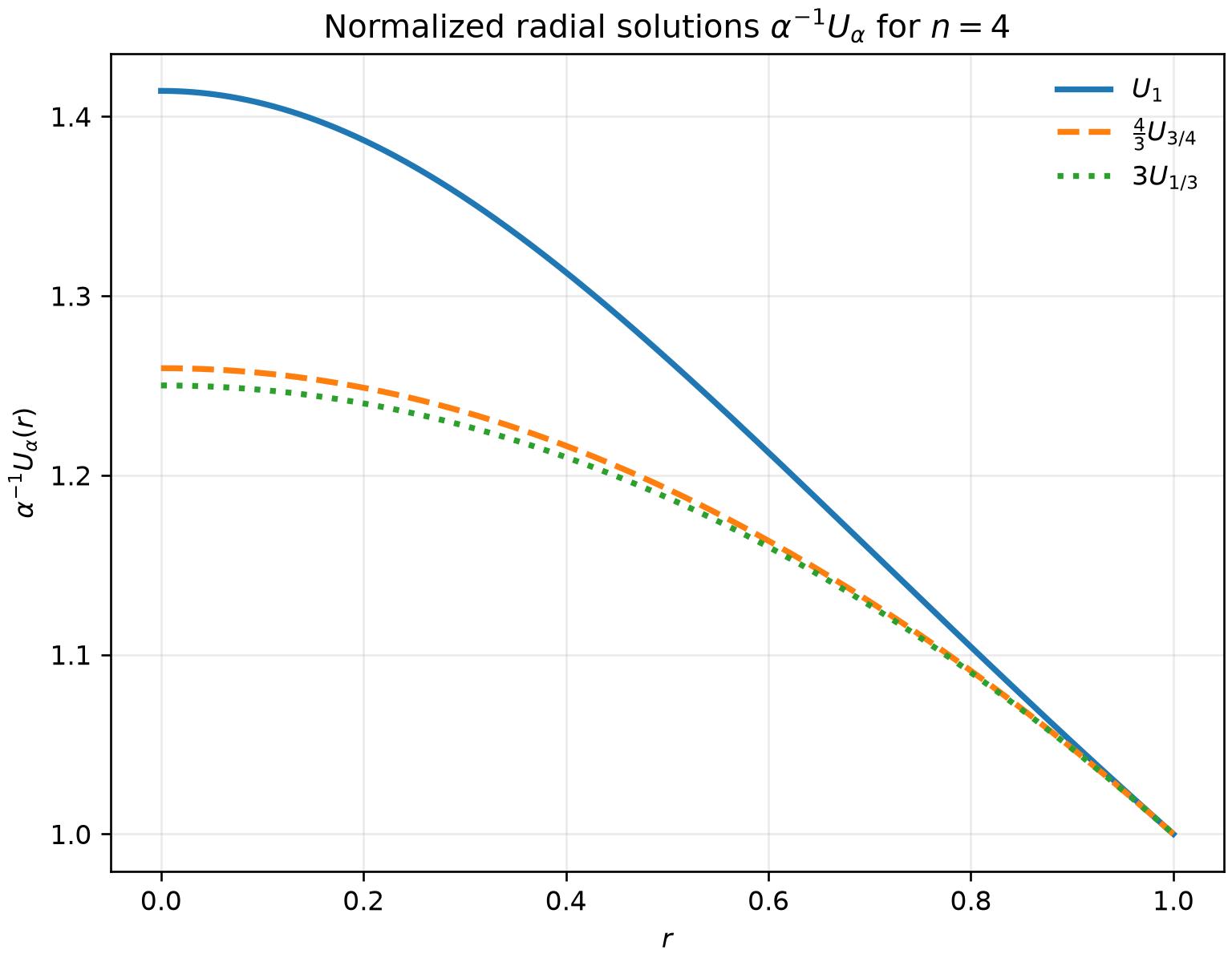}
{\small (b) Normalized radial solutions $\alpha^{-1}U_\alpha$.}
\end{minipage}
\begin{minipage}[t]{0.35\textwidth}
\centering
\includegraphics[width=\linewidth]{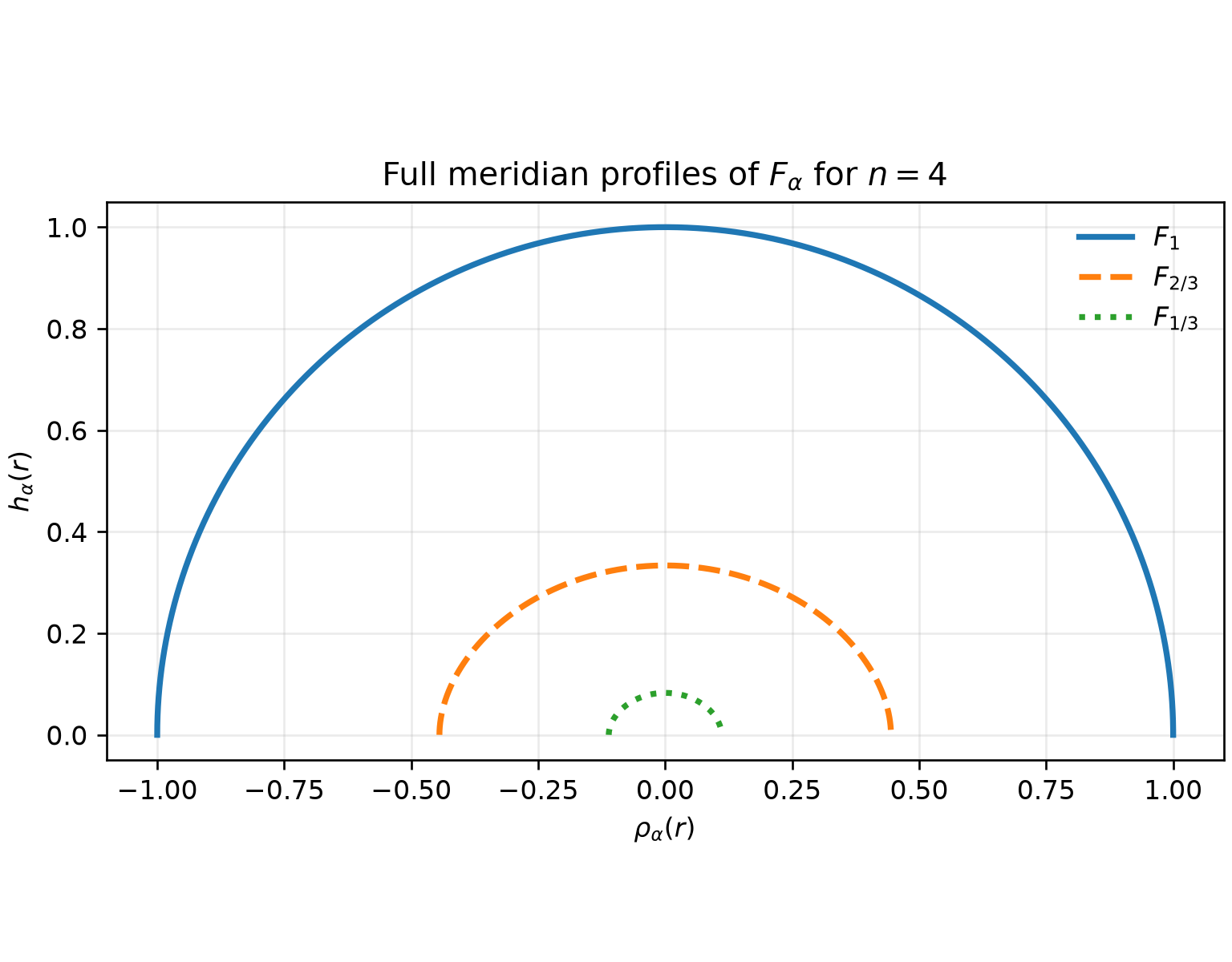}
{\small (c) Generating curves of the embeddings $F_\alpha$.}
\end{minipage}\hspace{0.05\textwidth}
\begin{minipage}[t]{0.35\textwidth}
\centering
\includegraphics[width=\linewidth]{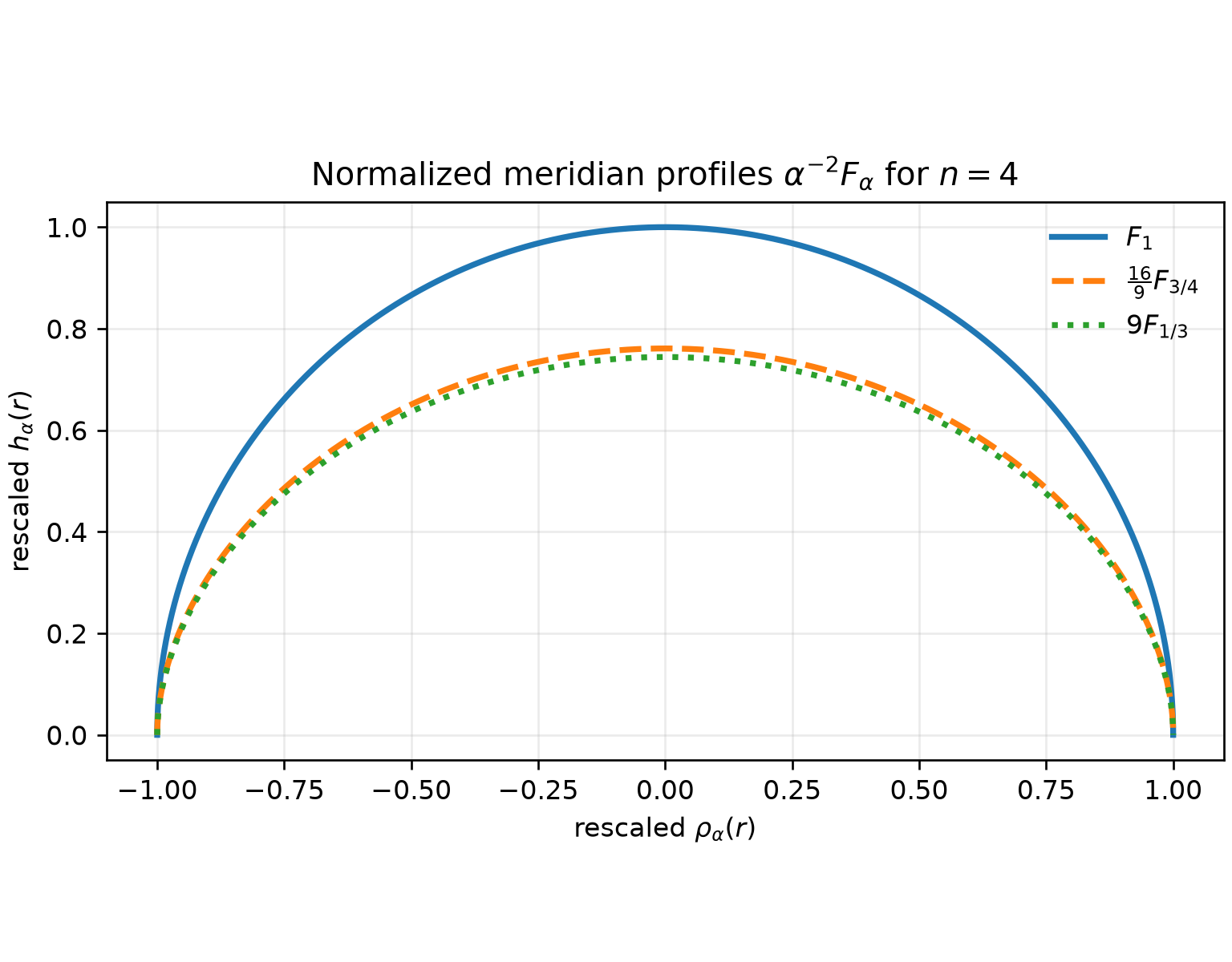}
{\small (d) Normalized generating curves $\alpha^{-2/(n-3)}F_\alpha$.}
\end{minipage}
\caption{Numerical radial solutions and generating curves of the isometric embeddings for $n=4$. Panels~(a) and~(b) show $U_\alpha$ and $\alpha^{-1}U_\alpha$, respectively. Panels~(c) and~(d) show $F_\alpha$ and $\alpha^{-2/(n-3)}F_\alpha=\alpha^{-2}F_\alpha$, respectively. In all panels, the solid, dashed, and dotted curves represent the cases $\alpha=1$, $\alpha=3/4$, and $\alpha=1/3$, respectively.}
\label{fig:geometry-u-and-f}
\end{figure}

		The paper is organized as follows. In Section \ref{sec:2}, we recall the conformal boundary operators, verify the complementing condition, and collect the associated Poisson kernel and Green function representations for biharmonic boundary value problems.  In Section \ref{sec:3}, we derive the integral representation formula and apply the method of moving spheres to prove the classification of the boundary trace, which leads to the proof of Theorem \ref{thm:main0}. In Section \ref{sec:4}, we construct radial solutions to the general boundary value problem  \eqref{BVP_U-1} and prove the existence result in Theorem \ref{thm:Existence} for sufficiently small boundary volume. Section \ref{sec:5} is devoted to the fourth-order ODE analysis on the unit ball and completes the proof of  Theorem \ref{thm:Main}.   
		Technical lemmas are collected in  Appendix \ref{app:A}.

\medskip

	\noindent\textbf{Notation.}~ We collect below the main notation used throughout the paper.
		\begin{itemize}
        \item For $n\geq 4$, we set  $p^*=\frac{n+5}{n-3}$,  	$p_k^* = \frac{n+2k-3}{n-3}$, $k \in \{1,2,3\}$, and $\kap=\frac{(n-3)(n-1)(n+1)(n+3)}{16}$.
        
			\item
			We use capital letters, such as $X=(x,t)$ and $Y=(y,s)$, to denote points in the upper half-space $\R^{n+1}_+$, where $x,y\in\R^n$ and $t,s>0$. We write
			$\d X=\d x\,\d t$, $\d Y=\d y\,\d s$, $\bar X=(x,-t)$ and  $\bar Y=(y,-s)$.
			The variables $\xi$ and $\eta$ are reserved for points in the unit ball $\mathbb{B}^{n+1}$.  We also write $\R_+=(0,+\infty)$.

			\item 	We denote by 
$\mathcal S:\mathbb R^{n+1}_+\to\mathbb B^{n+1}$ the conformal M\"obius transformation: 
  	\be\label{conformalmapS}
\mathcal S(X)=
\frac{2(X+\mathbf e_{n+1})}{|X+\mathbf e_{n+1}|^2}
-\mathbf e_{n+1},
	\ee
where $\mathbf e_{n+1}=(0,0,\cdots,1)\in \R^{n+1}$.

			\item  For $X_0 \in \R^{n+1}$,  we denote $\cB_{r}(X_0):=\{X \in \R^{n+1}:|X-X_0|<r\}$ and $\cB_{r}^{+}(X_0):=\cB_{r}(X_0) \cap \R^{n+1}_+$. If $X_0=(x_0,0)\in\pa \R ^{n+1}_+$, we write $B_r(x_0):=\{x\in\Rn:|x-x_0|<r\}$, so that $\pa '\cB_r^+(X_0)=B_r(x_0)$. More generally, for $X\in\overline{\R ^{n+1}_+}$, set $\pa '\cB_r^+(X):=\{y\in\Rn:|X-(y,0)|<r\}$.

            \item  We use the standard asymptotic notation: $g=O(f)$ means that there exists a constant $C>0$ such that
			$|g|\le C|f|$, while $g=o(f)$ means that $g/f\to0$ (in the indicated limit).
			

			\item  The symbol $C>0$ denotes a generic positive constant which may vary from line to line. We write $C(\al, \beta, \ldots)$ or $C_{\al,\beta,\ldots}$ to indicate that  the constant depends on $\al, \beta, \ldots$.
		\end{itemize}

		\section{Preliminaries}\label{sec:2}

		\subsection{Conformal boundary operators}\label{sec:2.1}
		For completeness, we recall the conformally covariant boundary operators
associated with the Paneitz operator and fix the geometric conventions used
throughout the paper. Let $(X^{n+1},g)$ be a smooth Riemannian manifold with
boundary
		$M^n=\partial X^{n+1}$, and assume throughout this section that $n\geq4$. We refer to
\cite{C2018,CZ202403,CZ202406} and the references therein for further
background.

		Let $\bar g=g|_{TM^n}$. We denote by $R_g$ and $\operatorname{Ric}_g$ the
		scalar curvature and Ricci curvature of $g$, respectively. Let $\nu_g$ be
		the outward unit normal along $M^n$. The second fundamental form is
\[
\pi(X,Y)=\langle\nabla_X\nu_g,Y\rangle, \qquad X,Y\in TM^n,
\]
		and its trace-free part is
\[
\mathring\pi(X,Y)=\pi(X,Y)-h_g\bar g(X,Y),
\]
		where $h_g=\frac{H}{n}$ is the normalized mean curvature of $M^n$.  We use
$\bar\nabla$ and $\bar\Delta$ for the Levi-Civita connection and Laplacian
associated with $\bar g$.

		The Paneitz operator,  introduced by Paneitz \cite{P2008} in 1983,  is the fourth-order conformally covariant operator
\[
P_4^g = \Delta_g^2 -\delta\Big[ \Big( \frac{n^2-2n+5}{2n(n-1)}R_g g -\frac{4}{n-1}\operatorname{Ric}_g \Big)\d \Big] +\frac{n-3}{2}Q_g,
\]
		where $\delta$ is the divergence operator and $\d$ is the exterior
		differential. The zeroth-order curvature term in $P_4^g$ is given by the $Q$-curvature, emphasized by Branson \cite{B1985}, namely
\[
Q_g = -\frac{1}{2n}\Delta_gR_g + \frac{n^3-n^2+11n-3}{8n^2(n-1)^2}R_g^2 -\frac{2}{(n-1)^2}|\operatorname{Ric}_g|^2.
\]
		Under a conformal change of metrics $\tilde g=e^{2\tau}g$, one has
\[
P_4^{\tilde g}(\phi) = e^{-\frac{n+5}{2}\tau} P_4^g(e^{\frac{n-3}{2}\tau}\phi), \quad \forall\, \phi\in C^\infty(\overline{X^{n+1}}).
\]

As mentioned in the introduction, the Paneitz operator is accompanied by a family of conformally covariant
boundary operators $\mathscr B_k^3$, $k=0,1,2,3$, which satisfy
\[
	(\sB_i^3)_{\tilde g}(\phi) = e^{-\frac{n+2i-3}{2}\tau} (\sB_i^3)_g(e^{\frac{n-3}{2}\tau}\phi), \quad \forall\, \phi\in C^\infty(\overline{X^{n+1}}).
    \]
	For \(i=1,2,3\), the associated boundary \(T\)-curvatures are defined by
	\be\label{Tcurvatures}
	(\sT ^3_i)_{g}=\frac{2}{n-3}(\sB^3_i)_g(1).
	\ee
   In particular, \(\sT_1^3\) is the normalized mean curvature, while \(\sT_2^3\) and \(\sT_3^3\) are higher-order boundary curvature quantities associated with the Paneitz operator.

With the above normalization, for
$\phi\in C^\infty(\overline X)$ the boundary operators are given by
		\begin{align*}
			\mathscr B_0^3\phi
			=&\,\phi,\\
			\mathscr B_1^3\phi
			=&\,
			\frac{\partial\phi}{\partial\nu_g}
			+\frac{n-3}{2}h_g\phi,\\
			\mathscr B_2^3\phi
			=&\,
			-\bar\Delta\phi
			+\nabla^2\phi(\nu_g,\nu_g)
			+(n-2)h_g\frac{\partial\phi}{\partial\nu_g}
			+\frac{n-3}{2}\sT_2^3\phi,\\
			\mathscr B_3^3\phi
			=&\,
			-\frac{\partial}{\partial\nu_g}\Delta_g\phi
			-2\bar\Delta\frac{\partial\phi}{\partial\nu_g}
			-\frac{n-3}{2}h_g\nabla^2\phi(\nu_g,\nu_g)
			+\frac{4}{n-1}\langle\mathring\pi,\bar\nabla^2\phi\rangle\\
			&\,
			-\frac{3n-5}{2}h_g\bar\Delta\phi
			-2(n-4)\langle\bar\nabla h_g,\bar\nabla\phi\rangle
			+S_2^3\frac{\partial\phi}{\partial\nu_g}
			+\frac{n-3}{2}\sT_3^3\phi;
		\end{align*}
        see  \cite{C2018}. 
		
        Let
\[
A_{ij} = \frac{1}{n-1} \Big( R_{ij}-\frac{R_g}{2n}g_{ij} \Big)
\]
		be the Schouten tensor of $g$, and set $J=\operatorname{tr}_g A$. We denote by $\bar A$ and $\bar J$ the corresponding Schouten tensor and
its trace for the boundary metric $\bar g$. Then
		\begin{align*}
			S_2^3
			=&\,
			-\frac{3n^2-7n+6}{4}h_g^2
			+\frac{n-7}{2}A(\nu_g,\nu_g)
			+\frac{3n-5}{2}\bar J
			+\frac12|\mathring\pi|^2,\\
			\sT_2^3
			=&\,
			\bar J-A(\nu_g,\nu_g)
			+\frac{n-2}{2}h_g^2,\\
			\sT_3^3
			=&\,
			\frac{\partial J}{\partial\nu_g}
			-2\bar\Delta h_g
			-\frac{4}{n-1}\langle\mathring\pi,\bar A\rangle
			+\frac{n-3}{2}h_gA(\nu_g,\nu_g)\\
			&\,
			+\frac{3n-1}{2}h_g\bar J
			+\frac{n+1}{2(n-1)}h_g|\mathring\pi|^2
			-\frac{n^2-n+2}{4}h_g^3.
		\end{align*}
		Here $S_2^3$ is a lower-order geometric coefficient appearing in
$\mathscr B_3^3$. 

We shall frequently use the specialization of these operators to the
Euclidean unit ball. On \((\B^{n+1},\Sn,|\d\xi|^2\)) with $r=|\xi|$, one has $\sB_0^3(U)=U$ and
		\be\label{Boundaryoperators-Ball}
		\begin{aligned}
			\sB_1^3(U)= &\, \frac{\pa U}{\pa r}+\frac{n-3}{2} U, \\
			\sB_2^3(U)= &\, \frac{\pa^2 U}{\pa r^2}-\Delta_{\Sn} U+(n-2) \frac{\pa U}{\pa r}+\frac{(n-3)(n-1)}{2} U, \\
			\sB_3^3(U)= & -\frac{\pa \Delta U}{\pa r}-2 \Delta_{\Sn} \frac{\pa U}{\pa r}-\frac{n-3}{2} \frac{\pa^2 U}{\pa r^2}-\frac{3 n-5}{2} \Delta_{\Sn} U \\
			&\, +\frac{n-3}{2} \frac{\pa U}{\pa r}+\frac{(n^2-1)(n-3)}{4} U.
		\end{aligned}
		\ee 
		The corresponding $T$-curvatures of the standard Euclidean ball are
		\[
		\sT ^3_1=1, \quad \sT ^3_2=n-1, \quad \text { and } \quad \sT ^3_3=\frac{n^2-1}{2}.
		\]


		\subsection{Complementing condition}


We first recall the complementing condition for fourth-order boundary value
problems and verify it for the conformal boundary operators introduced above.
This condition will allow us to apply standard boundary elliptic estimates
to the pairs $(\sB_i^3,\sB_3^3)$, $i\in\{1,2\}$.

		Following the notation of Gazzola-Grunau-Sweers \cite[p.~34]{GGS2010}, let $B_1(x,D)$ and $B_2(x,D)$ be two linear boundary differential operators, and let $B'_1(x,D)$ and $B'_2(x,D)$ denote their principal parts. Given an interior source term $f$ and boundary data $f_1,f_2$, we consider the fourth-order boundary value problem:
\[
\left\{\begin{aligned} &\Delta^2 u=f &&\text{ in }\, \R_{+}^{n+1},\\ &B_1(x,D) u=f_1&&\text{ on }\,\pa \R^{n+1}_{+},\\ &B_2(x,D) u=f_2 &&\text{ on }\, \pa \R^{n+1}_{+}. \end{aligned}\right.
\]
		Elliptic estimates for such boundary value problems depend crucially on the structure of the boundary operators. The relevant condition is the following complementing condition, also known as the Lopatinski-Shapiro condition; see Agmon-Douglis-Nirenberg \cite{ADN1964} and Gazzola-Grunau-Sweers \cite[p.~34]{GGS2010}.

		\begin{defn}[Complementing condition]
	Let $\Omega$ be a smooth domain.		The pair $\{B_1,B_2\}$ is said to satisfy the complementing condition for
			$\Delta^2$ if, for every $x\in\pa \Omega$ and every nonzero tangential
			covector $\tau\in T_x^*\pa \Omega$, the two polynomials
			$B'_1(x, \tau+s\nu)$ and $B'_2(x, \tau+s\nu)$ in the complex variable $s$ are linearly independent modulo
			$(s-\mathrm i |\tau|)^{2}$, where $\nu$ denotes the unit outward normal at $x$ and $\mathrm i=\sqrt{-1}$.
		\end{defn}
A direct computation of the principal boundary symbols gives the following result.
		\begin{lem}\label{lem:complementing-condition}
			On $(\R^{n+1}_+,\partial\R^{n+1}_+,|\d X|^2)$, for $0\leq i<j\leq 3$, the pair
			$(\sB_i^3,\sB_j^3)$ satisfies the complementing condition for $\Delta^2$ if
			and only if
			$i+j\ne3$.
		\end{lem}
%
%

		We also record here  a local elliptic estimate for homogeneous biharmonic boundary
		value problems. The proof follows from the standard Agmon-Douglis-Nirenberg
		estimates; see \cite[Theorem 9.3 or Theorem 15.2]{ADN1964} and
		\cite[Theorem 2.19 or Theorem 2.20]{GGS2010}.
		\begin{lem}\label{lem:elliptic-estimate}
			Let 	$i\in \{1,2\}$.	Assume that
			$u\in C^4(\overline{\R^{n+1}_+}\setminus\{0\})$ satisfies
			\[	\left\{\begin{aligned}
				&\Delta^2 u=0  &&\text{ in }\, \R_{+}^{n+1},\\
				&\sB^3_i(u)=0&&\text{ on }\, \pa \R^{n+1}_{+},\\
				&\sB_3^3(u)=0 &&\text{ on }\,  \pa \R^{n+1}_{+}.
			\end{aligned}\right.
			\]
			Then, for every $m\in\N$,
\[
\|u\|_{C^{m}(\cB^{+}_{3/4}\backslash \cB^{+}_{1/2})}\leq C(m,n)\|u\|_{L^1(\cB^{+}_{1}\backslash \cB^{+}_{1/4})}.
\]
		\end{lem}


			%

			\subsection{Poisson kernels and Green functions}\label{sec:2.3}

		The complementing condition established above provides the elliptic
framework for the biharmonic boundary value problems associated with the
pairs $(\sB_i^3,\sB_j^3)$ satisfying $i+j\neq3$. We now recall the explicit representation formulas for these problems in some standard models. More precisely, we record the Poisson kernels and Green functions obtained by Chen--Zhang \cite{CZ202406} on the upper half-space and, via conformal equivalence, on the unit ball.  For sixth-order analogues, we refer to the work of the last named author
\cite{Z2026}.
			%

			Let $i,j\in\{0,1,2,3\}$ with $i<j$. A pair of kernel functions $(P_i^3,P_j^3)$ on $\R^{n+1}_+$ is called the biharmonic Poisson kernel associated with the conformal boundary operator pair $(\sB_i^3,\sB_j^3)$ if, for suitable boundary data $f_i,f_j$, the Poisson integral
			\begin{align*}
				v(X) & =[P_i^3 * f_i+P_j^3 * f_j](X) \\
				& =\int_{\pa \mathbb{R}_{+}^{n+1}} P_i^3(x-y, t) f_i(y)\, \d  y+\int_{\pa  \mathbb{R}_{+}^{n+1}} P_j^3(x-y, t) f_j(y) \, \d  y
			\end{align*}
			is well-defined and gives a classical solution of
			\be\label{vzeroin}
			\left\{\begin{aligned}
				&\Delta^2 u=0  &&\text{ in }\,  \R_{+}^{n+1},\\
				&\sB^3_i(u)=f_i&&\text{ on }\, \pa \R_+^{n+1},\\
				&\sB_j^3(u)=f_j&&\text{ on }\,  \pa \R_+^{n+1}.
			\end{aligned}\right.
			\ee

			We shall use the following explicit formulas for the biharmonic Poisson kernels
			on the upper half-space, whose proof can be found in \cite[Theorem 1.1]{CZ202406}.
			\begin{lem}[Poisson kernel in $\R_{+}^{n+1}$]\label{lem:Poissonformula}
				Let $n\ge4$ and $X=(x,t)\in\R^{n+1}_+$.  Then
				\begin{align*}
					P_0^3(X)=&\,\frac{2(n+1)}{|\Sn|}\frac{t^3}{(t^2+|x|^2)^{\frac{n+3}{2}}},\nonumber\\
					P_1^3(X)=&\,-\frac{2}{|\Sn|}\frac{t^2}{(t^2+|x|^2)^{\frac{n+1}{2}}},\\
					P_2^3(X)=&\,-\frac{1}{(n-1)|\Sn|}\frac{t}{(t^2+|x|^2)^{\frac{n-1}{2}}},\\
					P_3^3(X)=&\,\frac{1}{(n-1)(n-3)|\Sn|}\frac{1}{(t^2+|x|^2)^{\frac{n-3}{2}}}.
				\end{align*}
				Suppose that, for $k\in\{i,j\}$, 
				$f_k\in C^{4-k,\gamma}(\Rn)$ for some $\gamma\in (0,1)$  and
				$f_k=O(|x|^{-k-\delta_k})$ as $|x| \to \infty$ for some $\delta_k >0$. Then, for $i,j\in\{0,1,2,3\}$ with $i<j$ and
				$i+j\ne3$, the Poisson kernel for \eqref{vzeroin} is given by
				$P_i^3\oplus P_j^3$.
			\end{lem}
			We next turn to the Green functions for the corresponding homogeneous
boundary conditions. Let $\Gamma$ denote the fundamental solution of
$\Delta^2$ in $\R^{n+1}$. For $n\geq2$,
\[
\Ga(X-Y)=\Ga(|X-Y|)=\left\{\begin{aligned}& \frac{1}{2(n-1)(n-3)|\Sn|}|X-Y|^{3-n}&& \text{ if }\, n\neq 3, \\&-\frac{1}{4|\S^3|}\log |X-Y|&&\text{ if }\, n=3. \end{aligned}\right.
\]
		
          Let $i,j\in\{0,1,2,3\}$ with $i<j$ and $i+j\neq3$, and consider
			\be\label{vzerobdr}
			\left\{\begin{aligned}
				&\Delta^2 u=f  &&\text{ in }\,  \R_+^{n+1},\\
				&\sB^3_i(u)=0&&\text{ on }\, \pa \R_+^{n+1},\\
				&\sB_j^3(u)=0&&\text{ on }\,  \pa \R_+^{n+1}.
			\end{aligned}\right.
			\ee
			A function $G^{(i, j)}:(X, Y) \in \overline{\mathbb{R}_{+}^{n+1}} \times \overline{\mathbb{R}_{+}^{n+1}} \mapsto \mathbb{R} \cup\{\infty\}$ is called the Green function for \eqref{vzerobdr} if $G^{(i, j)}(X, Y)=\Gamma(X-Y)+H^{(i, j)}(X, Y)$, where, for each fixed $Y\in\R^{n+1}_+$, the regular part
			$H^{(i,j)}(\cdot,Y)$ satisfies
\[
\left\{\begin{aligned} &\Delta^2 H^{(i,j)}=0 &&\text{ in }\, \R_+^{n+1},\\ &\sB^3_i(H^{(i,j)})=-\sB_{i}^3(\Ga)&&\text{ on }\, \pa \R_+^{n+1},\\ &\sB_j^3(H^{(i,j)})=-\sB_{j}^3(\Ga)&&\text{ on }\, \pa \R_+^{n+1}. \end{aligned}\right.
\]

		The following explicit formulas will be used later. Their proof can be found in \cite[Theorem 1.2]{CZ202406}.
			\begin{lem}[Green functions for $\R_{+}^{n+1}$]\label{lem:Greenfunctions}
				Let $n\geq4$. For $X=(x,t)\in\R^{n+1}_+$, set $\bar{X}=(x,-t)$.  Then the biharmonic Green functions in $\R^{n+1}_+$ are given by
				\begin{align*}
					G^{(0,1)}(X,Y)=&\,\Ga(X-Y)+\frac{1}{2}\Ga(|\bar{X}-Y|)\Big[(n-3)\frac{|X-Y|^2}{|\bar{X}-Y|^2}-(n-1)\Big],\\
					G^{(0,2)}(X,Y)=&\,\Ga(X-Y)-\Ga(\bar{X}-Y),\\
					G^{(1,3)}(X,Y)=&\,\Ga(X-Y)+\Ga(\bar{X}-Y),\\
					G^{(2,3)}(X,Y)=&\,\Ga(X-Y)-\frac{1}{2}\Ga(|\bar{X}-Y|)\Big[(n-3)\frac{|X-Y|^2}{|\bar{X}-Y|^2}-(n-1)\Big].
				\end{align*}
				Moreover,
\[
0 \leq G^{(0,1)}(X, Y) \leq G^{(0,2)}(X, Y) \leq G^{(1,3)}(X, Y) \leq G^{(2,3)}(X, Y)
\]
				for all $X,Y\in\overline{\R^{n+1}_+}$, $X\ne Y$, with equality if and only
				if either $X\in\partial\R^{n+1}_+$ or $Y\in\partial\R^{n+1}_+$.
			\end{lem}

\begin{rem}\label{rem:boundary-G-and-P}
				For $i\in\{1,2\}$, one has  $G^{(i,3)}(X,Y)<2\Ga(X-Y)$ for $X,Y\in\R^{n+1}_+$, $\ X\ne Y$.  Moreover, the boundary traces in the first variable satisfy
\[
G^{(i,3)}(X,Y)|_{X=(x,0)}=2\Ga((x,0)-Y) \quad \text{ and }\quad P_3^3(X-(y,0))|_{X=(x,0)}
=\frac12\Ga^n(x-y),\quad x\neq y,
\]
				where
\[
\Ga^n(x-y)=\frac{2}{(n-1)(n-3)|\Sn|}\frac{1}{|x-y|^{n-3}}
\]
				is the fundamental solution of $(-\Delta_{\Rn})^{3/2}$.
			\end{rem}

\begin{rem}\label{rem:Greenfuncions-in-Ball}
Since $\R^{n+1}_+$ and $\B^{n+1}$ are conformally equivalent, the
preceding Poisson kernels and Green functions induce corresponding kernels
$\bar P_k^3$ and $\bar G^{(i,j)}$ on the unit ball. Their explicit formulas
and basic properties can be found in \cite[Section~4]{CZ202406}.
\end{rem}

			\section{The method of moving spheres}\label{sec:3}

			\subsection{Differential and integral equations}
Throughout this section, we assume that $n\geq4$.  The moving-spheres argument will be carried out in 
the associated integral equation. We therefore begin by deriving a
Green--Poisson representation for finite-volume solutions of the
differential problem.

			\begin{prop}\label{prop:integral-representation}
				Let $i\in\{1,2\}$ and $c_i,c_3\in\mathbb R$.	Assume that  $u\in C^{4}(\overline{{\R}_{+}^{n+1}})$  is a nonnegative solution of
\[
\left\{\begin{aligned} &\Delta^2 u= u^{p^{*}} &&\text{ in }\, \R_+^{n+1},\\ &\sB^3_iu=c_iu^{p_i^{*}}&&\text{ on }\, \pa \R^{n+1}_{+},\\ &\sB^3_3u=c_3u^{p_3^{*}} &&\text{ on }\,\pa \R^{n+1}_{+}, \end{aligned}\right.
\]
satisfying the finite-volume conditions \eqref{Finitevolume}. Then $u$ admits the integral representation
				\be\label{Integral-equ}
				u(X)=\int_{\R^{n+1}_{+}}G^{(i,3)}(X,Y)u^{p^{*}}(Y)\, \d Y
				+\sum_{k=i,3}\int_{\Rn}c_kP^3_k(X-(y,0))u^{p_k^{*}}(y,0)\, \d y,
				\ee
				where $G^{(i,3)}$ and $P_k^3$ are the Green function and Poisson kernels,
respectively, introduced in Section \ref{sec:2.3}.
			\end{prop}

		We first establish a Green--Poisson representation for the associated
linear biharmonic problem.
			\begin{prop}\label{prop:2}
				Let $i\in\{1,2\}$. Assume that $f\in C^{\infty}(\overline{{\R}_{+}^{n+1}})\cap L^{\frac{2(n+1)}{n+5}}(\R_{+}^{n+1})$ and $f_k\in C^{\infty}(\Rn)\cap L^{\frac{2n}{n+2k-3}}(\Rn)$, $k \in \{1,2,3\}$.  Define
				\be\label{v-integralformula}
				v(X)=\int_{\R^{n+1}_{+}}G^{(i,3)}(X,Y)f(Y)\, \d Y+\sum_{k=i,3}\int_{\Rn}P^3_k(X-(y,0))f_k(y)\, \d y.
				\ee
				Then $v\in C^4(\overline{\R ^{n+1}_+})$ and satisfies
				\be\label{bdry_value-Possion_kernel}
				\left\{\begin{aligned}
					&\Delta^2 v=f  &&\text{ in }\,  \R^{n+1}_+,\\
					&\sB^3_i(v)=f_i&&\text{ on }\, \pa\R_+^{n+1},\\
					&\sB_3^3(v)=f_3&&\text{ on }\,  \pa\R_+^{n+1}.
				\end{aligned}\right.
				\ee
			\end{prop}


The main ingredient in the proof of Proposition \ref{prop:2} is the
following weak-type estimate for the Green--Poisson potential.
			\begin{lem}\label{lem:Weaknorm}
			Let $v$ be defined by \eqref{v-integralformula}, where  $f\in C^{\infty}(\overline{\R_{+}^{n+1}})\cap L^{\frac{2(n+1)}{n+5}}(\R_{+}^{n+1})$ and $f_k\in C^{\infty}(\Rn)\cap L^{\frac{2n}{n+2k-3}}(\Rn)$, $k\in\{i,3\}$. Then  $v$ belongs to $L_w^{\frac{2(n+1)}{n-3}}(\R^{n+1}_{+})$.  More precisely, there exists a constant $C>0$, depending only on $n$, such that, for every $\tau>0$,
				\[
				|\{X\in \R^{n+1}_{+}:|v(X)|>\tau \}|\leq C\tau ^{-\frac{2(n+1)}{n-3}}\Big(\|f\|_{L^{\frac{2(n+1)}{n+5}}(\R^{n+1}_{+})}^{\frac{2(n+1)}{n-3}}+\sum_{k=i,3}\|f_k\|_{L^{\frac{2n}{n+2k-3}}(\Rn)}^{\frac{2(n+1)}{n-3}}\Big).
				\]
				Equivalently,
\[
\|v\|_{L_w^{\frac{2(n+1)}{n-3}}(\R ^{n+1}_+)} \le C \Big( \|f\|_{L^{\frac{2(n+1)}{n+5}}(\R ^{n+1}_+)} + \sum_{k=i,3} \|f_k\|_{L^{\frac{2n}{n+2k-3}}(\R ^n)} \Big).
\]
			\end{lem}
\begin{proof}
The proof  is partially inspired by Dou and Zhu \cite{DZ2015-2,DZ2015-1} and analogous to Lemma 2.7 in \cite{SX2016}, so we omit  the details here. 
\end{proof}

We next record the covariance of the integral formulation under spherical
inversion, which will be used in the moving-spheres argument.

			For $X_0=(x_0,0)\in\pa \R ^{n+1}_+$ and $\lam>0$, we define the spherical inversion centered at $X_0$ with radius $\lam$ by
			\be \label{Spherical-inversion}
			X^\lam
			=(x^\lam,t^\lam)=
			\cI_{X_0,\lam}(X)
			:=
			X_0+\frac{\lam^2(X-X_0)}{|X-X_0|^2},
			\quad
			X=(x,t)\in\overline{\R^{n+1}_+}\setminus\{X_0\}.
			\ee
		Since $t^\lam = \frac{\lam^2t}{|X-X_0|^2}$, the inversion $\cI_{X_0,\lam}$ preserves $\R^{n+1}_+$ and its boundary
$\partial\R^{n+1}_+$. Its restriction to the boundary $\pa \R_+^{n+1}=\Rn$ is given by
			\be \label{bdry-Spherical-inversion}
			x^\lam = x_0+\frac{\lam^2(x-x_0)}{|x-x_0|^2}, \quad x\in\Rn\setminus\{x_0\}.
			\ee
	For $u\in C^\infty(\overline{\R ^{n+1}_+})$, we define its Kelvin transform with respect to $(X_0,\lam)$ by
\[
u_{X_0,\lam}(X) = \Big( \frac{\lam}{|X-X_0|}\Big)^{n-3} u(X^\lam), \quad X\in\overline{\R^{n+1}_+}\setminus\{X_0\}.
\]
For the interior datum $f$ and the boundary data $f_i,f_3$ in
\eqref{v-integralformula}, define
\[
f_{X_0,\lam}(X) = \Big( \frac{\lam}{|X-X_0|} \Big)^{n+5} f(X^\lam), \quad X\in\R ^{n+1}_+, \]
and, for $k\in\{i,3\}$, 
\[ (f_k)_{X_0,\lam}(x) = \Big( \frac{\lam}{|x-x_0|} \Big)^{n+2k-3} f_k(x^\lam), \quad x\in\Rn\setminus\{x_0\}.
\]

With the above notation, the integral representation
\eqref{v-integralformula} is covariant under the Kelvin transform.
			\begin{lem}\label{lem:Kelvin}
				Let $v$ be defined by \eqref{v-integralformula}. Then, for every
$X_0\in\partial\R^{n+1}_+$, $\lam>0$, and $X\in\R^{n+1}_+$, one has
		\[
					v_{X_0,\lam }(X)=\int_{\R^{n+1}_{+}}G^{(i,3)}(X,Y)f_{X_0,\lam }(Y)\, \d Y+\sum_{k=i,3}\int_{\Rn}P^3_k(X-(y,0))(f_k)_{X_0,\lambda}(y)\, \d y.
			\]
			\end{lem}

			\begin{proof}
				The identity follows from the conformal covariance of the Green kernel and the
				Poisson kernels under the spherical inversion $\cI_{X_0,\lam}$,
				together with the transformation laws for $f$, $f_i$, and $f_3$. 
				Hence we omit the details here.
			\end{proof}

With the preceding estimates and covariance properties at hand, we can now prove Proposition \ref{prop:2}.
			\begin{proof}[Proof of Proposition \ref{prop:2}]
				We first take $X_0=0$ and $\lam=1$. By the definitions of the Kelvin transforms, one checks directly that,
\[
f_{0,1}\in C^{\infty}(\overline{{\R}_{+}^{n+1}}\backslash\{0\})\cap L^{\frac{2(n+1)}{n+5}}(\R_{+}^{n+1}),\quad f_{0,1}(X)=O(|X|^{-n-5}) \quad \text{ as }\,|X|\to+\infty,
\]
				and, for $k\in\{i,3\}$,
                \[
                (f_{k})_{0,1}\in  C^{\infty}(\Rn \backslash\{0\})\cap L^{\frac{2n}{n+2k-3}}(\Rn), \quad
				(f_{k})_{_{0,1}}(x)=O(|x|^{3-2k-n}) \quad \text{ as }\, |x|\to+\infty.
				\]
				For each fixed boundary point $x\ne0$, the preceding regularity and decay allow us to apply the Lebesgue dominated convergence theorem and obtain
				\begin{align*} \lim_{t\to0^+} \int_{\Rn} \frac{t(f_3)_{0,1}(y)} {(t^2+|x-y|^2)^{\frac{n-1}{2}}}\,\d y &=0,\\ \lim_{t\to0^+} \int_{\Rn} P_0^3(x-y,t)(f_k)_{0,1}(y)\,\d y &=(f_k)_{0,1}(x),\\ \lim_{t\to0^+} \int_{\Rn} P(x-y,t)(f_k)_{0,1}(y)\,\d y &=(f_k)_{0,1}(x). \end{align*} Here $P(X)=\frac{2}{|\Sn |}\frac{t}{|X|^{n+1}}$ is the classical Poisson kernel on $\R ^{n+1}_+$. Combining these boundary limits with Lemma \ref{lem:Kelvin} and arguing
as in \cite[Theorem 1.1]{CZ202406}, we obtain
\[
\left\{ \begin{aligned} &\Delta^2 v_{0,1}=f_{0,1} &&\text{ in }\,\R ^{n+1}_+,\\ &\sB_i^3v_{0,1}=(f_i)_{0,1} &&\text{ on }\,\pa \R ^{n+1}_+\setminus\{0\},\\ &\sB_3^3v_{0,1}=(f_3)_{0,1} &&\text{ on }\,\pa \R ^{n+1}_+\setminus\{0\}. \end{aligned} \right.
\]
				Applying the inverse Kelvin transform and using the conformal covariance of
the biharmonic equation and the boundary operators, we conclude that
\[
\left\{\begin{aligned} &\Delta^2 v=f &&\text{ in }\,\R_{+}^{n+1},\\ &\sB^3_i(v)=f_{i}&&\text{ on }\, \pa \R^{n+1}_{+}\backslash\{0\},\\ &\sB_3^3(v)=f_{3} &&\text{ on }\, \pa \R^{n+1}_{+}\backslash\{0\}. \end{aligned}\right.
\]

It remains to recover the boundary identities at the origin.				Take an arbitrary $X_0\in\pa \R ^{n+1}_+\setminus\{0\}$ and $\lam=1$. Applying the same argument with this new center shows that $v$ satisfies the same boundary value problem on $\overline{\R ^{n+1}_+}\setminus\{X_0\}$. Since the center can be chosen away from any prescribed boundary point, the boundary identities hold on the whole $\pa \R ^{n+1}_+$. Hence $v$ satisfies \eqref{bdry_value-Possion_kernel}.
			\end{proof}
            
We are now in a position to prove Proposition \ref{prop:integral-representation}.

			\begin{proof}[Proof of Proposition \ref{prop:integral-representation}]
				Define
\[
v(X)=\int_{\R^{n+1}_{+}}G^{(i,3)}(X,Y)u^{p^{*}}(Y)\, \d Y+\sum_{k=i,3}\int_{\Rn}c_kP^3_k(X-(y,0))u^{p_k^{*}}(y,0)\, \d y,
\]
	and			set $w:=u-v$. By Proposition \ref{prop:2}, $v$ satisfies the same inhomogeneous boundary value problem as $u$. Hence $w$ solves
				\be\label{w-eq}
                \left\{\begin{aligned}
					&\Delta^2 w=0 &&\text{ in }\, \R_+^{n+1},\\
					&\sB^3_iw=0&&\text{ on }\,  \pa \R^{n+1}_{+},\\
					&\sB^3_3w=0 &&\text{ on }\, \pa \R^{n+1}_{+}.
				\end{aligned}\right.
				\ee

			We first show that $w$ vanishes at infinity. Consider the Kelvin transform
with center $0$ and radius $1$,
\[ 
w_{0,1}(X) = |X|^{3-n}w\Big(\frac{X}{|X|^2}\Big) .
\] By the conformal covariance of the biharmonic equation and the boundary
operators, $w_{0,1}$ satisfies the same homogeneous boundary value problem
on $\overline{\R ^{n+1}_+}\setminus\{0\}$.

				We claim that \[ w_{0,1}\in L^p(\cB_1^{+}) \quad\text{ for every }\,1\le p<\frac{2(n+1)}{n-3}. \]
				Indeed, by Lemma \ref{lem:Weaknorm}, for every $p<\frac{2(n+1)}{n-3}$, we have
				\begin{align*}
					&\,\int_{B_1^{+}}|v_{0,1}(X)|^p\,\d X\\ =&\, p\int_0^\infty| \{X\in \cB_1^{+}:\ |v_{0,1}(X)|>\tau\}| \tau^{p-1}\,  \d \tau\\ \le&\, |\cB_1^{+}| + C \Big( \|(u_{0,1})^{p^*}\|_{L^{\frac{2(n+1)}{n+5}}(\R ^{n+1}_+)} + \sum_{k=i,3} \|(u_{0,1})^{p_k^*}\|_{L^{\frac{2n}{n+2k-3}}(\Rn)} \Big) \int_1^\infty \tau^{p-\frac{2(n+1)}{n-3}-1}\,  \d \tau\\ \le&\, C \Big( 1+ \|u_{0,1}\|_{L^{\frac{2(n+1)}{n-3}}(\R ^{n+1}_+)}^{p^*} + \sum_{k=i,3} \|u_{0,1}(\cdot,0)\|_{L^{\frac{2n}{n-3}}(\Rn)}^{p_k^*} \Big) <+\infty.
				\end{align*}
				Here we used the conformal invariance of the critical norms:
\[
\|u_{0,1}\|_{L^{\frac{2(n+1)}{n-3}}(\R ^{n+1}_+)} = \|u\|_{L^{\frac{2(n+1)}{n-3}}(\R ^{n+1}_+)}\quad \text{ and }\quad \|u_{0,1}(\cdot,0)\|_{L^{\frac{2n}{n-3}}(\Rn)} = \|u(\cdot,0)\|_{L^{\frac{2n}{n-3}}(\Rn)}. \] Since $u_{0,1}\in L^{\frac{2(n+1)}{n-3}}(\cB_1^{+})$, it follows that $w_{0,1}=u_{0,1}-v_{0,1}\in L^p(\cB_1^{+})$ for every $p<\frac{2(n+1)}{n-3}$. Applying the local elliptic estimate in Lemma \ref{lem:elliptic-estimate} to $w_{0,1}$ on boundary annuli, we obtain, for sufficiently small $r>0$, \[ \|w_{0,1}\|_{L^\infty(\cB_{3r/4}\setminus \cB_{r/2}^{+})} \le Cr^{-(n+1)} \|w_{0,1}\|_{L^1(\cB_r^{+}\setminus \cB_{r/4}^{+})}.
\]
				By H\"older's inequality,
\[
\|w_{0,1}\|_{L^\infty(\cB_{3r/4}\setminus \cB_{r/2}^{+})} \le Cr^{-\frac{n+1}{p}} \|w_{0,1}\|_{L^p(\cB_r^{+}\setminus \cB_{r/4}^{+})}. 
\] Since $w_{0,1}\in L^p(\cB_1^{+})$, this implies 
\[ 
\|w_{0,1}\|_{L^\infty(\cB_{3r/4}\setminus \cB_{r/2}^{+})} = O(r^{-\frac{n+1}{p}}) \quad\text{ as }\,r\to0^{+}. 
\] 
Transforming back to the original variables, we obtain
\[ |w(X)| = O( |X|^{3-n+\frac{n+1}{p}} ) \quad\text{ as }\, |X|\to+\infty . 
\] 
Choose $\frac{n+1}{n-3}<p<\frac{2(n+1)}{n-3}$. Then $3-n+\frac{n+1}{p}<0$, and hence $ w(X)\to0 $ as $|X|\to+\infty$. Since $w$ is smooth up to the boundary, it follows that $w$ is bounded on
$\overline{\R^{n+1}_+}$. 

Finally, we will  show that a bounded solution of the above homogeneous
boundary value problem \eqref{w-eq} is constant. Applying the elliptic estimate to $w$ on large half-balls, there exists a constant $C>0$, independent of $R$, such that \be \label{w-decay-1} \|\nabla w\|_{L^\infty(\cB_R^{+})} \le CR^{-1}\|w\|_{L^\infty(\cB_{2R}^{+})} \quad\text{ for }\,R\ge1. \ee Since $w$ is bounded, letting $R\to+\infty$ in \eqref{w-decay-1} yields 
\[ 
\nabla w\equiv0 \quad\text{ in }\,\R ^{n+1}_+.
\] Thus $w$ is constant. Since $w(X)\to0$ as $|X|\to+\infty$, this constant must be zero. Hence $w\equiv0$, and therefore $u=v$, which proves
\eqref{Integral-equ}.
\end{proof}

\subsection{Proof of Theorem \ref{thm:main0}} 

Since $\kap>0$ is fixed, we first normalize the interior equation. Set $\tilde u=\kap^{\frac{n-3}{8}}u$.  Then $\Delta^2\tilde u=\tilde u^{p^*}$, while the boundary constants become $\tilde c_k=\kap^{-\frac{k}{4}}c_k$, $k=i,3$. Since the sign assumptions and the finiteness of the volume integrals are preserved, we henceforth assume $\kappa=1$ and omit the tildes. 

Let $v(x):=u(x,0)$, $x\in\Rn$. By Proposition \ref{prop:integral-representation}, Lemma \ref{lem:Poissonformula}, and Remark \ref{rem:boundary-G-and-P}, the pair $(u,v)$
satisfies the integral system \be \label{System-equation}
\left\{ \begin{aligned} u(X) =&\, \int_{\R ^{n+1}_+} G^{(i,3)}(X,Y)u^{p^*}(Y)\,\d Y +\sum_{k=i,3} c_k\int_{\Rn} P_k^3(X-(y,0))v^{p_k^*}(y)\, \d y,&&\forall\, X\in \R_{+}^{n+1},\\ v(x) =&\, 2\int_{\R ^{n+1}_+} \Ga((x,0)-Y)u^{p^*}(Y)\,\d Y + \frac{c_3}{2} \int_{\Rn} \Ga^n(x-y)v^{p_3^*}(y)\,\d y, &&\forall\, x\in \Rn. \end{aligned} \right. \ee Here $i\in\{1,2\}$, $c_i\le0$, and $c_3\ge0$. Moreover, \[ (u,v)\in ( C(\overline{\R ^{n+1}_+}) \cap L^{\frac{2(n+1)}{n-3}}(\R ^{n+1}_+) ) \times ( C(\Rn) \cap L^{\frac{2n}{n-3}}(\Rn) ).
\]

		We now apply the method of moving spheres to \eqref{System-equation},
following the integral approach used in related conformally invariant
problems; see Dou--Zhu \cite{DZ2015-2} and Tang--Dou \cite{TD2018}.

		Fix $X_0=(x_0,0)\in\partial\R^{n+1}_+$ and $\lam>0$. Let
$X^\lam=\cI_{X_0,\lam}(X)$ be the spherical inversion defined in
\eqref{Spherical-inversion}, and let $x^\lam$ denote its restriction to
the boundary as in \eqref{bdry-Spherical-inversion}. Set
\[
\cB_\lam^+(X_0) := \cB_\lam(X_0)\cap\R ^{n+1}_+\quad \text{ and }\quad \Si_{X_0,\lam} := \R ^{n+1}_+\setminus\overline{\cB_\lam^+(X_0)}.
\]
			Their boundary parts are denoted by
\[
\pa '\cB_\lam^+(X_0) = B_\lam(x_0) \quad \text{ and }\quad \Si_{x_0,\lam} := \Rn\setminus\overline{B_\lam(x_0)}.
\]
The Kelvin transforms of $u$ and $v$ are given by
\[
u_{X_0,\lam}(X) := \Big(\frac{\lam}{|X-X_0|}\Big)^{n-3} u(X^\lam), \quad X\in\R ^{n+1}_+,
\]
			and
\[
v_{X_0,\lam}(x) := \Big(\frac{\lam}{|x-x_0|}\Big)^{n-3} v(x^\lam), \quad x\in\Rn\setminus\{x_0\}.
\]

The conformal invariance of the critical exponents implies that
\eqref{System-equation} is preserved under these Kelvin transforms.
For later use, we record the corresponding integral identities.
			\begin{lem}\label{lem:Identity}
				Let $(u,v)$ be a positive solution of \eqref{System-equation}. Then, for
every $X_0=(x_0,0)\in\partial\R^{n+1}_+$ and every $\lam>0$, one has
				\begin{align*}
					\left(\frac{\lam}{|X-X_0|}\right)^{n-3}
					\int_{\cB_\lam^+(X_0)}
					G^{(i,3)}(X^\lam,Y)u^{p^*}(Y)\,\d Y
					&=
					\int_{\Si_{X_0,\lam}}
					G^{(i,3)}(X,Y)
					(u_{X_0,\lam}(Y))^{p^*}\,\d Y,
					\\
					\left(\frac{\lam}{|X-X_0|}\right)^{n-3}
					\int_{B_\lam(x_0)}
					P_i^3(X^\lam-(y,0))v^{p_i^*}(y)\,\d y
					&=
					\int_{\Si_{x_0,\lam}}
					P_i^3(X-(y,0))
					(v_{X_0,\lam}(y))^{p_i^*}\,\d y,
					\\
					\left(\frac{\lam}{|X-X_0|}\right)^{n-3}
					\int_{B_\lam(x_0)}
					P_3^3(X^\lam-(y,0))v^{p_3^*}(y)\,\d y
					&=
					\int_{\Si_{x_0,\lam}}
					P_3^3(X-(y,0))
					(v_{X_0,\lam}(y))^{p_3^*}\,\d y,
					\\
					\Big(\frac{\lam}{|x-x_0|}\Big)^{n-3}
					\int_{\cB_\lam^+(X_0)}
					\Ga((x^\lam,0)-Y)u^{p^*}(Y)\,\d Y
					&=
					\int_{\Si_{X_0,\lam}}
					\Ga((x,0)-Y)
					(u_{X_0,\lam}(Y))^{p^*}\,\d Y,
					\\
					\Big(\frac{\lam}{|x-x_0|}\Big)^{n-3}
					\int_{B_\lam(x_0)}
					\Ga^n(x^\lam-y)v^{p_3^*}(y)\,\d y
					&=
					\int_{\Si_{x_0,\lam}}
					\Ga^n(x-y)
					(v_{X_0,\lam}(y))^{p_3^*}\,\d y.
				\end{align*}
				The same identities remain valid if the regions
				$\cB_\lam^+(X_0)$ and $\Si_{X_0,\lam}$, or
				$B_\lam(x_0)$ and $\Si_{x_0,\lam}$, are interchanged.
				Consequently, the Kelvin transforms $(u_{X_0,\lam},v_{X_0,\lam})$ also
				satisfy the integral system
				\[	\left\{
				\begin{aligned}
					u_{X_0,\lam}(X)
					=&\,
					\int_{\R ^{n+1}_+}
					G^{(i,3)}(X,Y)
					(u_{X_0,\lam}(Y))^{p^*}\,\d Y
					+\sum_{k=i,3}\int_{\Rn}
					c_kP_k^3(X-(y,0))
					(v_{X_0,\lam}(y))^{p_k^*}	\,\d y,\\
					v_{X_0,\lam}(x)
					=&\,
					2\int_{\R ^{n+1}_+}
					\Ga((x,0)-Y)
					(u_{X_0,\lam}(Y))^{p^*}\,\d  Y
					+
					\frac{c_3}{2}
					\int_{\Rn}
					\Ga^n(x-y)
					(v_{X_0,\lam}(y))^{p_3^*}\,\d y.
				\end{aligned}
				\right.
				\]
			\end{lem}
			\begin{proof}
				The identities follow from the change of variables induced by the spherical inversion $\cI_{X_0,\lam}$, together with the conformal covariance of the Green kernel, the Poisson kernels, and the critical nonlinearities. The calculation is identical in spirit to that in Lemma \ref{lem:Kelvin}; hence we omit the details.
			\end{proof}

		For the remainder of this subsection, when the center 
$X_0=(x_0,0)\in\partial\R^{n+1}_+$ is fixed, we write, for simplicity,
\[
u_{\lam} :=u_{X_0,\lam}\quad \text{ and }\quad v_\lam:=v_{X_0,\lam}.
\]
Using Lemma \ref{lem:Identity}, we obtain the following difference identities, whose kernels have definite signs.
			\begin{lem}\label{lem:Difference}
			Let $(u,v)$ be a positive solution of the integral system
\eqref{System-equation}.  Then, for $X\in\Si_{X_0,\lam}$, one has
\[
u_{\lam} (X)-u(X)=I_1(X)+I_2(X)+I_3(X),
\]
				where
				\begin{align*}
					I_1(X)
					=&\,
					\int_{\Si_{X_0,\lam}}
					\cG^{(i,3)}(X,Y)
					[
					u_{\lam} ^{p^*}(Y)-u^{p^*}(Y)
					]\, \d Y,\\
					I_2(X)
					=&\,
					c_i\int_{\Si_{x_0,\lam}}
					P_i(X,y)
					[
					v_\lam^{p_i^*}(y)-v^{p_i^*}(y)
					]\, \d y,\\
					I_3(X)
					=&\,
					c_3\int_{\Si_{x_0,\lam}}
					P_3(X,y)
					[
					v_\lam^{p_3^*}(y)-v^{p_3^*}(y)
					]\, \d y.
				\end{align*}
				Similarly, for $x\in\Si_{x_0,\lam}$,
\[
v_\lam(x)-v(x)=II_1(x)+II_2(x),
\]
				where
				\begin{align*}
					II_1(x)
					=&\,
					2\int_{\Si_{X_0,\lam}}
					\cQ(x,Y)
					[
					u_{\lam} ^{p^*}(Y)-u^{p^*}(Y)
					]\, \d  Y,\\
					II_2(x)
					=&\,
					\frac{c_3}{2}\int_{\Si_{x_0,\lam}}
					\cQ^n(x,y)
					[
					v_\lam^{p_3^*}(y)-v^{p_3^*}(y)
					]\, \d  y.
				\end{align*}
				Here the kernels are defined by
				\begin{align*}
					\cG^{(i,3)}(X,Y)
					:=&\,
					G^{(i,3)}(X,Y)
					-
					\Big(\frac{\lam}{|X-X_0|}\Big)^{n-3}
					G^{(i,3)}(X^\lam,Y),\\
					P_i(X,y)
					:=&\,
					P_i^3(X-(y,0))
					-
					\Big(\frac{\lam}{|X-X_0|}\Big)^{n-3}
					P_i^3(X^\lam-(y,0)),\\
					P_3(X,y)
					:=&\,
					P_3^3(X-(y,0))
					-
					\Big(\frac{\lam}{|X-X_0|}\Big)^{n-3}
					P_3^3(X^\lam-(y,0)),\\
					\cQ(x,Y)
					:=&\,
					\Ga((x,0)-Y)
					-
					\Big(\frac{\lam}{|x-x_0|}\Big)^{n-3}
					\Ga((x^\lam,0)-Y),\\
					\cQ^n(x,y)
					:=&\,
					\Ga^n(x-y)
					-
					\Big(\frac{\lam}{|x-x_0|}\Big)^{n-3}
					\Ga^n(x^\lam-y).
				\end{align*}
			Moreover,
\[
\cG^{(i,3)}(X,Y)>0,\quad P_i(X,y)<0,\quad P_3(X,y)>0,
\]
				for $X,Y\in\Si_{X_0,\lam}$, $X\ne Y$, and $y\in\Si_{x_0,\lam}$, while
\[
\cQ(x,Y)>0,\quad \cQ^n(x,y)>0
\]
				for $x,y\in\Si_{x_0,\lam}$, $x\ne y$, and
				$Y\in\Si_{X_0,\lam}$.
			\end{lem}
			\begin{proof}
				We first derive the decomposition formula for $u_{\lam} -u$. For
				$X\in\Si_{X_0,\lam}$, splitting the integrals into the interior and exterior
of $\cB_\lam^+(X_0)$ gives
				\begin{align}
					&\,\int_{\R ^{n+1}_+}
					G^{(i,3)}(X,Y)u_{\lam} ^{p^*}(Y)\,\d Y
					-
					\int_{\R ^{n+1}_+}
					G^{(i,3)}(X,Y)u^{p^*}(Y)\,\d Y
					\nonumber\\
					=&\,
					\int_{\Si_{X_0,\lam}}
					G^{(i,3)}(X,Y)
					[
					u_{\lam} ^{p^*}(Y)-u^{p^*}(Y)
					]\,\d Y
					\nonumber\\
					&\,
					+
					\int_{\cB_\lam^+(X_0)}
					G^{(i,3)}(X,Y)u_{\lam}^{p^*}(Y)\,\d Y
					-
					\int_{\cB_\lam^+(X_0)}
					G^{(i,3)}(X,Y)u^{p^*}(Y)\,\d Y .
					\label{Diffi-Lem-formula-a}
				\end{align}
				By Lemma \ref{lem:Identity}, 
				\be\label{Diffi-Lem-formula-b}
				\int_{\cB_\lam^+(X_0)}
				G^{(i,3)}(X,Y)u_{\lam} ^{p^*}(Y)\,\d Y
				=
				\Big(\frac{\lam}{|X-X_0|}\Big)^{n-3}
				\int_{\Si_{X_0,\lam}}
				G^{(i,3)}(X^\lam,Y)u^{p^*}(Y)\,\d Y,
				\ee and 
				\be \label{Diffi-Lem-formula-c}
				\int_{\cB_\lam^+(X_0)}
				G^{(i,3)}(X,Y)u^{p^*}(Y)\,\d Y
				=
				\Big(\frac{\lam}{|X-X_0|}\Big)^{n-3}
				\int_{\Si_{X_0,\lam}}
				G^{(i,3)}(X^\lam,Y)u_{\lam} ^{p^*}(Y)\,\d Y .
				\ee
				Combining \eqref{Diffi-Lem-formula-a}--\eqref{Diffi-Lem-formula-c}, we obtain
				\begin{align*}
					&\,\int_{\R ^{n+1}_+}
					G^{(i,3)}(X,Y)u_{\lam} ^{p^*}(Y)\,\d Y
					-
					\int_{\R ^{n+1}_+}
					G^{(i,3)}(X,Y)u^{p^*}(Y)\,\d Y\\
					=&\,
					\int_{\Si_{X_0,\lam}}
					\Big[
					G^{(i,3)}(X,Y)
					-
					\Big(\frac{\lam}{|X-X_0|}	\Big)^{n-3}
					G^{(i,3)}(X^\lam,Y)
					\Big]
					[
					u_{\lam} ^{p^*}(Y)-u^{p^*}(Y)
					]\,\d Y .
				\end{align*}
				The same argument applied to the Poisson terms and to the boundary equation
for $v$ yields the remaining difference identities.

				It remains to establish the signs of the kernels. For
$X,Y\in\Si_{X_0,\lam}$, a direct computation gives
				\begin{align}
					&|X^\lam-Y|^2|X-X_0|^2-\lam^2|X-Y|^2\nonumber\\
					=&\,
					\Big|
					\frac{\lam^2(X-X_0)}{|X-X_0|}
					+
					|X-X_0|(X_0-Y)
					\Big|^2
					-
					\lam^2|(X-X_0)+(X_0-Y)|^2\nonumber\\
					=&\,
					(|X-X_0|^2-\lam^2)
					(|Y-X_0|^2-\lam^2)
					:=\delta^2>0 .\label{distance-identity1}
				\end{align}
                Since $X_0\in\pa \R ^{n+1}_+$, the same identity with $X$ replaced
				by its reflection $\bar X=(x,-t)$ gives
\be\label{distance-identity2}
|\bar X^\lam-Y|^2|\bar X-X_0|^2 - \lam^2|\bar X-Y|^2 = (|\bar X-X_0|^2-\lam^2) (|Y-X_0|^2-\lam^2) = \delta^2 .
\ee 
		
  For $i=1$, the explicit formula for $G^{(1,3)}$ in Lemma \ref{lem:Greenfunctions} and the preceding
distance inequalities \eqref{distance-identity1}--\eqref{distance-identity2} immediately imply
\[
G^{(1,3)}(X,Y) > \Big(\frac{\lam}{|X-X_0|} \Big)^{n-3} G^{(1,3)}(X^\lam,Y), \quad X,Y\in\Si_{X_0,\lam}.
\]

				We next consider $i=2$. Set 
				$C_n=2(n-1)(n-3)|\Sn |$.
				Using Lemma \ref{lem:Greenfunctions} and  \eqref{distance-identity1}--\eqref{distance-identity2} again, we compute
				\begin{align*} C_n
					\Big(\frac{\lam}{|X-X_0|}	\Big)^{n-3}
					G^{(2,3)}(X^\lam,Y)
					=&\,
					\lam^{n-3}
					\Big[
					\frac{1}{(\lam^2|X-Y|^2+\delta^2)^{\frac{n-3}{2}}}
					-
					\frac{n-3}{2}
					\frac{\lam^2|X-Y|^2+\delta^2}
					{(\lam^2|\bar X-Y|^2+\delta^2)^{\frac{n-1}{2}}}
					\\
					&\,
					+
					\frac{n-1}{2}
					\frac{1}
					{(\lam^2|\bar X-Y|^2+\delta^2)^{\frac{n-3}{2}}}
					\Big].
				\end{align*}
			
                We claim that
\[
\Big(\frac{\lam}{|X-X_0|} \Big)^{n-3} G^{(2,3)}(X^\lam,Y) < G^{(2,3)}(X,Y).
\]
				Indeed, for $0<a<b$, define
\[
\Phi(t) = (a+t)^{\frac{3-n}{2}} - \frac{n-3}{2}(a+t)(b+t)^{\frac{1-n}{2}} + \frac{n-1}{2}(b+t)^{\frac{3-n}{2}}, \quad t\ge0 .
\]
				Then
\[
\Phi'(t) = -\frac{n-3}{2} [ (a+t)^{\frac{1-n}{2}}+(b+t)^{\frac{1-n}{2}} ] + \frac{(n-3)(n-1)}{4} (a-b)(b+t)^{-\frac{n+1}{2}} <0 .
\]
				Thus $\Phi(t)<\Phi(0)$ for every $t>0$. Applying this with
				$a=\lam^2|X-Y|^2$,
				$b=\lam^2|\bar X-Y|^2$,
				$t=\delta^2$,
				proves the claim. Consequently, for each $i\in\{1,2\}$,
\[
\cG^{(i,3)}(X,Y)>0, \quad \forall\, X,Y\in\Si_{X_0,\lam},\, X\ne Y .
\]

			We next establish the corresponding signs for the Poisson kernels. From
the explicit formulas in Lemma \ref{lem:Poissonformula} and
				$t^\lam=\lam^2t/|X-X_0|^2$, we have, for $i\in\{1,2\}$,
\[
\Big(\frac{\lam}{|X-X_0|} \Big)^{n-3} P_i^3(X^\lam-(y,0)) = -C_i \Big(\frac{\lam}{|X-X_0|} \Big)^{n+3-2i} \frac{t^{3-i}}{|X^\lambda-(y, 0)|^{n+3-2 i}},
\]
				where $C_i=C_i(n)>0$. Similarly,
\[
\Big(\frac{\lam}{|X-X_0|} \Big)^{n-3} P_3^3(X^\lam-(y,0)) = C_3 \Big(\frac{\lam}{|X-X_0|} \Big)^{n-3} \frac{1}{|X^\lam-(y,0)|^{n-3}},
\]
				where $C_3=C_3(n)>0$. For
$X\in\Si_{X_0,\lam}$ and $y\in\Si_{x_0,\lam}$,
\[
|X^\lam-(y,0)|^2|X-X_0|^2 - \lam^2|X-(y,0)|^2 = (|X-X_0|^2-\lam^2) (|y-x_0|^2-\lam^2)>0.
\]
It follows that
\[
P_i(X,y)<0, \quad P_3(X,y)>0 .
\]

				Finally, let $x\in\Si_{x_0,\lam}$ and
$Y=(y,s)\in\Si_{X_0,\lam}$. Then
\[
(|x^\lam-y|^2+s^2)|x-x_0|^2 - \lam^2(|x-y|^2+s^2) = (|x-x_0|^2-\lam^2) (|Y-X_0|^2-\lam^2)>0 .
\]
Since $\Gamma(r)$ is strictly decreasing in $r$ for $n\geq4$, we conclude
that
\[
\Ga((x,0)-Y) > \Big(\frac{\lam}{|x-x_0|} \Big)^{n-3} \Ga((x^\lam,0)-Y),
\]
				and therefore $\cQ(x,Y)>0$. Taking $s=0$ in the same argument gives
\[
\cQ^n(x,y)>0, \quad\forall\, x,y\in\Si_{x_0,\lam},\ x\ne y .
\]
		This completes the proof.
			\end{proof}

			To start the moving sphere method, we define
\[
\Si_{u,\lam} := \{X\in\Si_{X_0,\lam}: u(X)<u_{\lam} (X)\} \quad \text{ and }\quad \Si_{v,\lam} := \{x\in\Si_{x_0,\lam}: v(x)<v_\lam(x)\}.
\]
The following lemma shows that the spheres can be started from sufficiently
small radii.
			\begin{lem}\label{lem:Start}
				Let $(u,v)$ be a positive solution of \eqref{System-equation}. For every
fixed $X_0=(x_0,0)\in\partial\R^{n+1}_+$, there exists
$\lam_0=\lam_0(X_0)>0$ such that, for  all
				$\lam\in (0,\lam_0)$,
\be\label{comparison-inequalities}
u_{\lam }\leq u\quad  \text{ a.e. in }\Si_{X_0,\lam} \quad \text{ and }\quad v_{\lam }\leq v \quad\text{ a.e. in } \Si_{x_0,\lam}.
\ee
			\end{lem}
			\begin{proof}
				For $X\in\Si_{u,\lam}$, the decomposition in Lemma
				\ref{lem:Difference} gives
\[
0<u_{\lam} (X)-u(X)=I_1(X)+I_2(X)+I_3(X).
\]
			We estimate the three terms on the right-hand side separately.

				Since $\cG^{(i,3)}(X,Y)>0$, the contribution of
$\Si_{X_0,\lam}\setminus\Si_{u,\lam}$ to $I_1$ is nonpositive. Moreover,
by Lemma \ref{lem:Greenfunctions},
\[
I_1(X) \le \int_{\Si_{u,\lam}} \cG^{(i,3)}(X,Y) [ u_{\lam} ^{p^*}(Y)-u^{p^*}(Y) ]\,\d Y\le C\int_{\Si_{u,\lam}} \frac{u_{\lam} ^{\frac{8}{n-3}}(Y)} {|X-Y|^{n-3}} (u_{\lam} (Y)-u(Y))\,\d Y,
\]
				where we used the inequality
				$a^p-b^p\le p a^{p-1}(a-b)$ for  $a>b>0$, $p>1$.

			Extending the integrand by zero outside $\Si_{u,\lam}$ and applying
Lemma \ref{HLS} in $\R^{n+1}$, we obtain
				\begin{align}
					\|I_1\|_{L^{\frac{2(n+1)}{n-3}}(\Si_{u,\lam})}
					&\le
					C
					\|
					u_{\lam} ^{\frac{8}{n-3}}
					(u_{\lam} -u)
					\|_{L^{\frac{2(n+1)}{n+5}}(\Si_{u,\lam})}
					\nonumber\\
					&\le
					C
					\|
					u_{\lam} ^{\frac{8}{n-3}}
					\|_{L^{\frac{n+1}{4}}(\Si_{u,\lam})}
					\|u_{\lam} -u\|_{L^{\frac{2(n+1)}{n-3}}(\Si_{u,\lam})}
					\nonumber\\
					&=
					C
					\|u\|_{L^{\frac{2(n+1)}{n-3}}(\Si_{u,\lam}^{\lam})}^{\frac{8}{n-3}}
					\|u_{\lam} -u\|_{L^{\frac{2(n+1)}{n-3}}(\Si_{u,\lam})},
					\label{lem:Start formula a}
				\end{align}
				where $\Si_{u,\lam}^{\lam}
				:=
				\cI_{X_0,\lam}(\Si_{u,\lam})
				\subset \cB_\lam^+(X_0)$.

				We next estimate $I_2$. Since $c_i\leq0$ and $P_i(X,y)<0$, the
contribution of
$\Si_{x_0,\lam}\setminus\Si_{v,\lam}$ is again nonpositive. Hence
\[
I_2(X) \le C\int_{\Si_{v,\lam}} K_i(X,y) v_\lam^{\frac{2i}{n-3}}(y) (v_\lam(y)-v(y))\,\d y,
\]
				where
\[
K_1(X,y)=\frac{t^2}{(t^2+|x-y|^2)^{\frac{n+1}{2}}} \quad\text{ and }\quad K_2(X,y)=\frac{t}{(t^2+|x-y|^2)^{\frac{n-1}{2}}}.
\]
				Applying the first Hardy--Littlewood--Sobolev type inequality in Lemma
				\ref{lem:Gluck}, with $(b,a)=(2,0)$ for $i=1$ and
$(b,a)=(1,2)$ for $i=2$, we obtain
				\begin{align}
					\|I_2\|_{L^{\frac{2(n+1)}{n-3}}(\Si_{u,\lam})}
					&\le
					C
					\|
					v_\lam^{\frac{2i}{n-3}}
					(v_\lam-v)
					\|_{L^{\frac{2n}{n+2i-3}}(\Si_{v,\lam})}
					\nonumber\\
					&\le
					C
					\|
					v_\lam^{\frac{2i}{n-3}}
					\|_{L^{\frac{n}{i}}(\Si_{v,\lam})}
					\|v_\lam-v\|_{L^{\frac{2n}{n-3}}(\Si_{v,\lam})}
					\nonumber\\
					&=
					C
					\|v\|_{L^{\frac{2n}{n-3}}(\Si_{v,\lam}^{\lam})}^{\frac{2i}{n-3}}
					\|v_\lam-v\|_{L^{\frac{2n}{n-3}}(\Si_{v,\lam})},
					\label{lem:Start formula b}
				\end{align}
				where $\Si_{v,\lam}^{\lam}
				:=
				\cI_{X_0,\lam}(\Si_{v,\lam})
				\subset B_\lam(x_0)$.

			Similarly, since $c_3\geq0$ and $P_3(X,y)>0$,
            \[
				I_3(X)
				\le
				C\int_{\Si_{v,\lam}}
				\frac{
					v_\lam^{\frac{6}{n-3}}(y)
					(v_\lam(y)-v(y))}
				{(t^2+|x-y|^2)^{\frac{n-3}{2}}}
				\,\d y .
				\]
				Applying again the first inequality in Lemma \ref{lem:Gluck}, now with
$(b,a)=(0,4)$, we have
				\be\label{lem:Start formula c}
				\|I_3\|_{L^{\frac{2(n+1)}{n-3}}(\Si_{u,\lam})}
				\le
				C
				\|
				v_\lam^{\frac{6}{n-3}}
				(v_\lam-v)
				\|_{L^{\frac{2n}{n+3}}(\Si_{v,\lam})}
				\le
				C
				\|v\|_{L^{\frac{2n}{n-3}}(\Si_{v,\lam}^{\lam})}^{\frac{6}{n-3}}
				\|v_\lam-v\|_{L^{\frac{2n}{n-3}}(\Si_{v,\lam})}.
				\ee

			Combining \eqref{lem:Start formula a}--\eqref{lem:Start formula c} and
using the facts
$\Si_{u,\lam}^{\lam}\subset\cB_\lam^+(X_0)$, $\Si_{v,\lam}^{\lam}\subset B_\lam(x_0)$,
we obtain
				\begin{align}
					\|u_{\lam} -u\|_{L^{\frac{2(n+1)}{n-3}}(\Si_{u,\lam})}
					\le&\,
					C_1
					\|u\|_{L^{\frac{2(n+1)}{n-3}}(\cB_\lam^+(X_0))}^{\frac{8}{n-3}}
					\|u_{\lam} -u\|_{L^{\frac{2(n+1)}{n-3}}(\Si_{u,\lam})}
					\nonumber\\
					&\,+
					C_1
					\|v\|_{L^{\frac{2n}{n-3}}(B_\lam(x_0))}^{\frac{2i}{n-3}}
					\|v_\lam-v\|_{L^{\frac{2n}{n-3}}(\Si_{v,\lam})}
					\nonumber\\
					&\,+
					C_1
					\|v\|_{L^{\frac{2n}{n-3}}(B_\lam(x_0))}^{\frac{6}{n-3}}
					\|v_\lam-v\|_{L^{\frac{2n}{n-3}}(\Si_{v,\lam})}.
					\label{lem:Start formula a1}
				\end{align}

				We now derive the corresponding estimate for $v_\lam-v$. For
$x\in\Si_{v,\lam}$, Lemma \ref{lem:Difference} gives
\[
0<v_\lam(x)-v(x)=II_1(x)+II_2(x).
\]
		Since $\cQ(x,Y)>0$, only the contribution from $\Si_{u,\lam}$ needs to
be retained, and hence
\[
				II_1(x)
				\le
				C\int_{\Si_{u,\lam}}
				\frac{
					u_{\lam} ^{\frac{8}{n-3}}(Y)
					(u_{\lam} (Y)-u(Y))}
				{(|x-y|^2+s^2)^{\frac{n-3}{2}}}
				\,\d Y,
				\quad Y=(y,s).
				\]
				The second Hardy--Littlewood--Sobolev type inequality in
Lemma \ref{lem:Gluck}, with $(b,a)=(0,4)$, yields
				\[
					\|II_1\|_{L^{\frac{2n}{n-3}}(\Si_{v,\lam})}
					\le
					C
					\|
					u_{\lam} ^{\frac{8}{n-3}}
					(u_{\lam} -u)
					\|_{L^{\frac{2(n+1)}{n+5}}(\Si_{u,\lam})}
					\le
					C
					\|u\|_{L^{\frac{2(n+1)}{n-3}}(\Si_{u,\lam}^{\lam})}^{\frac{8}{n-3}}
					\|u_{\lam} -u\|_{L^{\frac{2(n+1)}{n-3}}(\Si_{u,\lam})}.
			\]
		Similarly, extending the integrand by zero and applying
Lemma \ref{HLS} in $\Rn$, we obtain
\[
\|II_2\|_{L^{\frac{2n}{n-3}}(\Si_{v,\lam})} \le C \| v_\lam^{\frac{6}{n-3}} (v_\lam-v) \|_{L^{\frac{2n}{n+3}}(\Si_{v,\lam})} \le C \|v\|_{L^{\frac{2n}{n-3}}(\Si_{v,\lam}^{\lam})}^{\frac{6}{n-3}} \|v_\lam-v\|_{L^{\frac{2n}{n-3}}(\Si_{v,\lam})}.
\]
			Consequently,
				\begin{align}
					\|v_\lam-v\|_{L^{\frac{2n}{n-3}}(\Si_{v,\lam})}
					\le&\,
					C_2
					\|u\|_{L^{\frac{2(n+1)}{n-3}}(\cB_\lam^+(X_0))}^{\frac{8}{n-3}}
					\|u_{\lam} -u\|_{L^{\frac{2(n+1)}{n-3}}(\Si_{u,\lam})}
					\nonumber\\
					&+
					C_2
					\|v\|_{L^{\frac{2n}{n-3}}(B_\lam(x_0))}^{\frac{6}{n-3}}
					\|v_\lam-v\|_{L^{\frac{2n}{n-3}}(\Si_{v,\lam})}.
					\label{lem:Start formula a2}
				\end{align}

				Since
				$u\in L^{\frac{2(n+1)}{n-3}}(\R ^{n+1}_+)$ and
				$v\in L^{\frac{2n}{n-3}}(\Rn)$,
				we may choose $\lam_0>0$ sufficiently small such that, for every
				$\lam\in (0,\lam_0)$,
\[
C_* \|u\|_{L^{\frac{2(n+1)}{n-3}}(\cB_\lam^+(X_0))}^{\frac{8}{n-3}} <\frac13, \quad C_* \|v\|_{L^{\frac{2n}{n-3}}(B_\lam(x_0))}^{\frac{2i}{n-3}} <\frac13,\quad \text{ and }\quad C_* \|v\|_{L^{\frac{2n}{n-3}}(B_\lam(x_0))}^{\frac{6}{n-3}} <\frac13,
\]
				where $C_*=\max\{C_1,C_2\}$. Then it follows from
				\eqref{lem:Start formula a1} and \eqref{lem:Start formula a2} that
\[
\|u_{\lam} -u\|_{L^{\frac{2(n+1)}{n-3}}(\Si_{u,\lam})} \le \|v_\lam-v\|_{L^{\frac{2n}{n-3}}(\Si_{v,\lam})}~ \text{ and }~ \|v_\lam-v\|_{L^{\frac{2n}{n-3}}(\Si_{v,\lam})} \le \frac12 \|u_{\lam} -u\|_{L^{\frac{2(n+1)}{n-3}}(\Si_{u,\lam})}.
\]
				Hence both norms vanish. Therefore
				$|\Si_{u,\lam}|=|\Si_{v,\lam}|=0$,
			which proves the desired inequalities \eqref{comparison-inequalities}.
			\end{proof}
We now define the maximal radius for which the comparison inequalities
\eqref{comparison-inequalities} hold.
For
$X_0=(x_0,0)\in\partial\R^{n+1}_+$, set	\be\label{lambdaX0}
			\bar{\lam }(X_0)=\sup\{\mu>0: u_{\lam} \le u  \text{ a.e. in } \Si_{X_0,\lam}
			\ \text{and}\
			v_\lam\le v  \text{ a.e. in } \Si_{x_0,\lam},
			\ \forall\,\lam\in (0,\mu)\}.
			\ee
The next lemma characterizes the limiting position when
$\bar\lam(X_0)<\infty$.
			\begin{lem}\label{lem:Limit}
				Fix $X_0=(x_0,0)\in\pa \R ^{n+1}_+$. If
				$\bar\lam:=\bar\lam(X_0)<+\infty$, then
\[
u_{\bar{\lam }}(X)= u(X), \quad\forall\, X\in\Si_{X_0,\bar\lam},\quad \text{ and }\quad v_{\bar{\lam }}(x)= v(x), \quad\forall\, x\in\Si_{x_0,\bar\lam}.
\]
			\end{lem}
			\begin{proof}
				By the definition of $\bar\lam$ and continuity with respect to $\lam$,
\[
u_{\bar\lam}(X)\le u(X), \quad\forall\, X\in\Si_{X_0,\bar\lam},\quad \text{ and }\quad v_{\bar\lam}(x)\le v(x), \quad\forall\, x\in\Si_{x_0,\bar\lam}.
\]
				Suppose, by contradiction, that at least one of these two inequalities is not an
				identity. Then Lemma \ref{lem:Difference}, together with the sign properties of
				the kernels, implies the strict inequalities
\[
u_{\bar\lam}(X)<u(X), \quad\forall\, X\in\Si_{X_0,\bar\lam},\quad \text{ and }\quad v_{\bar\lam}(x)<v(x), \quad\forall\, x\in\Si_{x_0,\bar\lam}.
\]

				Fix $R\in (\bar\lam+1,+\infty)$ and $\delta\in (0,1)$. By continuity and compactness, there
				exist constants $C_1=C_1(R,\delta)>0$ and $C_2=C_2(R,\delta)>0$ such that
\[
u(X)-u_{\bar\lam}(X)\ge C_1, \quad\forall\, X\in \Si_{X_0,\bar\lam+\delta}\cap \overline{\cB_R^+(X_0)},
\]
				and
\[
v(x)-v_{\bar\lam}(x)\ge C_2, \quad\forall\, x\in \Si_{x_0,\bar\lam+\delta}\cap \overline{B_R(x_0)}.
\]
				By continuity with respect to $\lam$, there exists
				$\varepsilon=\varepsilon(R,\delta)\in (0,\frac{\delta}{2})$ such that, for every
				$\lam\in(\bar\lam,\bar\lam+\varepsilon)$,
\[
u(X)-u_{\lam} (X)\ge \frac{C_1}{2}, \quad\forall\, X\in \Si_{X_0,\bar\lam+\delta}\cap \overline{\cB_R^+(X_0)},
\]
				and
\[
v(x)-v_\lam(x)\ge \frac{C_2}{2}, \quad\forall\, x\in \Si_{x_0,\bar\lam+\delta}\cap \overline{B_R(x_0)}.
\]
				Consequently,
\[
\Si_{u,\lam} \subset (\R ^{n+1}_+\setminus \cB_R^+(X_0)) \cup (\Si_{X_0,\lam}\setminus\Si_{X_0,\bar\lam+\delta}) :=\Omega_{\lam,R},
\]
				and
\[
\Si_{v,\lam} \subset (\Rn\setminus B_R(x_0)) \cup (\Si_{x_0,\lam}\setminus\Si_{x_0,\bar\lam+\delta}) :=\Omega^n_{\lam,R}.
\]
				
                Under the inversion $\cI_{X_0,\lam}$, we have
\[
\Omega_{\lam,R}^{\lam} \subset \cB_{\lam^2/R}^+(X_0) \cup ( \cB_\lam^+(X_0) \setminus \cB_{\lam^2/(\bar\lam+\delta)}^+(X_0) ),
\]
				and
\[
(\Omega^n_{\lam,R})^{\lam} \subset B_{\lam^2/R}(x_0) \cup ( B_\lam(x_0) \setminus B_{\lam^2/(\bar\lam+\delta)}(x_0) ).
\]
				By choosing $R$ sufficiently large, then $\delta>0$ sufficiently small,
and finally decreasing $\varepsilon$ so that
$\varepsilon\in (0,\frac{\delta}{2})$, the measures of
$\Omega_{\lam,R}^{\lam}$ and $(\Omega^n_{\lam,R})^\lam$
can be made arbitrarily small, uniformly for
$\lam\in(\bar\lam,\bar\lam+\varepsilon)$.

			Since
$\Si_{u,\lam}^{\lam}\subset\Omega_{\lam,R}^{\lam}$ and 
$\Si_{v,\lam}^{\lam}\subset(\Omega^n_{\lam,R})^\lam$,
the estimates \eqref{lem:Start formula a1} and
\eqref{lem:Start formula a2}  yield
				\begin{align*}
					\|u_{\lam} -u\|_{L^{\frac{2(n+1)}{n-3}}(\Si_{u,\lam})}
					\le&\,
					C
					\|u\|_{L^{\frac{2(n+1)}{n-3}}(\Omega_{\lam,R}^{\lam})}^{\frac{8}{n-3}}
					\|u_{\lam} -u\|_{L^{\frac{2(n+1)}{n-3}}(\Si_{u,\lam})}\\
					&+
					C
					\|v\|_{L^{\frac{2n}{n-3}}((\Omega^n_{\lam,R})^\lam)}^{\frac{2i}{n-3}}
					\|v_\lam-v\|_{L^{\frac{2n}{n-3}}(\Si_{v,\lam})}\\
					&+
					C
					\|v\|_{L^{\frac{2n}{n-3}}((\Omega^n_{\lam,R})^\lam)}^{\frac{6}{n-3}}
					\|v_\lam-v\|_{L^{\frac{2n}{n-3}}(\Si_{v,\lam})},
				\end{align*}
				and
				\begin{align*}
					\|v_\lam-v\|_{L^{\frac{2n}{n-3}}(\Si_{v,\lam})}
					\le&\,
					C
					\|u\|_{L^{\frac{2(n+1)}{n-3}}(\Omega_{\lam,R}^{\lam})}^{\frac{8}{n-3}}
					\|u_{\lam} -u\|_{L^{\frac{2(n+1)}{n-3}}(\Si_{u,\lam})}\\
					&+
					C
					\|v\|_{L^{\frac{2n}{n-3}}((\Omega^n_{\lam,R})^\lam)}^{\frac{6}{n-3}}
					\|v_\lam-v\|_{L^{\frac{2n}{n-3}}(\Si_{v,\lam})}.
				\end{align*}
				Since
$u\in L^{\frac{2(n+1)}{n-3}}(\R^{n+1}_+)$ and 
$v\in L^{\frac{2n}{n-3}}(\Rn)$,
the absolute continuity of the Lebesgue integral allows us to choose
$R$, $\delta$, and $\varepsilon$ as above so that, uniformly for
$\lam\in(\bar\lam,\bar\lam+\varepsilon)$,
\[
C \|u\|_{L^{\frac{2(n+1)}{n-3}}(\Omega_{\lam,R}^{\lam})}^{\frac{8}{n-3}} <\frac13,\quad C \|v\|_{L^{\frac{2n}{n-3}}((\Omega^n_{\lam,R})^\lam)}^{\frac{2i}{n-3}} <\frac13, \quad \text{ and }\quad C \|v\|_{L^{\frac{2n}{n-3}}((\Omega^n_{\lam,R})^\lam)}^{\frac{6}{n-3}} <\frac13.
\]
				It follows that, for every $\lam\in(\bar\lam,\bar\lam+\varepsilon)$,
\[
\|u_{\lam} -u\|_{L^{\frac{2(n+1)}{n-3}}(\Si_{u,\lam})}=0\quad \text{ and }\quad \|v_\lam-v\|_{L^{\frac{2n}{n-3}}(\Si_{v,\lam})}=0.
\]
				Thus
\[
u_{\lam} \le u \quad\text{ a.e. in }\Si_{X_0,\lam} \quad \text{ and }\quad v_\lam\le v \quad\text{ a.e. in }\Si_{x_0,\lam},
\]
				for all $\lam\in(\bar\lam,\bar\lam+\varepsilon)$, contradicting the
				definition of $\bar\lam$. Therefore the inequalities at
				$\lam=\bar\lam$ must be identities.
			\end{proof}

            We are now ready to prove Theorem \ref{thm:main0}.
            \begin{proof}[Proof of Theorem \ref{thm:main0}]
				We first show that the moving sphere stops at a finite radius for at least
one boundary center. Suppose, to the contrary, that
				\[\bar\lam ((y,0))=+\infty,
				\quad
				\forall\,
				(y,0)\in\pa \R^{n+1}_+.
				\]
			Then, by the definition of $\bar\lam$ in \eqref{lambdaX0},
\[
v_{(y,0),\lam}(x)\le v(x), \quad \forall\,\lam>0, \, y\in\Rn,\, x\in\Rn\setminus \overline{B_\lam(y)} .
\]
			Lemma \ref{lem:Technique1}
therefore implies that $v$ is a positive constant, contradicting
				$\int_{\Rn}v^{\frac{2n}{n-3}}<+\infty$. Hence there exists $X_0\in\pa \R ^{n+1}_+$ such that
				$\bar\lam(X_0)<+\infty$.

				We next claim that
\[
\bar\lam((y,0))<+\infty, \quad \forall\, (y,0)\in\pa \R ^{n+1}_+ .
\]
				Indeed, by the definition of $\bar\lam((y,0))$, for any
				$\lam\in (0,\bar\lam((y,0)))$, we have
\[
u_{(y,0),\lam}(X) = \Big(\frac{\lam}{|X-(y,0)|}\Big)^{n-3} u\Big((y,0)+\frac{\lam^2(X-(y,0))}{|X-(y,0)|^2}\Big) \le u(X)
\]
				for $X\in\Si_{(y,0),\lam}$. Letting $|X|\to+\infty$, we obtain
				\be \label{Removable cond}
				\lam^{n-3}u(y,0)
				\le
				\liminf_{|X|\to+\infty}|X|^{n-3}u(X).
				\ee
				On the other hand, Lemma \ref{lem:Limit} applied at $X_0$ yields $u_{\bar\lam(X_0)}=u$ in $\Si_{X_0,\bar\lam(X_0)}$. Therefore,
				\be \label{Asy_decay_u}
				\bar\lam(X_0)^{n-3}u(X_0)
				=
				\lim_{|X|\to+\infty}|X|^{n-3}u(X).
				\ee
				Combining \eqref{Removable cond} and \eqref{Asy_decay_u}, and then letting
				$\lam\uparrow\bar\lam((y,0))$, gives
\[
\bar\lam((y,0)) \le \bar\lam(X_0) \Big(\frac{u(X_0)}{u(y,0)}\Big)^{\frac1{n-3}} <+\infty .
\]

				Applying Lemma \ref{lem:Technique2} to $v$, we conclude that
\[
v(x)=\Big(\frac{a}{\var^2+|x-x_0|^2}\Big)^{\frac{n-3}{2}}
\]
				for some $a> 0$, $\var>0$, and  $x_0\in\Rn$.

				Following the argument of Li and Zhu \cite{LZ1995}, we now identify the critical
				radius. For every $(y,0)\in\pa \R ^{n+1}_+$, we claim that
				\be\label{Bar-lambda}
				\bar{\lam}(y,0)=\sqrt{\var^2+|y-x_0|^2}.
				\ee
				Indeed, by Lemma \ref{lem:Limit}, $v_{\bar\lam((y,0))}=v$ in $\Si_{y,\bar\lam((y,0))}$. Hence, for
				$x\in\Rn\setminus \overline{B_{\bar\lam((y,0))}(y)}$,
\[
\Big(\frac{\bar{\lam}}{|x-y|}\Big)^{n-3}\Big(\frac{a}{\var^2+|y-x_0+\frac{\bar{\lam}^2(x-y)}{|x-y|^2}|}\Big)^{\frac{n-3}{2}}=\Big(\frac{a}{\var^2+|x-x_0|^2}\Big)^{\frac{n-3}{2}},
\]
				which is equivalent to
				\begin{align*}
					&\,\frac{|x-y|^2}{\bar{\lam}^2}\Big[\var^2+|y-x_0|^2+\frac{\bar{\lam}^4}{|x-y|^2}+\frac{2\bar{\lam}^2(x-y)\cdot(y-x_0)}{|x-y|^2}\Big]\\
					=&\,\var^2+|x-y|^2+|y-x_0|^2+2(x-y)(y-x_0).
				\end{align*}
				Comparing the coefficients of $|x-y|^2$ yields
				\eqref{Bar-lambda}.

				We now choose a conformal map from $\R ^{n+1}_+$ to $\mathbb B^{n+1}$
				so that the transformed solution has constant boundary value. Define
\[
u_{\var,x_0}(x,t) := \var^{\frac{n-3}{2}}u(\var x+x_0,\var t).
\]
				Then
				\be \label{u-u'}
				u_{\var,x_0}(x,0)
				=
				\var^{\frac{n-3}{2}}v(\var x+x_0)
				=
				\Big(
				\frac{a}{\var(1+|x|^2)}
				\Big)^{\frac{n-3}{2}} .
				\ee

				Let $\mathcal S:\R ^{n+1}_+\to\mathbb B^{n+1}$ be the conformal M\"obius transformation defined in \eqref{conformalmapS}, and set
\[
U(\xi)=\Big(\frac{2}{|\xi+e_{n+1}|^2}\Big)^{\frac{n-3}{2}}u_{\var,x_0}\circ \mathcal S^{-1}(\xi).
\]
				By \eqref{u-u'}, $U$ is constant on $\Sn$. Moreover, the critical
boundary norm is invariant under the preceding dilation and conformal
transformation. Therefore
\[
(U|_{\Sn })^{\frac{2n}{n-3}}|\Sn | = \int_{\Sn }U^{\frac{2n}{n-3}}\,\d \sigma = \int_{\Rn}u^{\frac{2n}{n-3}}(x,0)\,\d x = \al ^{\frac{2n}{n-3}}|\Sn |.
\]
				Hence $U|_{\Sn}=\al $.
	By the conformal covariance of the Paneitz operator and the boundary
operators $\sB_i^3$, the function $U$ satisfies the boundary value problem
\eqref{BVP_U-1}.	It remains to show that $U$ is radial.

By Lemma \ref{lem:Limit} and
				\eqref{Bar-lambda}, for any $Y=(y,0)\in\pa \R ^{n+1}_+$, 
\[
\Big( \frac{\sqrt{\var^2+|y-x_0|^2}}{|X-Y|} \Big)^{n-3} u\Big( Y+\frac{(\var^2+|y-x_0|^2)(X-Y)}{|X-Y|^2} \Big) = u(X). \]After the rescaling defining $u_{\var,x_0}$, this becomes \[ \Big( \frac{\sqrt{1+|y|^2}}{|X-(y,0)|} \Big)^{n-3} u_{\var,x_0}\Big( (y,0)+ \frac{(1+|y|^2)(X-(y,0))}{|X-(y,0)|^2} \Big) = u_{\var,x_0}(X).
\]
				Let $X=\mathcal S^{-1}(\xi)$, with $\xi=(\xi',\xi_{n+1})$. A direct computation gives
\[
\mathcal S\Big( (y,0)+ \frac{(1+|y|^2)(X-(y,0))}{|X-(y,0)|^2} \Big) = \xi - 2\frac{y\cdot\xi'+\xi_{n+1}}{1+|y|^2}(y,1).
\]
				Thus $\mathcal S$ conjugates the above Kelvin inversion to the Euclidean reflection
				across the hyperplane
\begin{align}\label{Plane}
y\cdot\xi'+\xi_{n+1}=0, \quad y\in\Rn.
\end{align}
				Furthermore, since $-\mathbf e_{n+1}$ is fixed by this inversion, the standard
				distance identity for inversion gives
\[
\Big| (y,0)+ \frac{(1+|y|^2)(X-(y,0))}{|X-(y,0)|^2} + \mathbf e_{n+1} \Big| = \frac{\sqrt{1+|y|^2}}{|X-(y,0)|} |X+\mathbf e_{n+1}|.
\]
				Consequently,
\[
\Big| \mathcal S \Big( (y,0)+ \frac{(1+|y|^2)(X-(y,0))}{|X-(y,0)|^2} \Big) +\mathbf e_{n+1} \Big|^2 = \frac{|X-(y,0)|^2}{1+|y|^2} |\xi+\mathbf e_{n+1}|^2.
\]
				Combining this identity with the Kelvin equality for $u_{\var,x_0}$, we obtain
\[
U\Big( \xi - 2\frac{y\cdot\xi'+\xi_{n+1}}{1+|y|^2}(y,1) \Big) = U(\xi).
\]
				Therefore $U$ is invariant under reflection across every hyperplane \eqref{Plane}.
				The corresponding normal vectors are $(y,1)$. As $y$ varies over
				$\Rn$, these vectors determine all directions with nonzero last
				component, and the remaining directions are obtained by taking limits. Since
				$U$ is continuous, $U$ is invariant under reflection across every hyperplane
				through the origin. Hence $U$ is radial:
$U(\xi)=U(|\xi|)$.
				The proof is complete.
			\end{proof}

			\begin{rem}\label{rem:Key}
\textbf{A comparison between the second-order and fourth-order ODEs.}
We conclude this section with a brief comparison between the second-order
model and its fourth-order counterpart, which also motivates the analysis
in the next section.

For the second-order equation studied by Li and Zhu \cite{LZ1995}, see
\eqref{LZ1}, the associated radial problem on $\B^{n+1}$ takes the form
\be\label{Sec-ODE}
\left\{
\begin{aligned}
&U''+\frac{n}{r}U'(r)
=-\frac{n^2-1}{4}U^{\frac{n+3}{n-1}}
&&\text{in }\,\B^{n+1},\\
&U'(0)=0,\qquad
\sB^1_1(U)=cU^{\frac{n+1}{n-1}}
&&\text{on }\,\Sn.
\end{aligned}
\right.
\ee
Thus, once the mean curvature $c$ is prescribed, the second-order equation
is supplied with two boundary conditions, and one expects uniqueness.
Moreover, for each admissible $c$, the corresponding solution is an
Aubin--Talenti bubble.

The fourth-order case is substantially different. If one only prescribes
the mean curvature $c_1$, then the radial problem becomes
\be\label{Fou-ODE}
\left\{
\begin{aligned}
&U^{''''}+\frac{2n}{r}U'''
+\frac{n(n-2)}{r^2}U''
-\frac{n(n-2)}{r^3}U'
=\kappa U^{\frac{n+5}{n-3}}
&&\text{ in }\,\B^{n+1},\\
&U'(0)=U'''(0)=0,\quad
\sB^3_1(U)=c_1U^{\frac{n-1}{n-3}}
&&\text{ on }\,\Sn.
\end{aligned}
\right.
\ee
This fourth-order problem imposes only three boundary conditions and hence
is not expected to be uniquely solvable. The missing condition should be
supplied by an admissible prescribed $T$-curvature. Since the pair
$(\sB_1^3,\sB_2^3)$ fails the complementing condition in Lemma
\ref{lem:complementing-condition}, the natural choice is the third-order
$T$-curvature $c_3$.

Moreover, for every $c_1\in\R$, there exists a unique $\lam>0$ such that
the radial Aubin--Talenti bubble $U_{0,\lam}$ satisfies
\eqref{Fou-ODE}. Indeed, the boundary condition is equivalent to
$c_1=\frac{n-3}{4}(\lam-\lam^{-1})$.
The corresponding third-order curvature $c_3$ is then determined by
$\lam$, and hence by $c_1$, through \eqref{eq:Bubble}. Consequently, an
Aubin--Talenti bubble does not in general solve the full boundary value
problem \eqref{BVP_U-1} for an independently prescribed pair $(c_1,c_3)$,
unless the two curvature constants satisfy the corresponding spherical-cap
constraint.

When $c_1=0$, one may prescribe one additional datum, either the boundary
value $\al:=U|_{\Sn}$
or the third-order $T$-curvature $c_3$. This suggests a nontrivial
correspondence between the boundary value $\al$ and the third-order
$T$-curvature constant $c_3$.
\end{rem}

			\section{Construction of radial solutions}\label{sec:4}
			In this section, we construct radial solutions of the boundary value problem
\eqref{BVP_U-1} arising from the conformal reduction in Theorem
\ref{thm:main0}. We work on the unit ball $\B^{n+1}$, where the conformal
boundary operators are given explicitly by \eqref{Boundaryoperators-Ball}.
Our basic comparison functions are the radial Aubin--Talenti bubbles
$U_{0,\lam}$ introduced in \eqref{bubble}.
			
            A direct computation using \eqref{Boundaryoperators-Ball} shows that $U_{0,\lam}$  satisfies
			\be\label{eq:Bubble}
			\left\{\begin{aligned}
				&\Delta^2 U_{0,\lam} =\kap U_{0,\lam} ^{p^*}   &&\text{ in }\, \B^{n+1},\\
				&U_{0,\lam} =\Big(
				\frac{2\lam}{1+\lam^2}\Big)^{\frac{n-3}{2}}
				&&\text{ on }\,  \Sn,\\
				& \sB_1^3(U_{0,\lam} )=\frac{n-3}{2}
				\frac{\lam^2-1}{\lam^2+1}
				U_{0,\lam} (1),&&\text{ on }\,  \Sn,\\ &\sB_2^3(U_{0,\lam} )=\frac{(n-1)(n-3)}{4} \Big[1+\Big(\frac{\lam^2-1}{\lam^2+1}\Big)^2\Big] U_{0,\lam} (1) &&\text{ on }\,  \Sn,\\ &\sB_3^3(U_{0,\lam} )=\frac{(n^2-1)(n-3)}{8} \frac{\lam^2-1}{\lam^2+1}\Big[3-\Big(\frac{\lam^2-1}{\lam^2+1}\Big)^2\Big] U_{0,\lam} (1) &&\text{ on }\,  \Sn.
			\end{aligned}\right.
			\ee
			These identities will be used repeatedly to choose bubbles with prescribed mean curvature and to build sub- and supersolutions.

		We begin with a radial comparison lemma for functions satisfying homogeneous
$\sB_0^3$- and $\sB_1^3$-boundary  data.
			\begin{lem}\label{lem:Comparison}
				Let $Z=Z(r)$ be a $C^4(\overline{\B^{n+1}}) $ radial function  satisfying
\[
\left\{ \begin{aligned} &\Delta^2 Z\geq 0 && \text{in }\quad \B^{n+1},\\ &Z=0,\quad \sB_1^3 Z=0 && \text{on } \quad\Sn . \end{aligned} \right.
\]
				Then
				$
				Z_{rr}(1)\geq 0$,
				$\sB_3^3 Z\leq 0$ on $\Sn
				$.
				Moreover, if $\Delta^2 Z>0$ in $\B^{n+1}$, then one also has
				$
				Z_{rr}(1)>0$,
				$\sB_3^3 Z<0$ on $\Sn$.
			\end{lem}
			\begin{proof} By the radial formula \eqref{Boundaryoperators-Ball} for the third boundary operator,
				\be\label{B^3_3Z}
				\sB^3_3Z
				=
				-\pa _r(\Delta Z)
				-
				\frac{n-3}{2}Z_{rr}
				+
				\frac{n-3}{2}Z_r
				+
				\frac{(n^2-1)(n-3)}{4}Z
				\quad\text{ on }\,\Sn,
				\ee
				and 	the boundary conditions  become $Z(1)=Z'(1)=0$. It remains to estimate the terms on the right-hand side of \eqref{B^3_3Z}.

				Set
				$
				\zeta(r)=\Delta Z(r)$. The radial Laplacian formula gives that
				$(r^n\zeta'(r))'=r^n\Delta \zeta(r)=r^n\Delta^2Z(r)$ for
				$0<r\le1$.
				Hence,
\[
r^n\zeta'(r)=\int_0^r s^n\Delta^2Z(s)\,\, \d s\ge0,\quad \forall\, 0<r\leq 1.
\]
				Therefore,
				\be\label{zeta-nondecresing}
				\zeta'(r)\ge0
				\quad\text{ for }\, 0<r\le1,
				\ee
				and consequently $\zeta$ is nondecreasing on $(0,1]$.

				Next, applying the radial Laplacian formula to $Z$, we obtain
				$(r^nZ'(r))'=r^n\Delta Z(r)=r^n\zeta(r)$.
				Since $r^nZ'(r)\to0$ as $r\to0$, integration over $(0,1)$ and the boundary condition $Z'(1)=0$ yield
				\be\label{Zprime-zero-average}
				Z'(1)=\int_0^1 r^n\zeta(r)\,\, \d r=0.
				\ee
				Since $\zeta$ is nondecreasing, \eqref{Zprime-zero-average} implies
				\be\label{zeta(1)}
				\zeta(1)\ge0.
				\ee

				Finally, from $\zeta(1)=\Delta Z(1)=Z_{rr}(1)+nZ_r(1)$ and $Z_r(1)=0$, we have $Z_{rr}(1)=\zeta(1)$. Substituting $Z(1)=Z_r(1)=0$ and $Z_{rr}(1)=\zeta(1)$ into \eqref{B^3_3Z}, we obtain
\[
\sB ^3_3Z = -\zeta'(1)-\frac{n-3}{2}\zeta(1)\leq 0 \quad\text{ on }\, \Sn, \] where the last inequality follows from \eqref{zeta-nondecresing} and \eqref{zeta(1)}. This proves the lemma. 
\end{proof}

For the construction below, we shall also use the Green functions on the
unit ball. Following the notation of Chen--Zhang \cite{CZ202406}, we denote
by $\bar G^{(i,j)}$ the Green function on $\B^{n+1}$ associated with the
boundary operator pair $(\sB_i^3,\sB_j^3)$; see also Remark
\ref{rem:Greenfuncions-in-Ball}.
\begin{lem}\label{lem:subsupsolution} 
Let $c_1\in\R$, and assume that $\al$ satisfies
\eqref{alphaAssumption}. Then there exists a positive smooth radial solution $U\in C^\infty(\overline{\B^{n+1}})$ satisfying \[\left\{\begin{aligned} &\Delta^2U=\kap U^{p^*} &&\text{ in }\,\B^{n+1},\\ &U=\al,~ \sB _1^3U=c_1U^{p_1^*}& &\text{ on }\,\Sn. \end{aligned}\right.\] Moreover, we obtain the following estimate \[ \al \Big[ 1+ \frac12 \Big( \frac{n-3}{2} - c_1\al ^{\frac{2}{n-3}} \Big) (1-r^2) \Big] \leq U(r)\leq \frac{\al }{U_{0,\lam}  (1)}\,U_{0,\lam} (r),
\]
				where $\lam$ is determined by \eqref{lamchoice}.  If
				$c_3
				:=
				\frac{\sB_3^3U|_{\Sn}}{\al ^{p_3^*}}$,
				then
				$\sB_3^3U=c_3U^{p_3^*}$
				on $\S^n$,
				and $c_3$ satisfies the estimates in \eqref{c3bdd}.
			\end{lem}

			\begin{proof}
				We first deduce from  the assumption \eqref{alphaAssumption} that
				\be\label{eq:lambda-admissible}
				1-\Big(\frac{2c_1\al^{\frac{2}{n-3}}}{n-3}\Big)^2>\al^{\frac{4}{n-3}}>0.
				\ee
				In particular,
				$|
				\frac{2c_1\al ^{\frac{2}{n-3}}}{n-3}|<1$.
				Hence there exists $\lam>0$ such that
				\be\label{lamchoice}
				\frac{\lam ^2-1}{\lam ^2+1}
				=
				\frac{2c_1\al^{\frac{2}{n-3}}}{n-3}.
				\ee
				For this choice of $\lam$, it follows from \eqref{eq:Bubble} and \eqref{eq:lambda-admissible} that
\[
U_{0,\lam}  ^{\frac{2}{n-3}}(1) = \frac{2\lam }{1+\lam ^2} = \Big[ 1- \Big( \frac{2c_1\al^{\frac{2}{n-3}}}{n-3} \Big)^2 \Big]^{1/2}> \al^{2/(n-3)}.
\]
				Therefore,  $U_{0,\lam }(1)>\al$. We next construct a supersolution and a subsolution.

				Define
				\be\label{supersolution}
				\overline U:
				=
				\frac{\al }{U_{0,\lam}  (1)}\,U_{0,\lam} .\ee
				Then
				$\overline U=\al$  on $\Sn$.  Moreover,
				by \eqref{eq:Bubble} and \eqref{lamchoice},
\[
\sB _1^3\overline U = \frac{\al }{U_{0,\lam}  (1)} \sB _1^3U_{0,\lam}  = \al \frac{n-3}{2} \frac{\lam ^2-1}{\lam ^2+1} = c_1\al ^{\frac{n-1}{n-3}} = c_1\overline U^{p_1^*} \quad\text{ on }\,\Sn.
\]
				Since $U_{0,\lam}  (1)>\al$ and $p^*>1$, it follows again from \eqref{eq:Bubble} that
\[
\Delta^2\overline U = \frac{\al }{U_{0,\lam}  (1)}\kap U_{0,\lam}  ^{p^*} > \kap \overline U^{p^*} \quad\text{ in }\,\B^{n+1}.
\]
				Thus $\overline U$ is a strict supersolution for the boundary value problem with prescribed $\sB_0^3$- and $\sB_1^3$-boundary data; namely,
				\be\label{Ubar_equation}
				\left\{\begin{aligned}	&\Delta^2\overline U>
					\kap \overline U^{p^*}&&\text{ in }\, \B^{n+1},\\
					&	\overline U=\al ,~
					\sB _1^3\overline U
					=
					c_1\overline U^{p_1^*}
					&&\text{ on }\,\Sn.
				\end{aligned}\right.
				\ee

				Now define
				\be\label{subsolution}
				\underline U(r)
				=
				\al +
				\frac12
				\Big(
				\frac{n-3}{2}\al
				-
				c_1\al ^{p_1^*}
				\Big)
				(1-r^2),\quad 0\leq r\leq 1.
				\ee
				By \eqref{alphaAssumption}, we have  $\underline U>0$ on $\overline{\B^{n+1}}$. Moreover, a direct calculation using \eqref{Boundaryoperators-Ball} gives
\[
\left\{\begin{aligned} & \Delta^2\underline U=0&& \text{ in }\,\B^{n+1},\\ & \underline U=\al ,~ \sB _1^3\underline U = c_1\underline U^{p_1^*} &&\text{ on }\,\,\Sn. \end{aligned}\right.
\]
				We claim that $\underline U< \overline U$  in   $\B^{n+1}$.
				Indeed, set
				$
				x=\frac{\lam ^2+r^2}{\lam ^2+1}
				$,
				then
				$\overline U(r)=\al  x^{-\frac{n-3}{2}}$.
				Since the function $x\mapsto x^{-\frac{n-3}{2}}$ is strictly convex on
				$(0,+\infty)$, we have, for $0\le r<1$,
\[
x^{-\frac{n-3}{2}} \ge 1-\frac{n-3}{2}(x-1).
\]
				Therefore,
\[
\overline U(r) > \al \Big( 1+ \frac{n-3}{2} \frac{1-r^2}{\lam ^2+1} \Big).
\]
				Using \eqref{lamchoice} we obtain
\[
\overline U(r) > \al \Big[ 1+ \frac12 \Big( \frac{n-3}{2} - c_1\al ^{\frac{2}{n-3}} \Big) (1-r^2) \Big] = \underline U(r).
\]
				This proves the claim.

				Let $\bar G^{(0,1)}$ be the Green kernel for the biharmonic operator in
				$\B^{n+1}$ with homogeneous $\sB_0^3$- and $\sB_1^3$-boundary data. For $f\in C^{\infty}(\overline{\B^{n+1}})$, write
				\[ \bar \cG^{(0,1)}(f)(\xi) := \int_{\B^{n+1}} \bar G^{(0,1)}(\xi,\eta)f(\eta)\,\d\eta, \quad \xi\in\B^{n+1}. \] By Boggio's formula (see \cite{B1905} or
				\cite[Lemma 2.27]{GGS2010}), one has
				\be\label{positiveGreen}
				\bar G^{(0,1)}(\xi,\eta)> 0 \quad \text{ for  }\, \xi,\eta\in \B^{n+1}, \ \xi\neq \eta.
				\ee
				Set
				$U_0=\underline U$,
				and define inductively
				\be\label{Uk}
				U_{k+1}
				=
				\underline U+\bar \cG^{(0,1)}(\kap  U_k^{p^*}),\quad k=0,1,2,\ldots.
				\ee
				Since the Boggio Green kernel is invariant under rotations,  $\bar G^{(0,1)}$ maps radial functions to radial functions. Thus, by induction, each $U_k$ is radial.
				Using  \eqref{positiveGreen} and the fact that  $t\mapsto t^{p^{*}}$ is increasing, we obtain
				$U_1> U_0$  in $\B^{n+1}$.
				If $U_k> U_{k-1}$ in $\B^{n+1}$ for some $k\geq 1$, then
\[
U_{k+1}-U_k = \bar \cG^{(0,1)}(\kap (U_k^{p^{*}}-U_{k-1}^{p^{*}})) >0\quad \text{ in }\, \B^{n+1}.
\]
				Thus $\{U_k\}$ is monotone increasing. Moreover, if $U_k\le\overline U$ in $\B^{n+1}$, then
\[
U_{k+1} \le \underline U+\bar \cG^{(0,1)}(\kap \overline U^{p^*}) \le \underline U+\bar \cG^{(0,1)}(\Delta^2\overline U) = \overline U\quad \text{ in }\, \B^{n+1}.
\]
				Consequently,
\[
\underline U\le U_k\le\overline U \quad\text{ in }\,\B^{n+1}, \quad \forall\, k=0,1,2,\ldots
\]
				Therefore, $U_k$ converges pointwise to a radial positive function
				$U_\infty$ satisfying
				\be\label{Ulimit}
				\underline U\le U_\infty\le\overline U
				\quad\text{ in }\,\B^{n+1}.
				\ee
				By the dominated
				convergence theorem, we can pass to the limit in \eqref{Uk} and obtain
				\be\label{Ulimit-expression}
				U_{\infty}=\underline U+\bar \cG^{(0,1)}(\kap  U_{\infty}^{p^*}).
				\ee
				Hence
				\be\label{Ulimt_equation}
				\left\{\begin{aligned}
					&\Delta^2 U_{\infty}=\kap  U_{\infty}^{p^*} &&\text{ in }\,\B^{n+1},\\
					&U_{\infty}=\al ,~
					\sB _1^3 U_{\infty}
					=
					c_1 U_{\infty}^{p_1^*}
					&&\text{ on }\,\Sn.
				\end{aligned}\right.
				\ee
				By \eqref{Ulimit}, $U_\infty$ is bounded and positive. Elliptic regularity and a standard bootstrap argument then give
				$U_{\infty}\in C^\infty(\overline{\B^{n+1}})$.
				It remains to estimate the resulting $\sB_3^3$-curvature.

				Set $W=\overline U-U_{\infty}$.
				Since $U_{\infty}\le\overline U=(\al /U_{0,\lam}  (1))U_{0,\lam}  $ by \eqref{Ulimit}, we obtain from \eqref{Ubar_equation} and \eqref{Ulimt_equation} that
				\be\label{Weq}
				\kap \frac{\al }{U_{0,\lam}  (1)}U_{0,\lam}  ^{p^{*}}	\geq 	\Delta^2W
				=
				\kap \frac{\al }{U_{0,\lam}  (1)}U_{0,\lam}  ^{p^{*}}
				-
				\kap  U_{\infty}^{p^{*}}
				\ge
				\kap
				\Big[
				\frac{\al }{U_{0,\lam}  (1)}
				-
				\Big(\frac{\al }{U_{0,\lam}  (1)}\Big)^{p^{*}}
				\Big]
				U_{0,\lam}  ^{p^{*}}\quad \text{ in }\, \B^{n+1},
				\ee
				and
				\be\label{Wbdy}
				W=0, ~  \sB ^3_1W=0\quad\text{ on }\,\Sn.
				\ee
				Let $H_\lam $ be the $\sB_0^3$- and $\sB_1^3$-homogeneous biharmonic lifting with
				the same boundary data as $U_{0,\lam}  $, namely, \[
				\left\{
				\begin{aligned}
					&\Delta^2H_\lam=0,
					&&\text{ in }\quad\B^{n+1},\\
					&H_\lam=U_{0,\lam} ,
					~
					\sB_1^3H_\lam=\sB_1^3U_{0,\lam} 
					&&\text{ on }\quad\Sn.
				\end{aligned}
				\right.
				\] Equivalently, since the boundary data are constant,
\[
H_\lam(r) = U_{0,\lam} (1) + \frac12 \Big( \frac{n-3}{2}U_{0,\lam} (1)-\sB_1^3U_{0,\lam} |_{\Sn} \Big)(1-r^2). \]Then by \eqref{eq:Bubble}, \[ \left\{ \begin{aligned} &\Delta^2(U_{0,\lam}  -H_\lam )=\kap U_{0,\lam}  ^{p^{*}} &&\text{ in }\,\B^{n+1},\\ &U_{0,\lam}  -H_\lam=0, ~ \sB_1^3(U_{0,\lam}  -H_\lam)=0 &&\text{ on }\,\Sn. \end{aligned} \right.
\]
				It follows from  \eqref{Weq}--\eqref{Wbdy} and 	Lemma \ref{lem:Comparison}   that
\[
\Big[ \frac{\al }{U_{0,\lam}  (1)} - \Big(\frac{\al }{U_{0,\lam}  (1)}\Big)^{p^{*}} \Big] ( \sB ^3_3H_\lam -\sB ^3_3U_{0,\lam}  ) \le -\sB ^3_3W \le \frac{\al }{U_{0,\lam}  (1)} ( \sB ^3_3H_\lam -\sB ^3_3U_{0,\lam}  )\quad \text{ on }\, \Sn.
\]
				Therefore, by \eqref{eq:Bubble},
				\be\label{B33lowerbdd}
				\sB ^3_3U_{\infty}
				\ge
				\frac{(n^2-1)(n-3)}{4}\al
				\Big[
				1-
				\Big(\frac{\al }{U_{0,\lam}  (1)}\Big)^{p^{*}-1}
				\Big(
				1-\frac{\sB ^3_3U_{0,\lam}  }{\frac{(n^2-1)(n-3)}{4} U_{0,\lam}  (1)}
				\Big)
				\Big]\quad \text{ on }\, \Sn,
				\ee
				and
				\be\label{B33upperbdd}
				\sB ^3_3U_{\infty}\le \frac{(n^2-1)(n-3)}{4} \al \quad \text{ on }\, \Sn.
				\ee
				It remains to derive the estimates of $c_3$.

				From \eqref{eq:Bubble} and \eqref{lamchoice}, we obtain
\[
\frac{\sB ^3_3U_{0,\lam} |_{\Sn} }{\frac{(n^2-1)(n-3)}{4} U_{0,\lam}  (1)} = \frac{c_1\al ^{\frac2{n-3}}}{n-3} \Big[ 3-\frac{4c_1^2\al ^{\frac4{n-3}}}{(n-3)^2} \Big]. \]Combining this with \eqref{Ulimt_equation} and \eqref{B33lowerbdd} gives \[ c_3\al^{p_3^*} =\sB _3^3U_{\infty} \ge \frac{(n^2-1)(n-3)}{4} \al L(\al ,c_1)\quad \text{ on }\, \Sn,
\]
				where $L(\al ,c_1)$ is given by \eqref{L}.
				Finally, 	we conclude from   \eqref{B33upperbdd} that
				\[	\frac{(n^2-1)(n-3)}{4} \al ^{-\frac{6}{n-3}}L(\al ,c_1)
				\le
				c_3
				\le
				\frac{(n^2-1)(n-3)}{4} \al ^{-\frac{6}{n-3}}.
				\]
				This completes the proof.
			\end{proof}

We next turn to the case $i=2$. In this setting, positive radial solutions
with small boundary value can be constructed perturbatively from an explicit
biharmonic function satisfying the prescribed $\sB_0^3$- and
$\sB_2^3$-boundary data.
			\begin{lem}
				\label{lem:ball-B0B2-construction-c3-asymptotic}
				Fix $c_2\le0$. Then there exist $\al _0>0$ and a
				$C^1$-family of smooth positive radial functions
				$\{U_\al \}_{0<\al <\al _0}$
				such that
				\be\label{Construction2}
				\left\{
				\begin{aligned}
					&\Delta^2U_\al =\kappa U_\al ^{p^*}
					&&\text{ in }\,\B^{n+1},\\
					&U_\al =\al ,
					~
					\sB^3_2U_\al =c_2\al ^{p^*_2}
					&&\text{ on }\,\Sn.
				\end{aligned}
				\right.
				\ee
				Moreover,  setting $c_3(\al)
:=
\frac{\sB_3^3U_\al|_{\Sn}}{\al^{p_3^*}}$,
one has
\[
c_3(\al ) = \frac{(n^2-1)(n-3)}{4}\al ^{1-p^*_3} + O(\al ^{p^*-p^*_3}) \to+\infty\quad \text{ as } \,\al \to 0^{+}.
\]
				In particular, for every sufficiently large $c_3>0$, there exists
				$\al \in(0,\al _0)$ such that $c_3=c_3(\al )$, and the corresponding
				$U_\al $ solves
\[
\left\{ \begin{aligned} &\Delta^2U_\al =\kappa U_\al ^{p^*} &&\text{ in }\,\B^{n+1},\\ &\sB^3_2U_\al =c_2U_\al ^{p^*_2}, ~ \sB^3_3U_\al =c_3U_\al ^{p^*_3} &&\text{ on }\,\Sn. \end{aligned} \right.
\]
			\end{lem}

			\begin{proof}
				Let $ \bar G^{(0,2)}$	denote the Green kernel for the biharmonic operator in
				$\B^{n+1}$ with homogeneous $\sB_0^3$- and $\sB_2^3$-boundary data.
				For $f\in C(\overline{\B^{n+1}})$, define
				\[ \bar \cG^{(0,2)}(f)(\xi) := \int_{\B^{n+1}} \bar G^{(0,2)}(\xi,\eta)f(\eta)\,\d\eta, \quad \xi\in\B^{n+1}.\]
				Fix any $\gamma\in(0,1)$. 	We shall use the standard Schauder estimate
				\[ \bar \cG^{(0,2)} :C^{0,\gamma}(\overline{\B^{n+1}})
				\to C^{4,\gamma}(\overline{\B^{n+1}})\]
				and the fact that $\bar \cG^{(0,2)} $ preserves radial symmetry. Set
				\[\mathscr X:=C^{4,\gamma}_{\rm rad}(\overline{\B^{n+1}})\quad
				\text{ and }\quad  	\mathscr Y:=C^{0,\gamma}_{\rm rad}(\overline{\B^{n+1}}).\]
				For $0<\al \ll1$, define
\[
H_\al (r)=\al +b_\al (r^2-1), \quad 0\leq r\leq 1,
\]
				where
				$b_\al
				:=
				\frac{c_2}{2(n-1)} \alpha^{p_2^*}-\frac{n-3}{4} \alpha $.
				Then $H_\al =\al $ on $\Sn$. Since $H_\al $ is radial, using \eqref{Boundaryoperators-Ball} gives
				$\sB^3_2H_\al =c_2\al ^{p^*_2}$
				on $\Sn$.
				Since $c_2\le0$, we have $b_\al <0$. Hence
\[
H_\al (r)=\al +b_\al (r^2-1)\ge \al \quad\text{ in }\,\overline{\B^{n+1}}\quad \text{ and }\quad \|H_\al \|_{\mathscr X}= O(\al ) .
\]
				The purpose of introducing $H_\al $ is to match the prescribed
				$\sB_0^3$- and $\sB_2^3$-boundary data in \eqref{Construction2}. Thus it
				remains to solve for a correction term with homogeneous boundary data. We look
				for a solution in the form
				$U_\al =H_\al +V$.
				Then $V$ should satisfy the fixed point equation
\[
V = \bar{\cG}^{(0,2)} (\kappa(H_\al +V)^{p^*}).
\]

				Let $M>0$ be
				chosen later and define
\[
\mathcal A_\al = \{V\in \mathscr X: \|V\|_{\mathscr X}\le M\al ^{p^*}\}.
\]
				For $V\in\mathcal A_\al $, set
\[
\mathcal K_\al (V) = \bar \cG^{(0,2)} (\kappa(H_\al +V)^{p^*}).
\]
				If $\al >0$ is sufficiently small, then
\[
|V|\le M\al ^{p^*}\le \frac{\al }{4}\quad \text{ and }\quad H_\al +V\ge \frac{3\al }{4}>0.
\]
				Moreover,  $\|H_\al +V\|_{\mathscr X}\le C\al $. Since
				$s\mapsto s^{p^*}$ is smooth on $(0,+\infty)$, the standard composition
				estimate in H\"older spaces gives
\[
\|\kappa(H_\al +V)^{p^*}\|_{\mathscr Y}\le C\kappa\al ^{p^*}.
\]
				Consequently,
				$\|\mathcal K_\al (V)\|_{\mathscr X}\le C\kappa\al ^{p^*}$.
				Choosing $M$ sufficiently large depending only on $n,\gamma,\kappa$, and
				then choosing $\al _0>0$ sufficiently small, we obtain
				$\mathcal K_\al (\mathcal A_\al )\subset\mathcal A_\al $ for all $\al \in (0,\al _0)$.
				If $V_1,V_2\in\mathcal A_\al $, then the mean value theorem and the
				H\"older composition estimate give
\[
\|\kappa(H_\al +V_1)^{p^*} -\kappa(H_\al +V_2)^{p^*}\|_{\mathscr Y} \le C\kappa\al ^{p^*-1}\|V_1-V_2\|_{\mathscr X}.
\]
				Therefore,
				\be\label{Lipschitzestimate}
				\|\mathcal K_\al (V_1)-\mathcal K_\al (V_2)\|_{\mathscr X}
				\le
				C\kappa\al ^{p^*-1}\|V_1-V_2\|_{\mathscr X}.
				\ee
				After decreasing $\al _0$ if necessary, $\mathcal K_\al $ is a
				contraction on $\mathcal A_\al $. Hence there exists a unique
				$V_\al \in\mathcal A_\al $ such that
\[
V_\al = \bar \cG^{(0,2)} (\kappa(H_\al +V_\al )^{p^*}).
\]

				Define
				$	U_\al :=H_\al +V_\al $.
				Then
				$\Delta^2U_\al =\kappa U_\al ^{p^*}$ in $\B^{n+1}$.
				Since $\bar \cG^{(0,2)} $ has zero $\sB^3_0$- and $\sB^3_2$-boundary data, we have
\[
U_\al =\al , ~ \sB^3_2U_\al =c_2\al ^{p^*_2} \quad\text{ on }\,\Sn.
\]
				Moreover,
\[
U_\al \ge \frac{3\al }{4}>0 \quad\text{ in }\,\overline{\B^{n+1}}.
\]
				The construction is radial, since $H_\al $ is radial and
				$\bar{\cG}^{(0,2)}$ preserves radial symmetry.
				It remains to estimate $\sB^3_3U_\al $.

				Since $U_\al $ is radial,
				$\sB^3_3U_\al $ is constant on $\Sn$. From the fixed point construction,
				$\|V_\al \|_{\mathscr X}\le C\kappa\al ^{p^*}$.
				Therefore,
\[
\sB^3_3V_\al =O(\al ^{p^*}) \quad\text{ on }\,\Sn.
\]
				A direct computation using \eqref{Boundaryoperators-Ball} gives
\[
\sB^3_3H_\al =\frac{(n^2-1)(n-3)}{4}H_\al (1)=\frac{(n^2-1)(n-3)}{4}\al \quad\text{ on }\,\Sn.
\]
				It follows that
\[
C_3(\al ):=\sB^3_3U_\al |_{\Sn} = \sB^3_3H_\al |_{\Sn}+\sB^3_3V_\al |_{\Sn} = \frac{(n^2-1)(n-3)}{4}\al +O(\al ^{p^*}).
\]
				Consequently,
\[
c_3(\al ) := \frac{C_3(\al )}{\al ^{p^*_3}} = \frac{(n^2-1)(n-3)}{4}\al ^{1-p^*_3} + O(\al ^{p^*-p^*_3}).
\]
				Since
				$1-p^*_3=-\frac6{n-3}$ and
				$p^*-p^*_3=\frac2{n-3}>0$,
				we have
				$c_3(\al )\to+\infty$ as $\al \to 0^{+}$.

				The fixed point depends continuously, indeed $C^1$, on $\al $. This
				follows either from the contraction mapping theorem with parameters or from
				the implicit function theorem applied to
\[
\mathcal F(\al ,V) = V-\bar \cG^{(0,2)}(\kappa(H_\al +V)^{p^*}).
\]
				At $V=V_\al $, we have
\[
D_V\mathcal F(\al ,V_\al )[\varphi] = \varphi - \bar \cG^{(0,2)} (\kappa p^*U_\al ^{p^*-1}\varphi),
\]
				where 	$D_V\mathcal F$ denotes the Fr\'echet derivative of $\mathcal F$ with
				respect to the second variable $V$. The preceding Lipschitz estimate in \eqref{Lipschitzestimate} also implies that the linear operator
				$\varphi
				\mapsto
				\bar{\cG}^{(0,2)}(\kappa p^*U_\al ^{p^*-1}\varphi)$
				has $\mathscr X\to\mathscr X$ norm $O(\al ^{p^*-1})$. Hence, after possibly decreasing $\al _0$, $D_V\mathcal F(\al ,V_\al )$ is invertible for all $\al \in (0,\al _0)$. Therefore $\al \mapsto V_\al $, and consequently $\al \mapsto c_3(\al )$, is continuous.

				Fix $\al _1\in(0,\al _0)$. Since $c_3(\al )\to+\infty$ as
				$\al \to 0^{+}$,  for every sufficiently large $c_3>0$
				there exists $\al \in(0,\al _1]$ such that $c_3=c_3(\al )$. For this choice of $\al $,
\[
\sB^3_3U_\al = C_3(\al ) = c_3\al ^{p^*_3} = c_3U_\al ^{p^*_3} \quad\text{ on }\,\Sn.
\]
				Together with the prescribed $\sB_2^3$-boundary condition, this completes the proof.
			\end{proof}

			\begin{proof}[Proof of Theorem \ref{thm:Existence}]
				Theorem \ref{thm:Existence} follows from Lemmas \ref{lem:subsupsolution} and \ref{lem:ball-B0B2-construction-c3-asymptotic}.
			\end{proof}

			\section{The minimal-boundary case: radial ODE analysis and parametrization}\label{sec:5}

			We now turn to the minimal-boundary case $c_1=0$ and complete the proof of
Theorem \ref{thm:Main}.  By the moving-sphere argument in Section \ref{sec:3}, the boundary trace has already been classified, and the problem is reduced, after conformal normalization, to a radial boundary value problem on the unit ball.
The purpose of this section is twofold. First, we analyze the associated radial fourth-order ODE and prove the uniqueness and nonexistence statements needed for the classification. Second, we study the dependence of the induced $\mathscr B_3^3$-curvature coefficient on the boundary height $\al $, thereby obtaining the parametrization of solutions by the constant $c_3$.

			\subsection{Reduction to an autonomous fourth-order ODE}
			We first pass from the radial boundary value problem on the unit ball to an
			autonomous ordinary differential equation on the  half-line. Throughout this
			subsection, let $\al>0$ and let
$U=U(r)$, $0\leq r\leq1$, be a smooth positive radial solution of
			\be\label{c1=0}
			\left\{
			\begin{aligned}
				&\Delta^2 U=\kap U^{p^*} &&\text{ in }\, \B^{n+1},\\
				&U=\al,~\sB_1^3U=0 &&\text{ on }\, \Sn.
			\end{aligned}
			\right.
			\ee

			Since $U$ is radial, the boundary quantity $\mathscr B_3^3U$ is constant on
			$\mathbb S^n$. We therefore define the associated third boundary curvature
			coefficient by
			\begin{equation}\label{c3_alpha}
				c_3(\al )
				:=
				\al ^{-p_3^*}\mathscr B_3^3U|_{\mathbb S^n},
				\qquad
				p_3^*=\frac{n+3}{n-3}.
			\end{equation}
			Equivalently,
			\[
				\mathscr B_3^3U=c_3(\al )U^{p_3^*}
				\quad\text{ on }\,\mathbb S^n.
	\]
			In particular, the standard bubble $U_*$ corresponds to $\al =1$, and
			\eqref{eq:Bubble} gives $c_3(1)=0$. At this stage, $c_3(\al)$ denotes the coefficient associated with the
chosen radial solution; the uniqueness result below will show that, in the
relevant range, it is uniquely determined by $\al$.

			We now introduce the Emden--Fowler variable
\[
r=e^{-t},\quad t\in[0,+\infty),
\]
			and set
			\begin{equation*}
				v(t)=e^{-at}U(e^{-t}),
				\quad
				a:=\frac{n-3}{2}.
			\end{equation*}
			Equivalently,
\[
U(r)=r^{-a}v(-\log r), \quad r\in(0,1].
\]
			Since $U$ is smooth at the origin, differentiation of
$v(t)=e^{-at}U(e^{-t})$ gives
\[
v^{(j)}(t)=O(e^{-at})\qquad\text{ as }\,  t\to+\infty
\] for  $j=0,1,\ldots,4$.
In particular,
\be\label{vdecay}
v(t)\to0
\qquad\text{ as }\,t\to+\infty.
\ee

			Set
			$b:=\frac{n+1}{2}$.
			A direct computation using the radial Laplacian in $\R^{n+1}$ gives
\[
\Delta^2 U(r) = r^{-a-4} [ v^{(4)}(t) - (a^2+b^2)v''(t) + a^2b^2v(t) ], \quad t=-\log r.
\]
			Since $a+4=ap^*$, the interior equation in \eqref{c1=0} is transformed into
the autonomous fourth-order equation
			\begin{equation}\label{fourorderODE2}
				v^{(4)}-(a^2+b^2)v''+a^2b^2v
				=
				\kappa v^{p^*},
				\quad t>0.
			\end{equation}

			We next rewrite the boundary conditions in terms of $v$. Since
$U(1)=\al$, we have
\[
v(0)=\al .
\]
			Moreover, differentiating $U(r)=r^{-a}v(-\log r)$ gives
\[
U_r(1)=-a v(0)-v'(0).
\]
			Since
\[
\mathscr B_1^3U=U_r+\frac{n-3}{2}U = U_r+aU \quad\text{on }\mathbb S^n,
\]
			the condition $\mathscr B_1^3U=0$ is equivalent to
\[
v'(0)=0.
\]
			Thus
			\begin{equation}\label{fourorderODEbdy}
				v(0)=\al ,
				\quad
				v'(0)=0.
			\end{equation}

			Similarly, using the formula for $\sB_3^3$ in
\eqref{Boundaryoperators-Ball}, expressed in the Emden--Fowler variable,
we obtain
\[
\mathscr B_3^3U = v'''(0)-\frac{3n^2-6n+7}{4}v'(0)=v'''(0) \quad\text{ on }\,\mathbb S^n
\]
	in view of \eqref{fourorderODEbdy}.	Consequently, by the definition of $c_3(\al )$ in \eqref{c3_alpha},
\[
v'''(0)=c_3(\al )\al ^{p_3^*}.
\]
			In particular,
			\begin{equation}\label{v'''0>0}
				c_3(\al )>0
				\quad\Longleftrightarrow\quad
				v'''(0)>0.
			\end{equation}

	The autonomous equation \eqref{fourorderODE2} on the whole line was studied in detail by Frank and K\"onig \cite{FK2019}. Our setting is different, however, since we consider a half-line problem with boundary data prescribed at $t=0$ and decay at $+\infty$ inherited from the regularity of $U$ at the origin. The conserved Hamiltonian associated with \eqref{fourorderODE2} provides the first restriction on the admissible initial data.

			Indeed, multiplying \eqref{fourorderODE2} by $v'$ yields
\[
\frac{\d}{\d t}\mathcal H(v)=0,
\]
where
\begin{equation}\label{Hamiltonian}
\mathcal H(v)
=
v'v'''
-\frac12(v'')^2
-\frac{a^2+b^2}{2}(v')^2
+\frac{a^2b^2}{2}v^2
-\frac{\kap}{p^*+1}v^{p^*+1}.
\end{equation}
Thus $\mathcal H(v)$ is constant along every solution of
\eqref{fourorderODE2}. Since $v$ and its derivatives decay as
$t\to+\infty$, we obtain
\[
\mathcal H(v)\equiv0.
\]
			Evaluating  \eqref{Hamiltonian} at $t=0$, and using \eqref{fourorderODEbdy},
			we find
			\be\label{v''(0)2}
			\frac12(v''(0))^2
			=
			\frac{a^2b^2}{2}\al^2
			-\frac{\kap}{p^*+1}\al^{p^*+1}=\frac{a^2\al^2}{2}[b^2-(b^2-1)\al^{p^*-1}].\ee
		Define $h(\al)\geq0$ by
			\be\label{d2alpha}
			h^2(\al)
			=
			a^2\al^2[b^2-(b^2-1)\al^{p^*-1}].
			\ee
			Hence $	h^2(\al)\geq 0$ if and only if
			\be\label{alpha_upperbound}
			0<\al\le \al_*=
			\Big(
			\frac{(n+1)^2}{(n-1)(n+3)}
			\Big)^{\frac{n-3}{8}}.
			\ee
			In particular, $h(1)=a>0$ and 	$v''(0)=\pm h(\al)$ for $\al$ satisfying \eqref{alpha_upperbound}. Note also that $\al_*>1$ for $n\geq 4$.






			To compare solutions corresponding to different admissible shooting data,
we shall use the following fourth-order comparison principle. The argument
is adapted from the proof of \cite[Theorem B.1]{CZ202406}.

			\begin{lem}\label{lem:Use}
				Let $v$ be a positive solution of  \eqref{fourorderODE2} and let $u$ be a  positive subsolution of  \eqref{fourorderODE2}, i.e.,
				\[
					u^{(4)}-(a^2+b^2)u''+a^2b^2u
					\leq
					\kappa u^{p^*},
					\qquad t>0.
				\]
		Assume that
			\[
					v(0)=u(0),\quad v'(0)=u'(0), \quad v''(0)\geq u''(0),\quad v'''(0)>u'''(0).
			\]
				Then  
\[
\lim_{t\to+\infty}(v(t)-u(t))=+\infty.
\]
			\end{lem}
			\begin{proof}
				Set $\phi(t)=v(t)-u(t)$. The assumptions on the initial data imply that
$\phi(t)>0$ for all sufficiently small $t>0$. We first show that $\phi(t)>0$ for all $t>0$. Suppose otherwise, and let
\[
T:=\sup\{\tau>0:\phi(t)>0\text{ for all }t\in(0,\tau)\}.
\]
				Then $T\in (0,+\infty)$, and by continuity
				of $\phi$, we have
				$\phi>0$ on $(0,T)$ and $\phi(T)=0$.  Since
$v>u$ on $(0,T)$, we have
\[
\phi^{(4)}-(a^2+b^2)\phi''+a^2b^2\phi \geq \kap (v^{p^*}-u^{p^*})>0 \quad\text{ on }\,(0,T).
\]
				
                Define $k(t):=\phi''(t)-b^2\phi(t)$.
				Then $k''-a^2k>0$ on $(0,T)$. Since $\phi(0)=\phi'(0)= 0$, we have
				\be\label{kinitial}
				k(0)=\phi''(0)\ge0,
				\quad
				k'(0)=\phi'''(0)>0.
				\ee
				By the Duhamel formula for $k''-a^2k$, we derive
				\be\label{integralk(t)}
				k(t)
				=
				k(0)\cosh(at)
				+
				\frac{k'(0)}{a}\sinh(at)
				+
				\int_0^t
				\frac{\sinh(a(t-s))}{a}
				(k''(s)-a^2k(s))\,\d s>0
				\ee
				for every $t\in(0,T)$, where $\sinh t=\frac{e^{t}-e^{-t}}{2}$.
				Applying the Duhamel formula to
$\phi''-b^2\phi=k$ with $\phi(0)=\phi'(0)=0$, we obtain
				\be\label{integralphi(t)}
				\phi(t)
				=
				\int_0^t
				\frac{\sinh(b(t-s))}{b}k(s)\,\d s
				>0
				\ee
				for every $t\in (0,T]$. In particular, $\phi(T)>0$,	contradicting $\phi(T)=0$.  Hence no such
				$T$ exists, and therefore $\phi(t)>0$  for all $t>0$.

				It remains to determine the behavior as $t\to+\infty$. Since
$v>u$ on $(0,+\infty)$,
\[
k''-a^2k = \phi^{(4)}-(a^2+b^2)\phi''+a^2b^2\phi \geq \kap(v^{p^*}-u^{p^*})>0 \quad\text{ on }\,(0,+\infty).
\]
			It follows from \eqref{kinitial} and \eqref{integralk(t)} that
\[
k(t)\ge \frac{k'(0)}{a}\sinh(at)>0 \quad\text{ on }\,(0,+\infty).
\]
				Substituting this estimate into \eqref{integralphi(t)}, we obtain
\[
\phi(t) = \int_0^t \frac{\sinh(b(t-s))}{b}k(s)\,\d s\geq \frac{k'(0)}{ab} \int_0^t \sinh(b(t-s))\sinh(as)\,\d s \to+\infty \quad\text{ as }\,t\to+\infty,
\]
				which completes the  proof.
			\end{proof}

            We now exclude the case $\al \ge1$ under the positivity assumption
			$c_3(\al )>0$.  By \eqref{v'''0>0}, this is equivalent to
$v'''(0)>0$.

			\begin{lem}\label{lem:no-positive-c3-C-geq-one}
				Let $\al\ge 1$. Then there is no positive smooth solution $v$ of \eqref{fourorderODE2}
				satisfying  \eqref{vdecay}, \eqref{fourorderODEbdy}, and \eqref{v'''0>0}.
			\end{lem}

			\begin{proof}
			By \eqref{alpha_upperbound}, no such solution can exist when
$\al>\al_*$. It therefore remains to consider
				$\al\in [1, \al_*]$.

				Suppose, by contradiction, that such a solution $v(t)$ exists.
				Since $\al\ge1$, \eqref{d2alpha} gives
				$|v''(0)|\le a\al$,
				and hence
				$v''(0)+a\al\ge0$.

				Let
				$w(t)=(\cosh t)^{-a}$, where $\cosh t=\frac{e^t+e^{-t}}{2}$.
				Then $w$ is the Emden-Fowler profile of the standard bubble and satisfies
\[
w^{(4)}-(a^2+b^2)w''+a^2b^2w=\kap w^{p^*},
\]
				with
\[
w(0)=1,\quad w'(0)=0,\quad w''(0)=-a,\quad w'''(0)=0.
\]

 We compare $v$ with $\al w$.               Set
				$\phi(t):=v(t)-\al w(t)$.
				Then
\[
\phi(0)=0,\quad \phi'(0)=0,
\]
				and
\[
\phi''(0)=v''(0)+a\al \ge0, \quad \phi'''(0)=v'''(0)>0.
\]
	Moreover, since $\al\geq1$ and $p^*>1$,
\[
(\al w)^{(4)}-(a^2+b^2)\al w''+a^2b^2\al w=\kap \al w^{p^{*}} \leq \kap (\al w)^{p^*} \quad\text{ on }\,(0,+\infty).
\]
			Thus $\al w$ is a positive subsolution of \eqref{fourorderODE2}.  Lemma \ref{lem:Use} therefore yields
				$v(t)-\al w(t)\to+\infty$  as $t\to +\infty$. On the other hand, by \eqref{vdecay} and the exponential decay of $w$,
$v(t)-\al w(t)\to 0$ as $t\to +\infty$,
a contradiction.
			\end{proof}

			It remains to analyze the case $\al \in(0,1)$. In this range, the Hamiltonian identity yields exactly two possible shooting branches:
			\be\label{twobranchs}
			v''(0)=\pm h(\al),
			\ee
			where $h(\al)>0$ for all $\al \in(0,1)$.
			We next prove that, for each fixed branch in \eqref{twobranchs}, there is at most one positive decaying solution with
$v'''(0)>0$.
			\begin{lem}\label{lem:same-sign-uniqueness}	Let $\al\in (0,1)$. For each choice of sign in \eqref{twobranchs},
				there exists at most one positive regular radial solution $v$ of
				\eqref{fourorderODE2} satisfying  \eqref{vdecay}, \eqref{fourorderODEbdy}, and \eqref{v'''0>0}.
			\end{lem}

			\begin{proof}
			Let $v_1$ and $v_2$ be two such solutions belonging to the same
Hamiltonian branch. Then
\[
v_1(0)=v_2(0)=\al,\quad v_1'(0)=v_2'(0)=0,\quad v_1''(0)=v_2''(0).
\]
				Thus the only initial datum that may differ is the third derivative at
$t=0$. 

 Suppose, by contradiction, that
				$v_1'''(0)>v_2'''(0)$. Applying Lemma \ref{lem:Use} with $v=v_1$ and $u=v_2$, we obtain
				$v_1(t)-v_2(t)\to +\infty$ as $t\to+\infty$, 
				which contradicts \eqref{vdecay}.
				Therefore $v_1'''(0)>v_2'''(0)$ is impossible. Interchanging $v_1$ and $v_2$ rules
out the opposite inequality.  Hence
				$v_1'''(0)=v_2'''(0)$.
				Therefore $v_1$ and $v_2$ have the same four initial data at $t=0$. By
				the uniqueness theorem for the initial value problem of the fourth-order ODE \eqref{fourorderODE2},
				we conclude that
				$v_1\equiv v_2$.
				The proof is complete.
			\end{proof}

			It remains to determine which of the two branches in
\eqref{twobranchs} can give rise to a positive decaying solution when
$c_3(\al)>0$. The next lemma shows that the positive branch
$v''(0)>0$ is impossible.
			\begin{lem}\label{lem:no-decay-positive-second-third}
				Let $\al\in  (0,1)$. There is no positive smooth solution $v$ of \eqref{fourorderODE2}
				satisfying  \eqref{vdecay}, \eqref{fourorderODEbdy}, \eqref{v'''0>0}, and $v''(0)>0$.
			\end{lem}

			\begin{proof}
				Suppose, by contradiction, that such a solution $v$ exists.
				Since $v''(0)>0$, \eqref{v''(0)2} and \eqref{d2alpha} give
\[
v''(0)=h(\al) = a\al \sqrt{b^2-(b^2-1)\al^{p^*-1}}:=a\al \delta,
\]
				where
				$\delta=\sqrt{b^2-(b^2-1)\al^{p^*-1}}\in (1,b)$ in view of $\al\in (0,1)$.

				Set
\[
y(t):=v'(t)\quad \text{ and }\quad z(t):=y''(t)-a^2y(t).
\]
				Differentiating \eqref{fourorderODE2}, we obtain
				\be\label{z''}
				z''-b^2z=\kap   p^* v^{p^*-1}y.
				\ee
				Using \eqref{fourorderODEbdy} and \eqref{v'''0>0}, we have initial data
				\be\label{yinitial}
				y(0)=0,
				\quad
				y'(0)=v''(0)>0,
				\ee
				and
\[
z(0)=y''(0)-a^2y(0)=v'''(0)>0.
\]

			We first claim that $z'(0)>0$.
				Indeed,
				$z'(0)=y'''(0)-a^2y'(0)
				=
				v^{(4)}(0)-a^2v''(0)$.
				Evaluating \eqref{fourorderODE2} at $t=0$ and using \eqref{fourorderODEbdy}, we get
				$
				v^{(4)}(0)
				=
				(a^2+b^2)v''(0)-a^2b^2\al+\kap  \al^{p^*}$.
				Hence
\[
z'(0) = b^2h(\al)-a^2b^2\al+\kap \al^{p^*}.
\]
				Using $h(\al)=a\al\delta$ and
				$\al^{p^*-1}
				=
				\frac{b^2-\delta^2}{b^2-1}$,
				we find
\[
z'(0) = a b\al [b(b-a)+\delta(b-\delta)]>0,
\]
				since $b>a$ and $\delta\in(1,b)$.

				We next show that   $y(t)>0$ for all $t>0$.
				By \eqref{yinitial},
				we have $y(t)>0$ for $t>0$ small. Let
\[
T:=\sup\{\tau>0: y(t)>0\ \text{for all }t\in(0,\tau)\}.
\]
			If $T<+\infty$, then $y>0$ on $(0,T)$. Since $v>0$, the right-hand
side of \eqref{z''} is positive on $(0,T)$. The Duhamel formula yields
\[
z(t) = z(0)\cosh(bt) + \frac{z'(0)}{b}\sinh(bt) + \int_0^t \frac{\sinh(b(t-s))}{b} \kap p^* v(s)^{p^*-1}y(s)\,\d s >0
\]for $0<t\leq T$. Applying the Duhamel formula once more to $y''-a^2y=z$
gives
\[
y(t) = \frac{h(\al)}{a}\sinh(at) + \int_0^t \frac{\sinh(a(t-s))}{a}z(s)\,\d s >0
\]
				for $0<t\leq T$. In particular, $y(T)>0$, contradicting the definition
of $T$. Hence $y(t)>0$ for all $t>0$.

				The preceding formulas now hold on the whole half-line. Since $z>0$,
we have
\[
y(t)
\geq
\frac{h(\al)}{a}\sinh(at),
\qquad t>0.
\]
Consequently,
\[
v(t)
=
v(0)+\int_0^t y(s)\,\d s
\geq
\al+\frac{h(\al)}{a^2}\bigl(\cosh(at)-1\bigr)
\longrightarrow+\infty,
\]
which contradicts \eqref{vdecay}. This completes the proof.
			\end{proof}


			\subsection{Continuity and monotonicity of \texorpdfstring{$c_3$}{c3}}

		We now study how the third boundary curvature coefficient depends on the
prescribed boundary value $\al$ in the minimal-boundary case $c_1=0$.  By Lemma
\ref{lem:subsupsolution}, applied with $c_1=0$, for each $\al \in(0,1)$ there exists a positive smooth radial solution $U_\al \in C^\infty(\overline{\B^{n+1}})$ of \eqref{c1=0}. Moreover, the lower bound in \eqref{c3bdd} shows that
the corresponding third boundary coefficient is positive. Hence, by
Lemmas \ref{lem:same-sign-uniqueness} and
\ref{lem:no-decay-positive-second-third}, this solution is unique among
positive radial solutions with positive third boundary coefficient.

			Since $U_\al $ is radial, $\sB_3^3U_\al $ is constant on $\S^n$. We therefore set
			\be\label{c3alpha}
			c_3(\al ) := \frac{\sB_3^3U_\al |_{\S^n}}{\al ^{p_3^*}},
			\ee
			so that
			\[ \sB_3^3U_\al  = c_3(\al )U_\al ^{p_3^*} \quad\text{ on }\,\Sn. \]
Thus $c_3(\al)$ is well defined for every $\al\in(0,1)$.

			When $c_1=0$, the lower bound in \eqref{c3bdd} becomes
			\be\label{c3lowerbdd}
			c_3(\al ) \ge \frac{(n^2-1)(n-3)}{4}\al ^{-\frac6{n-3}} (1-\al ^{\frac8{n-3}}),
			\ee
		and therefore
			\be\label{c3blowup}
			\lim_{\al\to 0^{+}}c_3(\al)=+\infty.
			\ee
			It remains to establish the continuity and strict monotonicity of the map $\al \mapsto c_3(\al )$, as well as the limiting behavior as $\al \to1^-$. This is the content of the following proposition.
			\begin{prop}\label{prop:c3}
				The map $c_3:(0,1)\to(0,+\infty)$ is continuous and strictly decreasing.  Moreover,
\[
\lim_{\al\to 0^{+}}c_3(\al)=+\infty\quad \text{ and }\quad \lim_{\al\to 1^{-}}c_3(\al)=0.
\]
			\end{prop}

			We divide the proof of Proposition \ref{prop:c3} into several lemmas.

			\begin{lem}\label{lem:boundary-comparison}
				Let $Z=Z(r)$ and $W=W(r)$ be  positive radial functions in
				$\B^{n+1}$, whose radial profiles are continuous on $[0,1]$ and $C^2$
				up to $r=1$. Assume that
\[
Z=W=0,\quad \sB_1^3Z=\sB_1^3 W=0 \quad\text{ on }\,\S^n,
\]
				and
\[
Z_{rr}(1)>0,~ W_{rr}(1)>0. \]Then there exists $\delta>0$ such that $Z\ge \delta W$ on $\overline{\B^{n+1}}$. 
\end{lem} 

\begin{proof} By \eqref{Boundaryoperators-Ball}, the boundary conditions
$Z=\sB_1^3Z=0$ imply $Z(1)=Z'(1)=0$. Similarly, $W(1)= W'(1)=0$. Since $Z''(1)>0$ and $W''(1)>0$, it follows from Taylor expansion at $r=1$ that \[ \frac{Z(r)}{W(r)} \to \frac{Z''(1)}{W''(1)}>0 \quad\text{ as }\,r\to1^-.
\]
				Hence $Z/W$ extends continuously to $\S^n$ with a positive boundary value.
				Since $Z$ and $W$ are strictly positive in
$\B^{n+1}$, the extended quotient is positive on the compact interval
$[0,1]$. Therefore
\[
\delta := \min_{\overline{\B^{n+1}}}\frac{Z}{W} >0,
\]
				and hence
				$	Z\ge \delta W$ on $\overline{\B^{n+1}}$.
			\end{proof}

			\begin{lem}\label{lem:c3continous}
				For every $\gamma\in(0,1)$, the map
				$\al \to U_\al
				\in C^{4,\gamma}(\overline{\B^{n+1}})$
				is continuous in $(0,1)$. In fact, this map is $C^1$.
			\end{lem}

			\begin{proof}
				Let $U_*$ be the standard bubble defined in \eqref{standardbubble}, and let
				$H$ be the radial homogeneous lifting defined in \eqref{Hr}. In the  case $c_1=0$, the subsolution \eqref{subsolution} constructed in
				Lemma \ref{lem:subsupsolution} reduces to $\underline U=\al  H$, whereas the supersolution \eqref{supersolution} reduces to $\overline U=\al U_*$. Hence, by \eqref{Ulimit},
\begin{equation}\label{Ualpha-bounds}
\al H\leq U_\al\leq\al U_*
\qquad\text{ in }\,\B^{n+1}.
\end{equation}
				Moreover, in view of \eqref{Ulimit-expression}, $U_\al $ admits the fixed
				point representation
				\be\label{Ualpha_expression}
				U_{\al}=\al H+\bar \cG^{(0,1)}(\kap  U_{\al}^{p^{*}}).
				\ee

			We first obtain a strict version of the upper bound in
\eqref{Ualpha-bounds}.  Set $W=\al  U_*-U_\al $.  Using \eqref{eq:Bubble} and \eqref{Ulimt_equation}, together with
				$U_\al \le\al  U_*$ and $\al \in (0,1)$, we obtain
				\be\label{W1}
				\Delta^2W
				=
				\kappa\al  U_*^{p^*}-\kappa U_\al ^{p^*}
				\ge
				\kappa(\al -\al ^{p^*})U_*^{p^*}>0
				\quad\text{ in }\,\B^{n+1},
				\ee
				Moreover,
				\be\label{W2}
				W=0,\quad \sB_1^3W=0
				\quad\text{ on }\, \S^n.
				\ee
			Therefore, by the positivity of the Green operator
$\bar{\cG}^{(0,1)}$,  $W>0$ in $\B^{n+1}$.

				We now prove the local invertibility of the fixed point equation
				\eqref{Ualpha_expression}.
				Fix $\gamma\in(0,1)$, and set
\[
\mathscr X_0=C^{0,\ga }_{\mathrm{rad}}(\overline{\B^{n+1}})
\]
				equipped with the standard $C^{0,\gamma}$-norm. Define the linear operator
				$\mathscr L:\mathscr X_0\to\mathscr X_0$ by
\[
\mathscr{L}\phi = \bar{\cG}^{(0,1)}(\kap p^{*} U_{\al}^{p^{*}-1}\phi), \quad \phi\in \mathscr X_0.
\]
				Since
				$\bar{\cG}^{(0,1)}:C^{0,\ga }(\overline{\B^{n+1}})
				\to C^{4,\ga }(\overline{\B^{n+1}})$
				is bounded and the embedding
				$C^{4,\ga }(\overline{\B^{n+1}})
				\hookrightarrow C^{0,\ga }(\overline{\B^{n+1}})$
				is compact, $\mathscr L:\mathscr X_0\to\mathscr X_0$ is compact. Moreover, $\mathscr L$ is positive by  \eqref{positiveGreen}.

				We claim that
				$r(\mathscr L)<1$,
				where $r(\mathscr L)$ denotes the spectral radius of $\mathscr L$.
				Indeed,
				by the strict positivity of $W$ and \eqref{W1}, we have
				\begin{align*}
					\Delta^2W
					&=
					\kap  \al U_*^{p^{*}}-\kap  U_{\al}^{p^{*}}>\kap ((\al U_*)^{p^{*}}-U_{\al}^{p^{*}})  \\
					&>
					\kap  p^{*}U_{\al}^{p^{*}-1}(\al U_*-U_{\al})
					=
					\kap  p^{*} U_{\al}^{p^{*}-1}W\quad \text{ in }\, \B^{n+1}.
				\end{align*}
				Equivalently,
				$W>\mathscr L W $ in   $\B^{n+1}$.
				Let $Z=W-\mathscr LW$. Then
				$Z>0$ and $\Delta^2 Z>0$  in $ \B^{n+1}$. Also, in view of \eqref{W2} and the fact that
				$\bar{\cG}^{(0,1)}$ has homogeneous $\sB_0^3$- and
				$\sB_1^3$-boundary data, we have
\[
Z=0,~ \sB _1^3Z=0 \quad\text{ on }\, \Sn
\]

			Applying Lemma \ref{lem:Comparison} to $W$ and $Z$, we obtain $Z_{rr}(1), W_{rr}(1)>0$. Then by Lemma \ref{lem:boundary-comparison}, there exists $\delta>0$ such that
				$Z\geq \delta W$  on $\overline{\B^{n+1}}$.	 Thus
				$\mathscr	L W\le (1-\delta)W$.
				Let $\lambda\ne0$ be an eigenvalue of $\mathscr L$, and let
				$\phi\in\mathscr X_0\setminus\{0\}$ be a corresponding eigenfunction:
				$\mathscr L\phi=\lam\phi$.
				Since $\lambda\ne0$, we have
				$\phi=\mathscr L(\lambda^{-1}\phi)$.
				Thus $\phi\in C^{4,\ga }_{\mathrm{rad}}(\overline{\B^{n+1}})$ satisfies
\[
\phi=0\quad \text{ and }\quad \sB _1^3\phi=0 \quad\text{ on }\, \Sn.
\]
				By the same boundary behavior, $|\phi|/W$ is bounded on
				$\overline{\B^{n+1}}$. Let
\[
M:=\sup_{\overline{\B^{n+1}}}\frac{|\phi|}{W}.
\]
				Then $0<M<+\infty$, and
				$|\phi|\le MW$.
				Using the positivity of $\mathscr L$, we obtain
\[
|\lam |\,|\phi| = |\mathscr L\phi| \le \mathscr L|\phi| \le M \mathscr L W \le M(1-\delta)W.
\]
				Dividing by $W$ and taking the supremum yields
				$|\lambda|M\le M(1-\delta)$.
				Since $M>0$, it follows that
				$|\lambda|\le 1-\delta<1$.
				Since
				$\mathscr L$ is compact, this implies
				$r(\mathscr L)<1$.
				In particular,
				$I-\mathscr L:\mathscr X_0 \to \mathscr X_0$
				is invertible.

				Now consider the nonlinear map
\[
\mathcal F(\al,U) = U-\al H-\bar{\cG}^{(0,1)}(\kap U^{p^{*}}).
\]
				Then
				$\mathcal F(\al,U_{\al})=0$,
				and the Fr\'echet derivative with respect to $U$ at $(\al ,U_\al )$ is
				$D_U\mathcal F(\al,U_{\al})=I-\mathscr L$.
				Since $I-\mathscr L$ is invertible, the implicit function theorem implies that $\al \mapsto U_\al $ is locally $C^1$ as a map into
				$\mathscr X_0$. By the fixed point representation and the Schauder estimate
				for $\bar{\cG}^{(0,1)}$, this dependence is in fact locally $C^1$
				in $C^{4,\gamma}(\overline{\B^{n+1}})$. Since $\al \in(0,1)$ was
				arbitrary, the branch $U_{\al}$ is continuous on $(0,1)$.
				%
				%
			\end{proof}
			\begin{rem}\label{rem:Sec5}
				Set $\dot U_{\al}:=\partial_{\al}U_{\al}$. Differentiating
\eqref{Ualpha_expression}, we obtain
$(I-\mathscr L)\dot U_{\al}=H$.
Since $r(\mathscr L)<1$ and $\mathscr L$ is positive,
$(I-\mathscr L)^{-1}$ is positive. Hence
$\dot U_{\al}=(I-\mathscr L)^{-1}H>0$ in $\B^{n+1}$.
Therefore $U_{\al}$ is strictly increasing in $\al$.
			\end{rem}

			\begin{lem}\label{lem:c3-limit-one}
				$	\lim_{\al\to 1^{-}}c_3(\al)=0$.
			\end{lem}

			\begin{proof}
				We first note that \eqref{c3lowerbdd} implies  $c_3(\al)>0$ for every $\al\in (0,1)$.
				We next prove the limit as $\al \to1^-$. Let $\al _j\to1^-$ be an
				arbitrary sequence, and set $U_j:=U_{\al _j}$. By the construction in
				Lemma \ref{lem:subsupsolution}, 
\[
\al_jH\le U_j\le \al_jU_* \quad \text{ in }\, \B^{n+1}.
\]
				Hence $\{U_j\}$ is uniformly bounded from above and below by positive
				constants. Standard elliptic estimates then imply that, after passing to a
				subsequence,
				\be\label{Ujconvergence}
				U_j\to U_\infty
				\quad\text{ in }\,C^4(\overline{\B^{n+1}}).
				\ee
				Passing to the limit, we obtain
				\be\label{Uinfinity}
				\left\{\begin{aligned}
					&\Delta^2U_\infty=\kap  U_\infty^{p^{*}}
					&&\text{ in }\,\B^{n+1},\\
					&U_\infty=1,~\sB^3_1U_\infty=0
					&&\text{ on }\,\Sn.
				\end{aligned}\right.
				\ee
				Moreover, since $c_3(\al _j)>0$ and
				$\sB_3^3U_j
				=
				c_3(\al _j)\al _j^{p_3^*}$ on $\Sn$,
				we have $\sB_3^3U_j>0$ on $\S^n$. Therefore, by
				\eqref{Ujconvergence},
				$\sB^3_3U_\infty\ge0$  on $\Sn$.

				We claim that $\sB _3^3U_\infty=0$
				on $\Sn$.
				Indeed, since $U_\infty$ is radial, $\sB_3^3U_\infty$ is constant on
				$\S^n$. If $\sB_3^3U_\infty\not\equiv0$, then
				$\sB_3^3U_\infty>0$ on $\S^n$. Because $U_\infty=1$ on $\S^n$, this
				contradicts Lemma \ref{lem:no-positive-c3-C-geq-one}. Hence
				$\sB _3^3U_\infty=0$  on $\Sn$.

				By Remark \ref{rem:1}(2), the only solution of \eqref{Uinfinity} satisfying
				$\sB_3^3U_\infty=0$ on $\S^n$ is an Aubin--Talenti bubble. Thus
				$U_\infty=U_*$.
				Consequently,
				$\sB^3_3U_j\to0 $ on $\Sn$.
				Since
				$c_3(\al_j)=\frac{\sB _3^3U_j|_{\Sn}}{\al_j^{p_3^*}}$
				and $\al_j\to1^{-}$, we obtain
				$c_3(\al_j)\to0$.
				Since the sequence
$\{\al_j\}$ was arbitrary, the conclusion follows.
			\end{proof}

			\begin{lem}\label{lem:c3decreasing}
				The map $\al \mapsto c_3(\al )$ is strictly decreasing on $(0,1)$.
			\end{lem}

			\begin{proof}
				Write $U_{\al}=\al  V_{\al}$.
				It follows from \eqref{Ulimt_equation} and \eqref{Ualpha_expression} that  $V_\al $ satisfies
\[
V_{\al} = H+\bar{\cG}^{(0,1)} (\kap \al^{p^*-1}V_{\al}^{p^*}),
\]
				together with the boundary conditions
\[
V_{\al}=1 , ~ \sB _1^3V_{\al}=0 \quad\text{ on }\,\Sn,
\]
				where  $\bar{\cG}^{(0,1)}$ is the  Boggio Green kernel and $H$ is given by \eqref{Hr}.
				Let $0<\al_1<\al_2<1$. Set
				$V_i=V_{\al_i}$ and $\mu_i=\al_i^{p^*-1}$ for $i=1,2$.
				Then $\mu_1<\mu_2$.  For each $i=1,2$, $V_i$ is obtained from the monotone iteration scheme
				corresponding to the fixed point equation
\[
V=H+\bar{\cG}^{(0,1)}(\kap \mu_i V^{p^*}),
\]
				which follows by normalizing the fixed point representation \eqref{Ulimit-expression}.
				Starting the iteration from $H$,
\[
V_{i,0}=H,\quad V_{i,k+1}=H+\bar{\cG}^{(0,1)}(\kap \mu_iV_{i,k}^{p^*}),\quad k=0,1,2,\ldots,
\]
				we obtain by induction that
\[
V_{1,k}\le V_{2,k}\quad\text{ in }\,\B^{n+1}, \quad\forall\, k=0,1,2,\ldots.
\]
				Passing to the limit yields $V_1\le V_2$ in $\B^{n+1}$. Moreover,
\[
V_2-V_1 = \bar{\cG}^{(0,1)}(\kap (\mu_2V_2^{p^*}-\mu_1V_1^{p^*})),
\]
				and
\[
\mu_2V_2^{p^*}-\mu_1V_1^{p^*} \ge (\mu_2-\mu_1)V_1^{p^*}>0.
\]
				By \eqref{positiveGreen}, it follows that
				$	V_2>V_1$ in $\B^{n+1}$.

				Set $\Phi:=V_2-V_1$.
				Then $\Phi$ is radial and positive in $\B^{n+1}$. Moreover,
\[
\Delta^2\Phi = \kap \al_2^{p^*-1}V_2^{p^*} - \kap \al_1^{p^*-1}V_1^{p^*} > 0 \quad\text{ in }\,\B^{n+1},
\]
				and \[
				\Phi=0,\quad  \sB _1^3\Phi=0
				\quad\text{ on }\,\Sn.
				\]
				Applying Lemma \ref{lem:Comparison} to $\Phi$ gives
				$\sB _3^3\Phi<0$ on $\Sn$.
				Consequently,
				$\sB _3^3V_2<\sB _3^3V_1$.
				Thus the map
				$\al \mapsto \sB _3^3V_{\al}$
				is strictly decreasing on $(0,1)$.

				Finally, since
\[
c_3(\al ) = \frac{\sB_3^3U_\al |_{\S^n}}{\al ^{p_3^*}} = \al ^{1-p_3^*}\sB_3^3V_\al |_{\S^n},
\]
				and
				$1-p_3^*<0$,
				the factor $\al ^{1-p_3^*}$ is strictly decreasing on $(0,1)$. Moreover,
				$c_3(\al )>0$ for every $\al \in(0,1)$, and hence
				$\sB_3^3V_\al >0$ on $\S^n$. Therefore the product
				$\al ^{1-p_3^*}\sB_3^3V_\al $
				is strictly decreasing in $\al $. Hence $c_3(\al )$ is strictly
				decreasing on $(0,1)$.
			\end{proof}

			\begin{proof}[Proof of Proposition \ref{prop:c3}]
				The proof  follows from \eqref{c3alpha}, \eqref{c3blowup} and Lemmas \ref{lem:c3continous}--\ref{lem:c3decreasing}.
			\end{proof}
			\begin{proof}[Proof of Theorem \ref{thm:Main}]
				The proof follows from Theorem \ref{thm:main0}, Proposition
\ref{prop:c3}, Remark \ref{rem:Sec5}, and Lemmas
\ref{lem:subsupsolution}, \ref{lem:no-positive-c3-C-geq-one},
\ref{lem:same-sign-uniqueness}, and
\ref{lem:no-decay-positive-second-third}.
			\end{proof}

			\appendix
			\section{Technical lemmas}\label{app:A}
For completeness, we collect in this appendix several standard analytic
tools used in the moving-spheres argument. 
We begin with the classical Hardy--Littlewood--Sobolev inequality in \cite{L1983}.
				\begin{lem}\label{HLS}
					Let $n\geq1$, $t\in (0,n)$, and $1<p<q<+\infty$ be such that $t+\frac{n}{q}=\frac{n}{p}$. Then we have
					\[
					\Big\|\int_{\Rn }\frac{f(y)}{|x-y|^{n-t}}\, \d y  \Big\|_{L^{q}(\Rn )}\leq C(n,t,p,q)\|f\|_{L^{p}(\Rn )},\quad \forall\, f\in L^{p}(\Rn ),
					\]
					where $C(n,t,p,q)>0$ is a constant  depending only on $n,t,p,q$.
				\end{lem}

			We  also recall a family of Hardy-Littlewood-Sobolev type inequalities due to Gluck \cite{G2020}, which play a crucial  role in the following singular integral estimates.  Let $a,b\in\R$ satisfy
			\be\label{Intro-a-b-1}
			b\geq 0, \quad 0<a +b<n+1-b
			\ee
			and
			\be\label{Intro-a-b-2}
			\frac{n+1-a -2b}{2(n+1)}+\frac{n+1-a }{2n}<1.
			\ee
		For $X=(x,t)\in\R^{n+1}_+$ and $Y=(y,s)\in\R^{n+1}_+$, define the
integral operators
\[
\mathrm{E}_{a ,b}f(X)=\int_{\Rn}\frac{t^{b}f(y)}{(|x-y|^2+t^2)^{\frac{n+1-a }{2}}}\, \d y
\]
			and
\[
\mathrm{R}_{a ,b}g(x)=\int_{\R^{n+1}_{+}}\frac{s^{b}g(Y)}{(|x-y|^2+s^2)^{\frac{n+1-a}{2}}}\, \d Y.
\]
			\begin{lem}[Theorem 1.1 and Corollary 1.2 in \cite{G2020}]\label{lem:Gluck}
				Assume the conditions \eqref{Intro-a-b-1} and \eqref{Intro-a-b-2} as above. Then there are two sharp constants $C_1(n,a ,b)$ and $C_2(n,a ,b)$ such that
\[
\|\mathrm{E}_{a ,b}f\|_{L^{\frac{2(n+1)}{n+1-a -2b}}(\R^{n+1}_{+})}\leq C_1(n,a ,b)\|f\|_{L^{\frac{2n}{n+a -1}}(\Rn)},\quad \forall\, f\in L^{\frac{2n}{n+a -1}}(\Rn),
\]
				and
\[
\|\mathrm{R}_{a ,b}g\|_{L^{\frac{2n}{n+1-a }}(\Rn)}\leq C_2(n,a ,b)\|g\|_{L^{\frac{2(n+1)}{n+1+a +2b}}(\R^{n+1}_{+})},\quad \forall\, g\in L^{\frac{2(n+1)}{n+1+a +2b}}(\R^{n+1}_{+}).
\]
			\end{lem}

		We conclude by recalling two inversion lemmas of Li \cite{L2004}
that are used to identify the limiting profile in the moving-spheres
argument.
			\begin{lem}[Lemma 5.7 in  \cite{L2004}]\label{lem:Technique1}
				For $n\geq 1$ and $\nu\in \R$, let $f$ be a function defined on $\Rn$ with values in $[-\infty,+\infty]$ satisfying
\[
\Big(\frac{\lam }{|y-x|} \Big)^{\nu}f \Big(x+\frac{\lam ^2(y-x)}{|y-x|^2} \Big)\leq f(y), \quad\forall\, \lam >0,~ |x-y|>\lam >0.
\]
				Then $f\equiv \text{const}$ or $\pm\infty$.
			\end{lem}

			\begin{lem}[Lemma 5.8 in \cite{L2004}]\label{lem:Technique2}
				Let $n\geq 1$, $\nu\in \R$ and  $f\in C(\Rn)$. Assume that for any $x\in \Rn$, there exists $\lam (x)>0$ such that
\[
\Big(\frac{\lam (x)}{|y-x|}\Big)^{\nu}f\Big(x+\frac{\lam (x)^2(y-x)}{|y-x|^2}\Big)=f(y), \quad \forall\, y\in\Rn\backslash\{x\}.
\]
				Then there exist $a\geq 0$, $\var>0$ and $x_0\in \Rn$ such that
\[
f(x)=\pm a\Big(\frac{1}{\var^2+|x-x_0|^2}\Big)^{\frac{\nu}{2}}.
\]
			\end{lem}

	\medskip

	\noindent\textbf{Acknowledgements.}~		
			We thank Professor Xuezhang Chen for reading an earlier version of this manuscript and for providing valuable comments on the necessity of studying the fourth-order ODE \eqref{BVP_U-1}.  L. Sun is partially supported by CAS Project for Young Scientists in Basic Research Grant (No.\,YSBR-031), Strategic Priority Research Program of the Chinese Academy of Sciences (No.\,XDB0510201), National Key R\&D Program of China (No.\,2022YFA1005601), NSFC China (No.\,12471115). H. Wang is supported by China Postdoctoral Science Foundation 2026M793418.
			S. Zhang is supported by the Postdoctoral Fellowship
			Program and China Postdoctoral Science Foundation under Grant Numbers BX20250062 and 2026M793381.

\medskip

		\noindent\textbf{Data availability statement.}~	 No external datasets were used in this study. The numerical computations
reported in this article are provided solely to illustrate the theoretical
results, and no separate research dataset is available.

\medskip

		\noindent\textbf{Declarations.}~ \textbf{Conflict of interest.} ~ On behalf of all authors, the corresponding author states that there is no conflict of interest.

\medskip

\noindent\textbf{AI assistance statement.} ~The authors   used OpenAI models only for language editing, expository refinement, and preliminary consistency checks. All mathematical statements and proofs were independently verified by the authors, who take full responsibility for the content of the manuscript.



			\bibliographystyle{plain}  \bibliography{QT_ref}
		\end{document}